\documentclass[leqno]{article}
\usepackage{authblk}
\usepackage{amsmath,amsthm,amssymb,latexsym,graphicx, mathtools, enumerate,bm,cite,todonotes}

\usepackage[nottoc]{tocbibind}

\usepackage[margin=1in]{geometry}

\usepackage{hyperref}
\hypersetup{
    colorlinks=true,
    allcolors=blue,
}

\newcommand{\EE}{{\mathbb E}}
\newcommand{\NN}{{\mathbb N}}

\newcommand{\RR}{{\mathbb R}}
\newcommand{\ZZ}{{\mathbb Z}}

\newcommand{\cadlag}{}
\def\cadlag/{c\`adl\`ag}

\newcommand{\del}{\partial}

\renewcommand{\div}{\on{div}}
\newcommand{\eps}{\varepsilon}

\newcommand{\Id}{\on{Id}}
\newcommand{\ind}{\mathbf{1}}

\newcommand{\loc}{\mathrm{loc}}
\newcommand{\norm}[1]{ \left\| #1 \right\| }
\newcommand{\nor}[2]{\left\|#1\right\|_{#2}}
\newcommand{\oline}[1]{\overline{#1}}
\newcommand{\oo}{\infty}
\newcommand{\sgn}{\on{sgn}}
\newcommand{\supp}{\on{supp }}
\newcommand{\tr}{\on{tr}}
\newcommand{\uline}[1]{\underline{#1}}
\newcommand{\w}{\mathrm{w}}
\newcommand{\ws}{{\mathrm{w}\text{-}\star}}

\DeclareMathOperator*{\esssup}{ess\,sup}

\newcommand{\mcl}{\mathcal}
\newcommand{\mbb}{\mathbb}
\newcommand{\mbf}{\mathbf}
\newcommand{\on}{\operatorname}

\newtheorem{lemma}{Lemma}
\newtheorem{proposition}{Proposition}
\newtheorem{theorem}{Theorem}
\newtheorem{corollary}{Corollary}
\newtheorem{remark}{Remark}

\theoremstyle{definition}
\newtheorem{definition}{Definition}

\newtheorem{example}{Example}

\numberwithin{equation}{section}
\numberwithin{lemma}{section}
\numberwithin{proposition}{section}
\numberwithin{theorem}{section}
\numberwithin{corollary}{section}
\numberwithin{definition}{section}
\numberwithin{remark}{section}
\numberwithin{example}{section}

\begin{document}
\title{Linear and nonlinear transport equations with coordinate-wise increasing velocity fields}
\author{Pierre-Louis Lions}
\affil{Universit\'e Paris-Dauphine and Coll\`ege de France\\
\href{lions@ceremade.dauphine.fr}{\nolinkurl{lions@ceremade.dauphine.fr}}}
\author{Benjamin Seeger
}
\affil{University of Texas at Austin\\ 
\href{mailto:seeger@math.utexas.edu}{\nolinkurl{seeger@math.utexas.edu}}\\
{\normalsize Partially supported by the National Science Foundation under award number DMS-1840314}
}

\maketitle

\begin{abstract}
	We consider linear and nonlinear transport equations with irregular velocity fields, motivated by models coming from mean field games. The velocity fields are assumed to increase in each coordinate, and the divergence therefore fails to be absolutely continuous with respect to the Lebesgue measure in general. For such velocity fields, the well-posedness of first- and second-order linear transport equations in Lebesgue spaces is established, as well as the existence and uniqueness of regular ODE and SDE Lagrangian flows. These results are then applied to the study of certain nonconservative, nonlinear systems of transport type, which are used to model mean field games in a finite state space. A notion of weak solution is identified for which a unique minimal and maximal solution exist, which do not coincide in general. A selection-by-noise result is established for a relevant example to demonstrate that different types of noise can select any of the admissible solutions in the vanishing noise limit.
\end{abstract}

\tableofcontents

\section{Introduction}

This paper has two main purposes. First, we develop a well-posedness theory for first- and second-order linear transport equations for a new class of irregular velocity fields, as well as the corresponding ODE and SDE regular Lagrangian flows. We then apply the results to the study of certain nonlinear transport systems, motivated in particular by applications to mean field games (MFG) with a finite state space.

For a fixed, finite time horizon $T >0$, we study both the terminal value problem (TVP) for the nonconservative equation
\begin{equation}\label{intro:TE}
	-\del_t u - b(t,x) \cdot \nabla u = 0 \quad \text{in } (0,T) \times \RR^d, \quad u(T,\cdot) = u_0
\end{equation}
as well as the dual, initial value problem (IVP) for the conservative equation
\begin{equation}\label{intro:CE}
	\del_t f + \div(b(t,x) f) = 0 \quad \text{in } (0,T) \times \RR^d, \quad f(0,\cdot) = f_0,
\end{equation}
under the assumption that $b$ is coordinate-by-coordinate (semi)-increasing:
\begin{equation}\label{intro:inc}
	\del_{x_j} b^i(t,\cdot) \ge -C(t) \delta_{ij} \quad \text{for all $t \in [0,T]$, $i,j = 1,2,\ldots, d$, and some } C \in L^1_+([0,T]).
\end{equation}
The divergence of such a vector field is bounded from below, which means that, formally, the flow
\begin{equation}\label{intro:ODE}
	\del_t \phi_{t,s}(x) = b(t, \phi_{t,s}(x)), \quad t \in [s,T], \quad \phi_{s,s}(x) = x
\end{equation}
will not concentrate at null sets. This indicates that the two problems \eqref{intro:TE} and \eqref{intro:CE} are amenable to a solution theory in Lebesgue spaces.

On the other hand, the measure $\div b$ is not in general absolutely continuous with respect to Lebesgue measure, and this leads to the formation of vacuum for $t > s$. It is therefore the case that the existing theory of renormalized solutions, initiated by DiPerna and the first author \cite{DL89} for Sobolev velocity fields and extended to the case $b \in BV_\loc$ and $\div b \in L^\oo$ by Ambrosio \cite{A04}, does not cover our present situation. In particular, the two problems \eqref{intro:TE} and \eqref{intro:CE} cannot be covered with a unified theory, due to the fact that \eqref{intro:TE} cannot be understood in the distributional sense if $\div b$ is not absolutely continuous with respect to Lebesgue measure and $u \in L^1_\loc$. Nevertheless, we exploit the dual relationship between the two problems, and provide a link to the forward, regular Lagrangian flow \eqref{intro:inc}. Analogous results are also proved for the degenerate, second order equations 
\begin{equation}\label{intro:2TE}
	-\del_t u - b(t,x) \cdot \nabla u  - \tr[ a(t,x) \nabla^2 u]= c(t,x) \quad \text{in } (0,T) \times \RR^d, \quad u(T,\cdot) = u_T
\end{equation}
and
\begin{equation}\label{intro:2CE}
	\del_t f + \div(b(t,x) f) - \nabla^2 \cdot(a(t,x)f) = 0 \quad \text{in } (0,T) \times \RR^d, \quad f(0,\cdot) = f_0,
\end{equation}
where $b$ satisfies \eqref{intro:inc},  $a(t,x) = \frac{1}{2} \sigma(t,x)\sigma(t,x)^T$ is a nonnegative, symmetric matrix, and $\sigma \in \RR^{d \times m}$; and the SDE flow
\begin{equation}\label{intro:SDE}
	d_t \Phi_{t,s}(x) = b(t,\Phi_{t,s}(x))dt + \sigma(t,\Phi_{t,s}(x))dW_t,
\end{equation}
where $W$ is an $m$-dimensional Brownian motion. 

We then turn to the study of nonlinear, nonconservative systems of transport type that take the form
\begin{equation}\label{intro:NTE}
	-\del_t u - f(t,x,u) \cdot \nabla u = g(t,x,u) \quad \text{in } (0,T) \times \RR^d, \quad u(T,\cdot) = u_T.
\end{equation}
Here, $u$, $f$, and $g$ are vector-valued, with $f \in \RR^d$ and $g,u \in \RR^m$ for some integers $m,d \ge 1$, $f$ and $g$ are local functions of $(t,x,u) \in [0,T] \times \RR^{d} \times \RR^m$, and the equation reads, for each coordinate $i = 1,2,\ldots, m$,
\[
	-\del_t u^i - \sum_{j=1}^d f^j(t,x,u) \del_{x_j} u^i = g^i(t,x,u).
\]
The primary motivation for the consideration of \eqref{intro:NTE} comes from the study of mean field games (MFG). These are models for large populations of interacting rational agents, which strategize in order to optimize an outcome, based on the collective behavior of the remaining population, while subject to environmental influences. The master equation for mean field games with a (finite) discrete state space takes the general form of the system \eqref{intro:NTE} with $d = m$, as described in \cite{L_lectures}; see also \cite{BLL_19,GMR_13}. Alternatively, systems of the form \eqref{intro:NTE} arise upon exploiting dimension reduction techniques for continuum-state MFG models in which the various data depend on the probability distribution of players through a finite number of observables, i.e.
\[
	\mu \mapsto u\left( t, \int \Phi d\mu \right) \quad \text{for probability measures }\mu \text{ and some given continuous $\RR^d$-valued function }\Phi.
\]
This connection is explored by the authors and Lasry in \cite{LLS_22}. We note also that the special case where $d = m$, $f(t,x,u) = -u$, and $g(t,x,u) = 0$ leads to the system
\begin{equation}\label{intro:mburgers}
	-\del_t u + u \cdot \nabla u = 0,
\end{equation}
which arises in certain models describing the flow of compressible gasses at low density with negligible pressure \cite{LT_02,GS_97}.

The nonlinear equation \eqref{intro:NTE} can formally be connected to a system of characteristic ODEs on $\RR^d \times \RR^m$. In order to draw the analogy to MFG PDE systems and the master equation in a continuum state space (see for instance \cite{CDLL_19}), it is convenient to represent the characteristics as the forward backward system
\begin{equation}\label{intro:chars}
	\begin{dcases}
		-\del_s U_{s,t}(x) = g(s,X_{s,t}(x),U_{s,t}(x)) & U_{T,t}(x) = u_T( X_{T,t}(x)), \\
		\del_s X_{s,t}(x) = f(s,X_{s,t}(x), U_{s,t}(x)) & X_{t,t}(x) = x.
	\end{dcases}
\end{equation}
If $f$, $g$, and $u_T$ are smooth and the interval $[t,T]$ is sufficiently small, then \eqref{intro:chars} can be uniquely solved, and the unique, smooth solution $u$ of \eqref{intro:NTE} is given by $u(t,x) := U_{t,t}(x)$. The argument fails for arbitrarily long time intervals, in view of the coupling between $X$ and the terminal condition for $U$.\footnote{When written in forward form, the breakdown manifests as a failure to invert the characteristics describing the state variable $x \in \RR^d$, which may cross in finite time.} 

The monotone regime is explored in \cite{L_lectures,BLL_19}, i.e., $d = m$ and $(-g,f): \RR^d \times \RR^d \to \RR^d \times \RR^d$ and $u_T: \RR^d \to \RR^d$ are smooth and monotone (one of which is strictly monotone). In that case, \eqref{intro:chars}, and therefore \eqref{intro:NTE}, can be uniquely solved on any time interval. This regime is exactly analogous to the monotonicity condition of Lasry and the first author for MFG systems with continuum state space \cite{LL06cr1,LL06cr2,LLJapan}, and, as in that setting, strong regularity and stability results can be established for \eqref{intro:NTE}. The monotone regime also allows for a well-posed notion of global weak solutions of \eqref{intro:NTE}, even when $f$, $g$, and $u_T$ fail to be smooth \cite{Ber_21}.

If there exist functions $H: \RR^d \times \RR^d \to \RR$ and $v_T: \RR^d \to \RR$ such that $(g,f) = (\nabla_x, \nabla_p) H$ and $u_T = \nabla v_T$ (the so-called ``potential regime''), then, formally, one expects $u(t,\cdot) = \nabla v(t,\cdot)$, where $v$ solves the Hamilton-Jacobi equation
\begin{equation}\label{eq:HJ}
	\del_t v + H(t,x,\nabla_x v) = 0 \quad \text{in }(0,T) \times \RR^d, \quad v(T,\cdot) = v_T,
\end{equation}
for which global, weak solutions can be understood with the theory of viscosity solutions \cite{CIL}. Weak solutions of \eqref{intro:NTE} can then be indirectly understood as the distributional derivative of $v$, an approach which is taken in \cite{CeDe_22}. In the special case where $d = m = 1$, \eqref{intro:NTE} can be studied with the theory of entropy solutions of scalar conservation laws \cite{Kr_70}.

In this paper, using the theory developed here for linear equations, we identify new regimes of assumptions on $f$, $g$, and $u_T$ for which a notion of weak solution can be identified for any dimensions $d,m \ge 1$. Under a certain ordering structure, the existence of a unique maximal and minimal solution is established, which do not coincide in general. This nonuniqueness is further explored from the viewpoint of stochastic selection, and we prove, for a specific but informative example, that \emph{any} of the family of solutions can be distinguished by a certain vanishing noise limit, indicating that the choice of a mean field game equilibrium is very sensitive to the manner in which low-level, systemic noise is introduced to the model.

\subsection{Summary of main results}\label{ss:results}

We list the main results of the paper here, in an informal setting. More precise statements and discussions can be found within the body of the paper.

The assumption that $b$ satisfy the (semi)-increasing condition \eqref{intro:inc} implies that $b \in BV_\loc$. We emphasize again, however, that the measure $\div b$ will in general have a singular part with respect to Lebesgue measure, and we therefore cannot appeal to the existing results on renormalized solutions to transport equations with irregular velocity fields. We do not give a full account of the vast literature for such problems, but refer the reader to the thorough surveys \cite{ACetraro,Am_survey_17,Am_Cr_survey_14,Am_Tr_survey_17} and the references therein.

Our approach to the first-order transport problem is to study the well-posedness of the regular Lagrangian flow for \eqref{intro:ODE} directly, rather than using PDE methods. The assumption \eqref{intro:inc} on $b$ allows for a comparison principle with respect to the partial order
\begin{equation}\label{intro:order}
	x,y \in \RR^d, \; x \le y \quad \Leftrightarrow \quad x_i \le y_i \text{ for all } i = 1,2,\ldots, d.
\end{equation}
A careful regularization procedure then leads to the existence of a minimal and maximal flow, which coincide a.e., and we have the following result (see Section \ref{sec:ODE} below for more precise statements):

\begin{theorem}
	Assume $b$ satisfies \eqref{intro:inc}. Then, for a.e. $x \in \RR^d$, there exists a unique absolutely continuous solution of \eqref{intro:ODE}, and there exists a constant $C > 0$ such that, for all $0 \le s \le t \le T$,
	\[
		\left| \phi^{-1}_{t,s}(A) \right| \le C|A| \quad \text{for all measurable } A \subset \RR^d.
	\]
	If $(b^\eps)_{\eps > 0}$ is a family of smooth approximations of $b$ and $(\phi^\eps)_{\eps > 0}$ are the corresponding flows, then, as $\eps \to 0$, $\phi^\eps \to \phi$ in $C_t L^p_{x,\loc}$ and $L^p_{x,\loc}W^{1,1}_t$ for all $p \in [1,\oo)$.
\end{theorem}

We also obtain analogous results for the SDE \eqref{intro:SDE}.Degenerate linear parabolic equations and SDEs with irregular data have been studied in a number of works that generalize the DiPerna-Lions theory and the Ambrosio superposition principle; see \cite{ChJ_SDE_18, Fig,LbL_08,LbL_19,Trevisan_16}. A common source of difficulty involves the dependence of $\sigma$ on the spatial variable, even if it is smooth. This is the case, for instance, when $b \in BV_\loc$ and $\div b \in L^\oo$ treated by Figalli \cite{Fig}, or when $b$ satisfied a one-sided Lipschitz condition from below, as considered by the authors in \cite{LSother}; in both settings, the results are constrained to $\sigma(t,x) = \sigma(t)$ constant in $\RR^d$. In our present setting, we can relax the spatial dependence, and we assume that $\sigma$ is Lipschitz and satisfies
\begin{equation}\label{intro:sigma}
	\sigma^{ik}(t,x) = \sigma^{ik}(t, x_i), \quad \text{for all } (t,x) \in [0,T] \times \RR^d, \quad i = 1,2,\ldots, d, \; k = 1,\ldots,m.
\end{equation}

Then, in Section \ref{sec:TE}, we turn to the study of the linear transport equations \eqref{intro:TE}-\eqref{intro:CE}, as well as the second-order equations \eqref{intro:2TE}-\eqref{intro:2CE}, which can be related to the ODE and SDE flows in Section \ref{sec:ODE}.

\begin{theorem}
	The flow \eqref{intro:ODE} (resp. \eqref{intro:SDE}) gives rise to continuous solution operators on $L^p_\loc$ (resp. $L^p$), $p \in (1,\oo)$ for the dual problems \eqref{intro:TE}-\eqref{intro:CE} (resp. \eqref{intro:2TE}-\eqref{intro:2CE}). The resulting solutions are stable under regularizations of $b$ or vanishing viscosity limits in the spaces $C_t L^p_{x}$, with the convergence being strong for the nonconservative equations \eqref{intro:TE}/\eqref{intro:2TE} and weak for the conservative equations \eqref{intro:CE}/\eqref{intro:2CE}.
\end{theorem}

The solution operator for the nonconservative equation \eqref{intro:TE} cannot be made sense of in the sense of distributions, because the measure $\div b$ can have a singular part. We nevertheless provide a PDE-based characterization for solutions that are \emph{increasing or decreasing} with respect to the partial order \eqref{intro:order}. In order to give meaning to the ill-defined product $b \cdot \nabla u$ in this context, we introduce mollifications of a one-sided nature that lead to commutator errors which are shown to possess a sign; this is to be compared with the renormalization theory initiated in \cite{DL89}, in which the commutator errors are shown to converge to zero with the convolution parameter. This leads to a notion of sub and supersolution for \eqref{intro:TE}, which are proved to satisfy a comparison principle. The solution operator for \eqref{intro:TE} can thus be alternatively characterized in terms of these regularizations, and, moreover, the regular Lagrangian flow \eqref{intro:ODE} can be recovered as the (vector)-valued solution of the terminal value problem \eqref{intro:TE} with $u_T(x) = x$. These two viewpoints on the transport equation and ODE flow are instrumental in our understanding of the nonlinear equations to follow.

The continuity equation \eqref{intro:CE} (or the Fokker-Planck equation \eqref{intro:2CE} in the second-order case) can then be related to the nonconservative equation through duality. Importantly, arbitrary distributional solutions of \eqref{intro:CE} are not unique in general. We prove that, if $f_0 \ge 0$, then there exist unique distributional solutions of \eqref{intro:CE} and \eqref{intro:2CE} which coincide with the duality solution. Moreover, this result is proved independently of the superposition principle; instead, we use the duality with the nonconservative equation, and the characterization of its solutions in terms of one-sided regularizations.

A consequence of the uniqueness of nonnegative distributional solutions of \eqref{intro:CE} is that, if $f$ and $|f|$ satisfy \eqref{intro:CE} in the sense of distributions, then $f$ is the ``good'' (duality) solution (see Corollary \ref{C:criterion} below). We do not know whether this property characterizes the duality solution, or, in other words, whether the duality solution satisfies the renormalization property in general. This should be compared with \cite{LSother}, where the authors resolve the same questions for half-Lipschitz velocity fields.

The paper concludes in Section \ref{sec:NTE} with the study of the nonlinear equation \eqref{intro:NTE} and the associated system \eqref{intro:chars}. We operate under the assumption that the discontinuous nonlinearities $f$ and $g$ satisfy, for some $C \in L^1_+([0,T])$,
\begin{equation}\label{intro:fg}
	\left\{
	\begin{split}
	&\text{for all } i,j = 1,2,\ldots, d \text{ and } k,\ell = 1,2,\ldots, m,\\
	&\del_{x_i} f^j(t,\cdot,\cdot) \ge - C(t) \delta_{ij},\quad \del_{u_k} g^\ell(t,\cdot,\cdot) \ge -C(t) \delta_{k\ell}, \quad
	\del_{x_i} g^\ell \le 0, \quad \text{and} \quad \del_{u_k} f^j \le 0.
	\end{split}
	\right.
\end{equation}
Observe that \eqref{intro:fg} is satisfied with $C \equiv 0$ by the particular example of the Burgers-like equation \eqref{intro:mburgers}.

We develop a theory for solutions of \eqref{intro:NTE} that are \emph{decreasing} with respect to the partial order \eqref{intro:order}. The first observation is that the decreasing property is propagated, formally, by the solution operator. On the other hand, shocks form in finite time, and so, even if $u_T$, $f$, and $g$ are smooth, $u(t,\cdot)$ will develop discontinuities for some $t < T$ in general, requiring a notion of weak solution. We next note that, under the above assumptions, if $u$ is decreasing, then the velocity field $b(t,x) := f(t,x,u(t,x))$ satisfies \eqref{intro:inc}. Solutions of the nonlinear equation \eqref{intro:NTE} can then be understood as fixed points for the linear problem \eqref{intro:TE}, and at the same time through the system of forward-backward characteristics \eqref{intro:chars}, using the theory for the regular Lagrangian flows in Section \ref{sec:ODE}.

\begin{theorem}
	Assume $f$ and $g$ satisfy \eqref{intro:fg} and $u_T$ is decreasing. Then there exist a maximal and minimal decreasing solution $u^+$ and $u^-$ of \eqref{intro:NTE} in the fixed point sense. If $u$ is any other such solution, then $u^- \le u \le u^+$. 
\end{theorem}

Continuous decreasing solutions of \eqref{intro:NTE} satisfy a comparison principle, which in particular implies that $u^+$ is below every continuous supersolution and $u^-$ is above every continuous subsolution. In general, $u^-$ and $u^+$ do not coincide, which must then be a consequence of the formation of discontinuities. This nonuniqueness of solutions is closely related to the question of multiplying distributions of limited regularity. Indeed, in the equation \eqref{intro:NTE}, the product $f(t,x,u) \cdot \nabla u$ cannot be defined in a stable (with respect to regularizations) way in view of the fact that, in general, $u \in BV_\loc$ and $\nabla u$ is a locally finite measure.

The well-posedness of both strong and weak solutions of MFG master equations in the continuum state space setting has been explored under various sets of monotonicity assumptions \cite{Ah_16, BLL_23, Ber_21_cts,CDLL_19,GaMe_22,GaMeMZh_22,CS_22,CD_18,MZ_anti,GM_mono,GM_consv}. The approach in our setting, which involves appealing to Tarski's fixed point theorem for increasing functions on lattices \cite{Tarski}, has also been taken in the continuum state space setting, where maximal and minimal solutions were found under related assumptions; see for instance \cite{DFFN_21,DFFN_22,Di_22,MZ_min}. The partial order used in \cite{MZ_min} comes from the notion of stochastic dominance for probability measures. We note that, for the equation \eqref{intro:NTE} posed on an infinite dimensional Hilbert space of $L^2$ random variables, the partial order \eqref{intro:order} is related to the analogous notion of stochastic dominance for random variables, and we aim to study the infinite dimensional version of \eqref{intro:NTE} in future work.

We explore the nonuniqueness issue in more detail for the Burgers equation \eqref{intro:mburgers} in one dimension, where the decreasing terminal value has a single discontinuity at $0$:
\begin{equation}\label{intro:1burgers}
	-\del_t u + u \del_x u = 0 \quad \text{in }(0,T) \times \RR, \quad u(T,x) = \ind\left\{ x \le 0\right\}.
\end{equation}
It turns out that \eqref{intro:1burgers} admits infinitely many fixed point solutions, consisting of a shock traveling with variable speed between $0$ and $1$. Of course, \eqref{intro:1burgers} can be reframed as a scalar conservation law, whose unique entropy solution is the shock-wave solution with speed $1/2$. We note that the notion of entropy solution does not extend to the nonconservative equations \eqref{intro:NTE} or \eqref{intro:mburgers}.

We characterize the family of fixed point solutions \eqref{intro:1burgers} as limits under distinct types of  regularizations of the equation \eqref{intro:1burgers}.

\begin{theorem}\label{T:intro:selection}
	For any $c \in W^{2,1}([0,T])$ satisfying $c(T) = 0$ and $-c' \in (0,1)$, there exists $\theta_c \in L^1([0,T])$ such that, if $u^\eps_T$ is smooth, $u^\eps_T \xrightarrow{\eps \to 0} \ind_{(-\oo,0)}$ in $L^1_\loc$, and $u^\eps$ is the unique classical solution of 
	\begin{equation}\label{intro:NL:epsburgers}
		-\del_t u^\eps + u^\eps\del_x u^\eps = \eps  \left( \del_x^2 u^\eps + \theta_c(t) |\del_x u^\eps|^2 \right), \quad u^\eps(T,\cdot) = u^\eps_T,
	\end{equation}
	then, for all $1 \le p < \oo$, as $\eps \to 0$, $u^\eps \to \ind\left\{ x \le c(t)\right\}$ strongly in $C([0,T], L^p_\loc)$.
\end{theorem}

We interpret Theorem \ref{T:intro:selection} as a selection-by-noise result for the nonunique problem \eqref{intro:1burgers}. Indeed, the result can be reformulated on the level of the system \eqref{intro:chars}, which, for \eqref{intro:NL:epsburgers}, becomes the forward-backward system of SDEs
\begin{equation}\label{intro:FBSDE}
	\begin{dcases}
	- d_s U^\eps_{s,t}(x) = Z^\eps_{s,t}(x)dW_s - \frac{1}{2} \theta_c(s) Z^\eps_{s,t}(x)^2 ds, & U^\eps_{T,t}(x) = u^\eps_T(X^\eps_{T,t}(x)), \\
	d_s X^\eps_{s,t}(x) = - U^\eps_{s,t}(x)ds + \sqrt{2\eps} dW_s, & X^\eps_{t,t}(x) = x.
	\end{dcases}
\end{equation}

Various selection methods have been proposed to study mean field games models that do not admit unique solutions, and we refer in particular to \cite{DF_20,CeDe_22} for problems involving stochastic selection. Our result is distinguished by the consideration of several different descriptions of small noise, each of which selects a different solution for the deterministic problem in the vanishing noise limit. 

\subsection{A note on velocity fields with a one-sided Lipschitz condition}

Let us remark that the regime in which $b$ satisfies \eqref{intro:inc} shares many similarities with the setting in which $b$ is half-Lipschitz from below\footnote{Indeed, \eqref{intro:inc} and \eqref{b:halfLip} are identical in one spatial dimension.}, that is,
\begin{equation}\label{b:halfLip}
	(b(t,x) - b(t,y)) \cdot (x-y) \ge -C(t)|x-y|^2 \quad \text{for } (t,x,y) \in [0,T] \times \RR^{2d}  \text{ and some } C \in L^1_+([0,T]).
\end{equation}
A key commonality in both settings is that $\div b$ is bounded from below, but not necessarily absolutely continuous with respect to Lebesgue measure. Transport equations and flows for velocity fields satisfying \eqref{b:halfLip} have been studied from a variety of different viewpoints \cite{Conway_67,BJ,BJM,CSmfg,Pe_Po_1d_99,Pe_Po_01,PR_97}, and in \cite{LSother}, the authors obtain very similar results to those described above regarding the existence, uniqueness, and stability of the regular Lagrangian flow forward in time, as well as proving well-posedness and studying properties and characterizations of solutions to the problems \eqref{intro:TE} and \eqref{intro:CE} in Lebesgue spaces.

A key difference between the two regimes is the behavior of the flow for \eqref{intro:ODE} in the compressive direction, that is, backward in time. For velocity fields satisfying the half-Lipschitz condition \eqref{b:halfLip}, the backward ODE is uniquely solvable for all $x \in \RR^d$. Moreover, the resulting backward flow is Lipschitz continuous, and it can be identified as the left-inverse to the forward, regular Lagrangian flow; see \cite{LSother} for more precise statements, as well as new characterizations of the time-reversed versions of \eqref{intro:TE} and \eqref{intro:CE}.

On the other hand, when $b$ satisfies \eqref{intro:inc}, the backward problem \eqref{intro:ODE} is not in general unique for every $x \in \RR^d$, nor is it true that a globally Lipschitz flow can always be found. We note that, even for examples where the backward flow has a unique solution for Lebesgue-a.e. $x \in \RR^d$, neither the stability nor the solvability of the time-reversed versions of \eqref{intro:TE}-\eqref{intro:CE} in Lebesgue spaces can be expected to hold, because, in general, any backward flow solution to \eqref{intro:ODE} will concentrate on sets of Lebesgue-measure zero. For a detailed discussion and examples, see subsections \ref{ss:reverse?} and \ref{ss:timereversed?} below.

\subsection{Notation}

Given a bounded function $\phi:\RR^d \to \RR$, $\phi_*$ and $\phi^*$ denote the lower and uppersemicontinuous envelopes, and, if $\phi$ is $\RR^m$, the same notation is used coordinate-by-coordinate.

We often denote arbitrary functions spaces on $\RR^d$ as $X(\RR^d) = X$ when there is no ambiguity over the domain. For $p \in [1,\oo]$, $L^p_\w$ and $L^p_\ws$ denote the space of $p$-integrable functions with respectively the weak and weak-$\star$ topologies. $L^p_\loc$ denotes the space of of locally $p$-integrable functions with the topology of local $L^p$-convergence, and $L^p_{\loc,\w}$ and $L^p_{\loc,\ws}$ are understood accordingly.

The notation $\mbf 1$ denotes the vector $(1,1,\ldots, 1)$ in Euclidean space, the dimension being clear from context. Given two sets $A$ and $B$, $A \triangle B := (A\backslash B) \cup (B \backslash A)$.

\section{Preliminary results}

This section contains a collection of results regarding vector-valued notions of increasing/decreasing, as well as a vector-valued maximum principle.

\subsection{Properties of increasing functions} 

We first introduce a partial order on $\RR^d$ that is used throughout the paper.

\begin{definition}\label{D:order}
For $x,y \in \RR^d$, we will write
\begin{equation}\label{order}
	x \le y \quad \text{if} \quad x_i \le y_i \text{ for all } i = 1,2,\ldots, d.
\end{equation}
Given $a,b \in \RR^d$, $a \le b$, we denote by $[a,b]$ the cube $\prod_{i=1}^d [a_i, b_i]$. A function $\phi: \RR^d \to \RR^m$, $d,m \in \NN$, is said to be increasing if $\phi(x) \le \phi(y)$ whenever $x \le y$. Equivalently, $\phi$ is increasing if, for each $i = 1,2,\ldots,m$, $\phi^i$ is increasing in the $x_j$-coordinate for each $j = 1,2,\ldots, d$.
\end{definition}

\begin{lemma}\label{L:increasing}
	Assume that $\phi: \RR^d \to \RR^m$ is increasing. Then $\phi \in BV_\loc(\RR^d)$, and, for each $i = 1,2\ldots, m$,
	\[
		\liminf_{y \to x} \phi^i(y) = \liminf_{y \to x, \, y \gneq x} \phi^i(y)
	\quad \text{and} \quad
		\limsup_{y \to x} \phi^i(y) = \limsup_{y \to x, \, y \lneq x} \phi^i(y).
	\]
\end{lemma}

\begin{proof}
	For all $j = 1,2,\ldots, d$ and $i = 1,2,\ldots, m$, the distribution $\del_j \phi^i$ is nonnegative, and is therefore a locally finite measure. 
	
	For any sequence $y_n \to x$, if $y_n \nleq x$, then $y_n$ may be replaced with a value $y_n'$ such that $y_n' \le x$, $y_n' \le y_n$, and $\phi^i(y_n') \le \phi^i(x)$. A similar argument holds if $y_n \ngeq x$ for all $n$, and the claim follows.
\end{proof}

Lemma \ref{L:increasing} implies that each component of an increasing function $\phi: \RR^d \to \RR^m$ has limits from both the ``left'' and ``right.'' We will call an increasing function $\phi$ \cadlag/ if, for $i = 1,2,\ldots, m$,
\[
	\phi^i(x) = \limsup_{y \to x} \phi^i(y) = \limsup_{y \to x, \, y \gneq x} \phi^i(y).
\]

\begin{remark}\label{R:incex}
Given a nonnegative measure $\mu$, the repartition function
\[
	\phi(x) := \int_{(-\oo,x_1)} \int_{(-\oo,x_2)} \cdots \int_{(-\oo,x_d)} d\mu
\]
is an example of a \cadlag/ increasing function, but such functions do not cover the full range of increasing functions if $d \ge 2$; indeed, they are distinguished by the fact that mixed derivatives $\prod_{j=1}^k \del_{x_{\ell_j}} \phi$ for any distinct set $(\ell_j)_{j=1}^k \subset \{1,2,\ldots, d\}$ are still measures. 

Consider a smooth surface $\Gamma \subset \RR^d$ that partitions $\RR^d$ into two open sets, that is, $\RR^d = U_- \cup \Gamma \cup U_+$, $\Gamma = \del U_+ = \del U_-$, $U_+ \cap U_- = \emptyset$. Let $n$ be the normal vector to $\Gamma$ that, at all points of $\Gamma$, points inward to $U_+$. If $n \ge 0$ everywhere on $\Gamma$, in the sense of \eqref{order}, then $\phi = \ind_{U_+ \cup \Gamma}$ is a \cadlag/ increasing function that is not a repartition function.
\end{remark}

\subsection{$ABV$ functions} In one dimension, functions of bounded variation can be written as a difference of non-decreasing functions. With respect to the partial order \eqref{order}, the generalization of this notion is a strict subspace of $BV$. 

\begin{definition}\label{D:ABV}
Given $-\oo < a_i \le b_i < \oo$ for $i = 1,2,\ldots d$ and a function $\phi: Q := \prod_{i=1}^d [a_i,b_i] \to \RR$, we say $\phi \in ABV(Q)$ if
\[
	\norm{\phi}_{ABV(Q)} := \sup_{\gamma} \norm{\phi \circ \gamma}_{BV([0,1])} < \oo,
\]
where the supremum is taken over all curves $\gamma: [0,1]^d \to Q$ such that $\gamma_i: [0,1] \to [a_i,b_i]$ is increasing for all $i = 1,2,\ldots, d$. We will say $\phi \in ABV = ABV(\RR^d)$ if $\norm{\phi}_{ABV(Q)} < \oo$ for all boxes $Q \subset \RR^d$. 
\end{definition}

For example, $C^{0,1} \subset ABV$. It is straightforward to see that $ABV = BV$ when $d = 1$.

\begin{remark}\label{R:ABV}
	Several generalizations of the notion of finite variation to multiple dimensions, besides the space $BV$, exist in the literature, and the one in Definition \ref{D:ABV} is due to Arzel\`a \cite{ABV}. It is a strictly smaller subspace than $BV$ (as we show below). This notion of variation, along with several others, seems not to have had the same ubiquity in the theory of PDEs as the usual notion of $BV$, but is particularly relevant in this paper. More details about $ABV$ functions, and many other notions of variation in multiple dimension, can be found in \cite{Breneis}.
\end{remark}

\begin{lemma}\label{L:ABV}
	Let $\phi: \RR^d \to \RR$. Then $\phi \in ABV$ if and only if $\phi = \phi_1 - \phi_2$ for two increasing functions $\phi_1,\phi_2:\RR^d \to \RR$.
\end{lemma}

\begin{proof}
	Let $a, b \in \RR^d$, $a \le b$. If $\phi: \RR^d \to \RR$ and $\gamma: [0,1] \to [a,b]$ are increasing, then $\phi \circ \gamma$ is increasing, and thus $\norm{ \phi \circ \gamma}_{BV([0,1])} = \phi(b) - \phi(a)$. It follows that $\phi \in ABV$, and, by linearity, differences of increasing functions belong to $ABV$. 

	Now assume $\phi \in ABV([a,b])$, and, for $x \in [a,b]$, set
	\[
		\phi_1(x) = \frac{1}{2} \left( \norm{\phi_1}_{ABV([a,x])} + \phi(x)\right) \quad \text{and} \quad 
		\phi_2(x) = \frac{1}{2} \left( \norm{\phi_1}_{ABV([a,x])} - \phi(x) \right).
	\]
	Then $\phi = \phi_1 - \phi_2$, and $\phi_1$ and $\phi_2$ are increasing.
\end{proof}

\begin{lemma}\label{L:ABVcts}
	If $\phi \in ABV$, then $\phi$ is almost everywhere continuous and differentiable.
\end{lemma}

\begin{proof}
	It suffices to prove the claim about differentiability, and, by Lemma \ref{L:ABV}, we may assume without loss of generality that $\phi$ is increasing. 
	
	We argue using the characterization by Stepanoff \cite{St1923} of a.e.-differentiability, that is, we prove that, for almost every $x \in \RR^d$,
	\[
		\limsup_{y \to x} \frac{ |\phi(y) - \phi(x)|}{|y-x|} < \oo,
	\]
	or, equivalently,
	\[
		\limsup_{y \to x} \frac{ \phi(y) - \phi(x)}{|y-x|} < \oo \quad \text{and} \quad
		\liminf_{y \to x} \frac{ \phi(y) - \phi(x)}{|y-x|} > -\oo.
	\]
	Denote by $Q_r := [0,r]^d$, and note that, because $\phi$ is increasing,
	\[
		\limsup_{y \to x} \frac{ \phi(y) - \phi(x)}{|y-x|} = \inf_{r > 0} \sup_{y \in x + Q_r} \frac{ \phi(y) - \phi(x)}{|y-x|} = \inf_{r > 0} \frac{ \phi(x + r \mbf 1) - \phi(x)}{\sqrt{d} r}.
	\]
	The function $[0,\oo) \ni r \to \phi(x+r \mbf 1)$ is increasing, and, thus, is differentiable almost everywhere in $[0,\oo)$. The finiteness of the above expression for almost every $x \in \RR^d$ is then a consequence of Fubini's theorem. The argument for the $\liminf$ is identical.
\end{proof}

\begin{remark}\label{R:cadlag}
	In view of Lemmas \ref{L:increasing} and \ref{L:ABVcts}, an increasing function is equal almost everywhere to a \cadlag/ function, and we assume for the rest of the paper that increasing functions are \cadlag/.
\end{remark}

\subsection{One-sided regularizations} \label{ss:regs}

It will convenient at several times in the paper to specify regularizations of discontinuous functions that enjoy certain ordering properties. We specify two such regularization procedures here, each of which has merit in different situations.

We first discuss the $\inf$- and $\sup$-convolutions given, for some measurable function $\phi$, by
\begin{equation}\label{infsup}
	\phi^\eps(x) = \sup_{y \in \RR^d} \left\{ \phi(x-y) - \frac{|y|}{\eps}  \right\}
 \quad \text{and} \quad
	\phi_\eps(x) = \inf_{y \in \RR^d} \left\{ \phi(x-y) + \frac{|y|}{\eps} \right\}.
\end{equation}
Then the following are either well-known properties or easy to check.

\begin{lemma}\label{L:infsup}
	Assume that $\phi$ is measurable and, for some $M > 0$ and all $x \in \RR^d$, $|\phi(x)| \le M(1 + |x|)$, and let $\phi_\eps$ and $\phi^\eps$ be defined by \eqref{infsup}, which are finite as long as $\eps < M^{-1}$. Then
	\begin{enumerate}[(a)]
	\item\label{smoothorder} For all $\eps \in (0, M^{-1})$, $\phi_\eps$ and $\phi^\eps$ are Lipschitz with constant $\eps^{-1}$, and $\phi_\eps \le \phi \le \phi^\eps$.
	\item\label{updownlimits} For all $x \in \RR^d$, as $\eps \to 0$, $\phi^\eps(x) \searrow \phi^*(x)$ and $\phi_\eps(x) \nearrow \phi_*(x)$
	\item\label{shiftover} If, moreover, $\phi$ is increasing, then so are $\phi_\eps$ and $\phi^\eps$.
	\end{enumerate}
\end{lemma}

\begin{remark}
	If $\phi$ is increasing, then the $\sup$ or $\inf$ can be restricted to ``one side'' of $x$. More precisely, given $y \in \RR^d$, define
	\[
		\hat y = - (|y_1|, |y_2|, \ldots, |y_d|),
	\]
	Then $\hat y \le y$ and $|\hat y| = |y|$, and so, because $\phi$ is increasing,
	\[
		\phi(x-y) - \frac{|y|}{\eps} \le \phi(x - \hat y) - \frac{|\hat y|}{\eps}.
	\]
	It follows that
	\[
		\phi^\eps(x) = \sup_{y \ge 0 }\left\{ \phi(x +y) - \frac{|y|}{\eps}  \right\},
	\]
	and similarly
	\[
		\phi_\eps(x) = \inf_{y \le 0} \left\{ \phi(x+y) + \frac{|y|}{\eps} \right\}.
	\]
\end{remark}

The second example involves convolving with mollifying functions that are weighted to one side. Let $\rho$ be a smooth, positive function with support contained in $[-1,1]^d$ and $\int \rho = 1$, and define
\begin{equation}\label{onesidedmollifier}
	\rho_\eps(z) = \frac{1}{\eps^d} \rho \left( \frac{z - \eps \mbf 1}{\eps} \right) \quad \text{and} \quad
	\rho^\eps(z) = \frac{1}{\eps^d} \rho \left( \frac{z + \eps \mbf 1}{\eps} \right).
\end{equation}

\begin{lemma}\label{L:approxid}
	Assume $\phi: \RR^d \to \RR$ is locally bounded, and set $\phi^\eps := \rho^\eps * \phi$ and $\phi_\eps := \rho_\eps * \phi$. Then, as $\eps \to 0$, $\phi^\eps$ and $\phi_\eps$ converge to $\phi$ in $L^p_\loc$ for any $p \in [1,\oo)$. Moreover, if $\phi$ is increasing, then
	\begin{enumerate}[(a)]
	\item\label{moll:order} $\phi^\eps$ and $\phi_\eps$ are increasing, and $\phi_\eps \le \phi \le \phi^\eps$.
	\item\label{moll:conf} As $\eps \to 0$, $\phi^\eps \searrow \phi^*$ and $\phi_\eps \nearrow \phi_*$.
	\item\label{moll:shift} For all $\eps > 0$ and $x \in \RR^d$, $\phi^\eps(x - 2\eps \mbf 1) = \phi_\eps(x)$.
	\end{enumerate}
\end{lemma}

\begin{proof}
	The convergence in $L^p_\loc$ is standard. For the rest of the proof, we assume $\phi$ is increasing. We see immediately that $\phi^\eps$ and $\phi_\eps$ are increasing. Observe that
	\[
		\supp \rho_\eps \subset (0,2\eps)^d \quad \text{and} \quad \supp \rho^\eps \subset (-2\eps,0)^d.
	\]
	As a consequence, because $\phi$ is increasing,
	\[
		\phi^\eps(x) = \int_{\RR^d} \phi(x-y) \rho^\eps(y)dy = \int_{(-2\eps,0)^d} \phi(x-y) \rho^\eps(y)dy \ge \phi^*(x),
	\]
	and similarly $\phi_\eps(x) \le \phi_*(x)$, so that part \eqref{moll:order} is established.
	
	Now assume that $0 < \eps < \delta$. Using again that $\phi$ is increasing, we find
	\begin{align*}
		\phi^\delta(x) &= \frac{1}{\delta^d} \int_{(-2\delta,0)^d} \phi(x - y) \rho\left( \frac{y}{\delta} + \mbf 1 \right) dy
		\ge \frac{1}{\delta^d} \int_{(-2\delta,0)^d} \phi\left(x - \frac{\eps}{\delta}y\right) \rho\left( \frac{y}{\delta} + \mbf 1 \right) dy\\
		&= \frac{1}{\eps^d} \int_{(-2\eps,0)^d} \phi(x - y) \rho \left( \frac{y}{\eps} + \mbf 1 \right) dy
		= \phi^\eps(x),
	\end{align*}
	and similarly $\phi_\delta \le \phi_\eps$. The convergence statements in part \eqref{moll:conf} then easily follow.
	
	Finally, part \eqref{moll:shift} is seen upon computing
	\[
		\phi^\eps(x - 2\eps \mbf 1)
		= \frac{1}{\eps^d}\int_{\RR^d} \phi(y) \rho\left( \frac{x  - y - \eps \mbf 1}{\eps} \right) dy
		= \phi_\eps(x).
	\]
\end{proof}

\begin{remark}
	An analogue of part \eqref{moll:shift} in Lemma \ref{L:approxid} can also be seen for the $\sup$- and $\inf$-convolutions $\phi^\eps$ and $\phi_\eps$ in \eqref{infsup}, namely, for all $R > 0$, there exists $C_R > 0$ depending on the linear growth of $\phi$ such that, for all $x \in B_R$,
	\[
		\phi^\eps(x - C_R \mbf 1) \le \phi_\eps(x).
	\]
	The advantage of the one-sided mollifications $\phi * \rho^\eps$ and $\phi* \rho_\eps$ is that the constant $C_R$ can be replaced by a uniform constant that does not depend on the growth of $\phi$, which will be convenient when we consider $\phi$ depending on an additional time parameter in an $L^1$ way. On the other hand, the $\sup$- and $\inf$-convolutions are very flexible one-sided regularizations even when $\phi$ is not itself increasing.
\end{remark}

\subsection{A maximum principle} 

The following multi-dimensional maximum principle is used at various points in the paper.
\begin{lemma}\label{L:particles}
	For some $M \in \NN$, assume that
	\begin{equation}\label{bounded}
		\left\{
		\begin{split}
		&(a^i)_{i = 1}^M \subset C([0,T] \times \RR^d,\mbb S^d), \quad
		(b^i)_{i=1}^M \subset C([0,T] \times \RR^d, \RR^d), \quad \text{and} \\
		&(c^i)_{i=1}^M, \;
		(d^i)_{i=1}^M, \;
		(\lambda^i_j)_{i,j=1}^M, \;
		(\mu_{k \ell}^i)_{i,k,\ell=1}^M,\\
		&\text{are uniformly bounded}
		\end{split}
		\right.
	\end{equation}
	and
	\begin{equation}\label{positive}
		\left\{
		\begin{split}
			&a^i(t,x) \ge 0, \quad d^i(t,x) \ge 0, \quad \lambda_j^i(t,x) \ge 0,  \quad \text{and} \quad \mu^i_{k \ell}(t,x) \le 0\\
			&\text{for all } (t,x) \in [0,T] \times \RR^d \text{ and } i,j,k,\ell \in \{1,2,\ldots,M\}.
		\end{split}
		\right.
	\end{equation}
	Let $(V^i)_{i=1}^M \subset C^1([0,T] \times \RR^d)$ be bounded on $[0,T]  \times \RR^d$ and satisfy the system
	\begin{equation}\label{E:Vsystem}
		\del_t V^i - \tr[ a^i D^2 V^i] + b^i \cdot D V^i + c^i V^i + d^i = \sum_{j\ne i} \lambda^i_j V^j + \sum_{k, \ell \ne i} \mu^i_{k\ell} V^k V^\ell   \quad \text{in } (0,T) \times \RR^d.
	\end{equation}
	Suppose that, for all $x \in \RR^d$ and $i = 1,2,\ldots,M$, $V^i(0,x) \le 0$. Then, for all $(t,x) \in (0,T] \times \RR^d$ and $i = 1,2,\ldots,M$, $V^i(t,x) \le 0$.	
\end{lemma}

\begin{proof}
	We prove the result first under the additional assumption that
	\[
		\inf_{(t,x) \in [0,T] \times \RR^d} d^i > 0, \quad \sup_{x \in \RR^d} V^i(0,x) < 0,
	\]
	and, for all $t \in [0,T]$ and $i \in \{1,2,\ldots,M\}$, $x \mapsto V^i(t,x)$ attains a global maximum on $\RR^d$. In that case, define
	\[
		t_0 := \inf \left\{ t \in (0,T] : \max_{x \in \RR^d, \; i \in \{1,2,\ldots,M\} } V^i(t,x) = 0 \right\},
	\]
	so that $t_0 > 0$. Let $x_0 \in \RR^d$ and $i \in \{1,2,\ldots,M\}$ be such that the maximum is achieved. Note then that $D V^i(t_0,x_0) = 0$, $D^2 V^i(t_0,x_0) \le 0$, $\del_t V^i(t_0,x_0) \ge 0$, and, for all $j \in \{1,2,\ldots,M\}$ and $x \in \RR^d$, $V^j(t_0,x) \le V^i(t_0,x_0) \le 0$. We thus obtain
	\[
		d^i(t_0,x_0) \le \sum_{j \ne i} \lambda^i_j V^j(t_0,x_0) + \sum_{k,\ell \ne i} \mu^i_{k \ell} V^k(t_0,x_0) V^\ell(t_0,x_0) \le 0,
	\]
	which is a contradiction in view of the assumption on $d^i$. In this case, we indeed conclude that $V^i(t,x) < 0$ for all $(t,x,i) \in [0,T] \times \RR^d \times \{1,2,\ldots,M\}$.
	
	We now turn to the general case. For $x \in \RR^d$, define $\nu(x) := \sqrt{1 + |x|^2}$, and note that $D \nu$ and $D^2 \nu$ are globally bounded on $\RR^d$. Then standard arguments yield that, for all $i \in \{1,2,\ldots,M\}$, $\beta > 0$, and $t \in [0,T]$, if $S_\beta$ is set of maximum points of
	\[
		x \mapsto V^i(t,x) - \beta \nu(x),
	\]
	then $\lim_{\beta \to 0} \beta \sup_{x \in S_\beta} \nu(x) = 0$. It follows that there exists a smooth bounded function $\nu_\beta: \RR^d \to \RR$ such that 
	\[
		\sup_{\beta > 0} \left( \nor{D\nu_\beta}{\oo} + \nor{D^2\nu_\beta}{\oo} \right)< \oo, \quad \lim_{\beta \to 0} \beta \nor{\nu_\beta}{\oo} = 0,
	\]
	and
	\[
		\max_{(t,x,i) \in [0,T] \times \RR^d \times \{1,2,\ldots,M\} } \left( V^i(t,x) - \beta \nu(x)\right)
		= \max_{(t,x,i) \in [0,T] \times \RR^d \times \{1,2,\ldots,M\} } \left( V^i(t,x) - \beta \nu_\beta(x)\right).
	\]
	Fix $\delta > 0$ and a constant $C > 0$ to be determined. For $(t,x) \in [0,T] \times \RR^d$, $i \in \{1,2,\ldots,M\}$, set
	\[
		\tilde V^i(t,x) := V^i(t,x) - \beta v_\beta(x) - \delta e^{Ct}.
	\]
	Then, for all $i \in \{1,2,\ldots,M\}$, $\tilde V^i(0,\cdot) \le -\delta$, and, for all $t \in [0,T]$, $\tilde V^i(t,\cdot)$ attains a global maximum over $\RR^d$. Moreover, in $(0,T) \times \RR^d$, $(\tilde V^i)_{i=1}^M$ solves the system
	\[
		\del_t \tilde V^i - \tr[ a^i D^2 \tilde V^i] + b^i(t,x) \cdot D \tilde V^i(t,x) + c^i \tilde V^i + \tilde d^i = \sum_{j\ne i} \lambda^i_j \tilde V^j + \sum_{k, \ell \ne i} \mu^i_{k\ell} \tilde V^k \tilde V^\ell,
	\]
	where
	\begin{align*}
		\tilde d^i(t,x) &= d^i(t,x) + C\delta e^{Ct} - \beta \tr[a^i D^2\nu_\beta] + \beta b^i \cdot D\nu_\beta + \beta \nu_\beta(x) \left( c^i - \sum_{j \ne i} \lambda^i_j - \sum_{k,\ell \ne i} \mu^i_{k \ell} V^k \right)\\
		& + \delta e^{Ct} \left( c^i - \sum_{j \ne i} \lambda^i_j - \sum_{k,\ell \ne i} \mu^i_{k \ell}V^k\right).
	\end{align*}
	From the boundedness of the coefficients and the $V^i$'s, and from the nonnegativity of $d^i$, it follows that there exists $\tilde C > 0$ depending only on the bounds of the coefficients and the $V^i$ such that, as $\beta \to 0$,
	\[
		\tilde d^i(t,x) \ge (C - \tilde C) \delta e^{Ct} - o(1).
	\]
	Taking $C > \tilde C$ and letting $\beta$ be sufficiently small, in relation to $\delta$, we then see that $\tilde d^i > 0$. From the first step, we conclude that, if $\beta$ is sufficiently small, then, for all $(t,x) \in [0,T] \times \RR^d$ and $i \in \{1,2,\ldots,M\}$,
	\[
		V^i(t,x) < \beta \nu_\beta(x) + \delta e^{Ct}.
	\]
	Sending first $\beta \to 0$ and then $\delta \to 0$ yields the result.
\end{proof}

\section{The regular Lagrangian flow}\label{sec:ODE}

The object of this section is to study the forward-in-time flow
\begin{equation}\label{ODE}
	\del_t \phi_{t,s}(x) = b(t,\phi_{t,s}(x)) \quad \text{for } s \in [0,T], \quad t \in [s,T], \quad \phi_{s,s}(x) = x
\end{equation}
for a vector field $b:[0,T] \times \RR^d \to \RR^d$ satisfying
\begin{equation}\label{A:b}
	\left\{
	\begin{split}
		&\text{for some } C_0, C_1 \in L^1_+([0,T]),\\
		& |b(t,y)| \le C_0(t) (1 + |y|) \quad \text{for a.e. } (t,y) \in [0,T] \times \RR^d,  \text{ and}\\
		&x \mapsto b(t,x) + C_1(t) x \quad \text{is increasing  for a.e. }t \in [0,T].
	\end{split}
	\right.
\end{equation}
The second condition reads equivalently as $\del_{x_j} b^i(t,\cdot) \ge -C_1(t) \delta_{ij}$ for all $i,j = 1,2,\ldots, d$. Moving forward, for convenience, we define the positive, increasing, absolutely continuous functions
\begin{equation}\label{acomega}
	\omega_0(t) = \int_0^t C_0(s)ds, \quad \omega_1(t) = \int_0^t C_1(s)ds, \quad t \in [0,T].
\end{equation}

\begin{remark}
	Implicit in the assumption \eqref{A:b} is a choice of basis on $\RR^d$. Indeed, the results of the paper continue to hold if $b$ is replaced by $A b(t,A^Tx)$ for a $d \times d$ orthogonal matrix $A$.
\end{remark}

The precise interpretation of the problem \eqref{ODE}, wherein $b$ is discontinuous, is made sense of as a differential inclusion. Namely, at every discontinuity $x$ of $b(t,\cdot)$, we have $b_*(t,x) \lneq b^*(t,x)$, and so the natural formulation is
\begin{equation}\label{ODE:inc}
	\del_t \phi_{t,s}(x) \in [b_*(t,\cdot), b^*(t,\cdot)](\phi_{t,s}(x)) \quad \text{for }s \in [0,T], \quad t \ge s, \quad \phi_{s,s}(x) = x.
\end{equation}

\begin{remark}
	In order to make notation less cumbersome, for a function $\phi: [0,T] \times \RR^d \to \RR$, we will always denote by $\phi_*$ and $\phi^*$ the lower and upper semicontinuous envelopes of $\phi$ in the space variable only; that is, for $(t,x) \in [0,T] \times \RR^d$, $\phi_*(t,x) = \phi(t,\cdot)_*(x)$ and $\phi^*(t,x) = \phi(t,\cdot)^*(x)$. 
	
	If $\phi: \RR^d \to \RR^m$ for $m > 1$, then $\phi^*$ and $\phi_*$ are taken to be the coordinate-by-coordinate lower and upper-semicontinuous envelopes, so, for instance, $\phi_* = (\phi^1_*, \phi^2_*, \ldots, \phi^d_*)$; equivalently,
	\[
		\phi_*(x) = \sup_{r > 0} \inf_{|x-y| \le r} \phi(y)
	\]
	with respect to the partial order \eqref{order} on $\RR^m$. 
\end{remark}

\begin{remark}\label{R:F?}
	The differential inclusion \eqref{ODE:inc}, wherein $b(t,x)$ is replaced with the smallest box $[\alpha,\beta] = \prod_{i=1}^d [\alpha_i,\beta_i]$ containing all limit points of $b(t,y)$ as $y \to x$, is a slightly weaker formulation than the standard Filippov regularization \cite{Fil60}, where boxes are replaced with general convex sets.
\end{remark}

Before developing the general theory, it is useful to record the following a priori ODE bounds, which are an easy consequence of Gr\"onwall's lemma.

\begin{lemma}\label{ODE:apriori}
	Assume, for some $C_0 \in L^1_+([0,T])$, that $b: [0,T] \times \RR^d \times \RR^d$ satisfies
	\[
		\esssup_{x \in \RR^d} \frac{|b(t,x)|}{1 + |x|} \le C_0(t) \quad \text{for a.e. } t \in [0,T].
	\]
	Then there exists a constant $C > 0$, depending only on $T > 0$, such that
	\[
		|X(t) - X(s)| \le C(1 + |x|) \int_s^t C_0(r)dr.
	\]
\end{lemma}

\subsection{A comparison principle} The following comparison result for ODEs is at the heart of much of the analysis of this paper. It leads to the existence and uniqueness of the regular Lagrangian flow, as well as stable notions of solutions to the transport and continuity equations with velocity fields satisfying \eqref{A:b}.

\begin{lemma}\label{L:comparison}
	Assume that $B: [0,T] \times \RR^d \to \RR^d$ satisfies
	\begin{equation}\label{B:localLip}
		\left\{
		\begin{split}
		&\left\{ t \mapsto \norm{\nabla B(t,\cdot)}_{L^\oo(B_R)} \right\} \in L^1([0,T]) \quad \text{for all }R > 0, \text{ and}\\
		&\del_{x_i} B^j \ge 0 \quad \text{for all } i \ne j.
		\end{split}
		\right.
	\end{equation}
	Let $X,Y \in C^{0,1}([0,T], \RR^d)$ be such that, with respect to the partial order \eqref{order}, $X(0) \le Y(0)$ and
	\begin{equation}\label{E:comparison}
		\dot X(t) \le B(t,X(t)) \quad \text{and} \quad \dot Y(t) \ge B(t,Y(t)) \quad \text{for a.e. } t \in [0,T].
	\end{equation}
	Then $X(t) \le Y(t)$ for all $t \in [0,T]$.
\end{lemma}

\begin{proof}
	The continuity of $X$ and $Y$ implies that there exists $R > 0$ such that $|X(t,x)| \vee |Y(t,x)| \le R$ for all $(t,x) \in [0,T] \times \RR^d$. We may thus assume without loss of generality that $\int_0^T \norm{\nabla B(t,\cdot)}_\oo dt < \oo$.

	Define $\Delta = X - Y$ and, for $i,j = 1,2,\ldots, d$,
	\[
		a_{ij}(t) = \int_0^1 \del_{x_j} B^i(t, \tau X(t) + (1-\tau) Y(t))d\tau.
	\]
	Observe that $\int_0^T \norm{a_{ij}(t,\cdot)}_\oo dt < \oo$ for all $i,j$, and $a_{ij} \ge 0$ whenever $i \ne j$. Then, for $i = 1,2,\ldots, d$,
	\[
		\Delta^i(0) \le 0 \quad \text{and} \quad \dot \Delta^i(t) \le a_{ii}(t) \Delta^i(t) + \sum_{i\ne j} a_{ij}(t) \Delta^j(t) \quad \text{for }t \in [0,T].
	\]	
	Fix $\delta > 0$, set
	\[
		\psi_1(t) = \int_0^t \max_{i=1,2,\ldots, d} a_{ii}(s)ds \quad \text{and} \quad \psi_2(t) = \int_0^t \max_{i=1,2,\ldots, d} \sum_{j \ne i} a_{ij}(s)ds,
	\]
	and define
	\[
		t_0 := \inf \left\{ t \in (0,T] : \text{there exists } i  \in \{1,2,\ldots ,d\} \text{ such that } e^{-\psi_1(t)} \Delta^i(t) - \delta e^{\psi_2(t)} > 0 \right\}.
	\]
	We have $t_0 > 0$. Assume by contradiction that $t_0 < T$, and let $i$ be such that the maximum is attained. Then, for all $j$ and $s \in [0,t_0]$,
	\begin{equation}\label{maxcond}
		e^{- \psi_1(s)} \Delta^j(s) \le e^{-\psi_2(t_0)} \Delta^i(t_0) = \delta e^{\psi_2(t_0)}.
	\end{equation}
	We compute
	\[
		\frac{d}{dt} e^{- \psi_1(t)} \Delta^i(t)
		\le \sum_{j\ne i} a_{ji}(t) e^{-\psi_1(t)} \Delta^j(t),
	\]
	and so
	\begin{align*}
		\delta e^{\psi_2(t_0)}
		= e^{-\psi_1(t_0)} \Delta^i(t_0)
		\le \int_0^{t_0} \sum_{j\ne i} a_{ji}(s) e^{-\psi_1(s)} \Delta^j(s)ds
		\le \delta \int_0^{t_0} \sum_{j \ne i} a_{ji}(s) e^{\psi_2(s)}ds
		\le \delta ( e^{\psi_2(t_0)} - 1 ),
	\end{align*}
	which is a contradiction. It follows that $\Delta^i(t) < \delta e^{\psi_1(t) + \psi_2(t)}$ for all $t \in [0,T]$ and $i = 1,2,\ldots, d$. Sending $\delta \to 0$ yields the result.
\end{proof}

\subsection{Maximal and minimal flows}

The comparison principle from the previous subsection is now used to establish the existence of maximal and minimal (with respect to the order \eqref{order}) semicontinuous solutions of the differential inclusion \eqref{ODE:inc} for vector fields satisfying $\del_{x_i} b^j \ge 0$ for $i \ne j$.

\begin{proposition}\label{P:maxmin}
	Assume that $b: [0,T] \times \RR^d \to \RR^d$ satisfies
	\[
		\int_0^T \esssup_{y \in \RR^d} \frac{|b(t,y)|}{1 + |y|} dt < \oo \quad \text{and} \quad
		\del_{x_i} b^j \ge 0 \quad \text{for all } i \ne j.
	\]
	Then there exist solutions $\phi^+$ and $\phi^-$ of \eqref{ODE:inc} that are absolutely continuous in time such that
	\begin{enumerate}[(a)]
	\item\label{maxmin:props} For all $0 \le s \le t \le T$, $\phi^+_{t,s}$ and $\phi^-_{t,s}$ are increasing in the sense of \eqref{order}, $\phi^+_{t,s}$ is right-continuous, and $\phi^-_{t,s}$ is left-continuous.
	\item\label{maxmin:sandwich} If $\phi$ is any other solution of \eqref{ODE:inc}, then $\phi^- \le \phi \le \phi^+$.
	\item\label{maxmin:flow} If $0 \le r \le s \le t \le T$, then $\phi^+_{s,t} \circ \phi^+_{r,s} = \phi^+_{r,t}$ and $\phi^-_{s,t} \circ \phi^-_{r,s} = \phi^-_{r,t}$
	\item\label{maxmin:env} For all $0 \le s \le t \le T$, $(\phi^+_{t,s})_* \ge \phi^-_{t,s}$.
	\end{enumerate}
\end{proposition}

\begin{remark}
	When $d = 1$, the condition on the derivatives of $b$ becomes vacuous, and we recover the existence of a unique minimal and maximal solution under the sole assumption that $b$ is locally bounded.
\end{remark}

\begin{proof}[Proof of Proposition \ref{P:maxmin}]
	For $\eps > 0$, let $b^\eps$ and $b_\eps$ denote the coordinate-by-coordinate $\sup$- and $\inf$-convolutions as in \eqref{infsup}. By Lemma \ref{L:infsup}, $\nabla b^\eps$ and $\nabla b_\eps$ are bounded on $[0,T] \times \RR^d$, and so, for all $(s,x) \in [0,T] \times \RR^d$, there exist unique solutions $\phi^{\pm,\eps}$ of
	\[
		\del_t \phi^{+,\eps}_{t,s}(x) = b^\eps(t, \phi^{+,\eps}_{t,s}(x)) \quad \text{and} \quad
		\del_t \phi^{-,\eps}_{t,s}(x) = b_\eps(t, \phi^{-,\eps}_{t,s}(x)) \quad \text{in } [s,T], \quad \phi^{+,\eps}_{s,s}(x) = \phi^{-,\eps}_{s,s}(x) = x.
	\]
	By Lemma \ref{ODE:apriori}, for every $(s,x) \in [0,T] \times \RR^d$, $\phi^{+,\eps}_{\cdot,s}(x)$ and $\phi^{-,\eps}_{\cdot,s}$ are bounded and continuous uniformly in $\eps$.
	
	The comparison result, Lemma \ref{L:comparison}, immediately gives $\phi^{-,\eps} \le \phi^{+,\eps}$, and, if $x \le y$, $\phi^{-,\eps}_{t,s}(x) \le \phi^{-,\eps}_{t,s}(y)$ and $\phi^{+,\eps}_{t,s}(x) \le \phi^{+,\eps}_{t,s}(y)$.
	
	Fix $0 < \eps < \delta$. Then Lemma \ref{L:infsup}\eqref{updownlimits} implies that $b^\eps \le b^\delta$ and $b_\eps \ge b_\delta$. It follows that
	\[
		\del_t \phi^{+,\eps}_{t,s} \le b^{\delta}(t, \phi^{+,\eps}_{t,s}),
	\]
	and so Lemma \ref{L:comparison} yields $\phi^{+,\eps} \le \phi^{+,\delta}$. A similar argument gives $\phi^{-,\eps} \ge \phi^{-,\delta}$. It now follows that that there exist bounded and Lipschitz functions $\phi^+_{\cdot,s}(x)$ and $\phi^-_{\cdot,s}(x)$ on $[s,T]$, which are increasing in $x$, such that, as $\eps \to 0$,
	\[
		\phi^{+,\eps}_{t,s}(x) \searrow \phi^+_{t,s}(x) \quad \text{and} \quad \phi^{-,\eps}_{t,s}(x) \nearrow \phi^-_{t,s}(x) \quad \text{uniformly for }t \in [s,T].
	\]
	Fix $0 \le s \le t \le T$ and $x \in \RR^d$. If $h \ge 0$, then, for arbitrary $\eps > 0$,
	\[
		\phi^+_{t,s}(x) \le \phi^+_{t,s}(x + h) \le \phi^{+,\eps}_{t,s}(x+h).
	\]
	Sending first $h \searrow 0$ and then $\eps \to 0$ gives $\lim_{h \searrow 0} \phi^+_{t,s}(x+h) = \phi^+_{t,s}(x)$, so that $\phi^+_{t,s}$ is right-continuous. A similar argument shows that $\phi^-_{t,s}(x)$ is left-continuous, and this completes the proof of part \eqref{maxmin:props}.
	
	We next prove part \eqref{maxmin:sandwich}. If $\phi$ is any other solution of \eqref{ODE:inc}, then, for all $\eps > 0$,
	\[
		b_\eps(t, \phi_{t,s}(x)) \le \del_t \phi_{t,s}(x) \le b^\eps(t,\phi_{t,s}(x)) \quad \text{for } t \in [s,T] \quad \text{and} \quad \phi_{s,s}(x) = x.
	\]
	Appealing once more to Lemma \ref{L:comparison} gives $\phi^{-,\eps}_{t,s}(x) \le \phi_{t,s}(x) \le \phi^{+,\eps}_{t,s}(x)$, and sending $\eps \to 0$ gives the result.
	
	We now prove the flow property stated in part \eqref{maxmin:flow}. Suppose $x \in \RR^d$ and $0 \le r \le s \le t \le T$. Then, for $\eps > 0$, because $\phi^{+,\eps}_{s,t}$ is increasing,
	\[
		\phi^{+,\eps}_{s,t} \left( \phi^+_{r,s}(x) \right) \le \phi^{+,\eps}_{s,t} \left( \phi^{+,\eps}_{r,s}(x) \right)
		= \phi^{+,\eps}_{r,t}(x).
	\]
	Taking $\eps \to 0$ yields $\phi^+_{s,t} \left( \phi^+_{r,s}(x) \right) \le \phi^+_{r,t}(x)$. For the opposite inequality, observe that $[s,T] \ni t \mapsto \phi^+_{r,t}(x)$ solves \eqref{ODE:inc} with value $\phi^+_{r,s}(x)$ at time $t = s$. It follows from part \eqref{maxmin:sandwich} that $\phi^+_{r,t}(x) \le \phi^+_{s,t} \left( \phi^+_{r,s}(x) \right)$. The argument for the flow property for $\phi^-$ is analogous.
	
	Let $x \in \RR^d$. Sending $y \nearrow x$ in the inequality $\phi^-_{t,s}(y) \le \phi^+_{t,s}(y)$, using the fact that $\phi^-$ and $\phi^+$ are increasing and $\phi^-$ is left-continuous, yields $\phi^-_{t,s}(x) \le (\phi^+_{t,s})_*(x)$. This concludes the proof of part \eqref{maxmin:env}.
\end{proof}

\subsection{Uniqueness and stability of the regular Lagrangian flow}

We now use the full assumption \eqref{A:b}, and in particular, the lower bounds on $\del_{x_i} b^i$ for $i = 1,2,\ldots, d$ that were not needed to prove the existence of the maximal and minimal solution of \eqref{ODE:inc}. In particular, we prove that $\phi^+$ and $\phi^-$ are equal almost everywhere, giving rise to a unique regular Lagrangian flow.

\begin{proposition}\label{P:shiftover}
	Assume that $b$ satisfies \eqref{A:b}. If $\phi^+$ and $\phi^-$ are the maximal and minimal flow from Proposition \ref{P:maxmin}, then, for all $0 \le s \le t \le T$, $\left(\phi^+_{t,s}\right)_* = \phi^-_{t,s}$. Moreover, for all $(s,x) \in [0,T] \times \RR^d$ and a.e. $t \in [s,T]$,
	\[
		\del_t \phi^+_{t,s}(x) = b^*(t,\phi^+_{t,s}(x)) \quad \text{and} \quad \del_t \phi^-_{t,s}(x) = b_*(t,\phi^-_{t,s}(x)).
	\]
\end{proposition}

\begin{proof}
	Let 
	\[
		\tilde \phi^\pm_{t,s}(x) = e^{\omega_1(t)} \phi^\pm_{t,s}\left(e^{-\omega_1(t)} x \right) 
	\]
	and
	\[
		\tilde b(t,x) = e^{\omega_1	(t)} b\left(t, e^{-\omega_1(t)} x \right) + C_1(t) x,
	\]
	where $\omega_1$ is as in \eqref{acomega}. Then $\tilde b$ satisfies \eqref{A:b} with $C_1 = 0$ and with a possibly different $C_0 \in L^1_+$, and $\tilde \phi^{\pm}$ are the corresponding maximal and minimal flows. We may therefore assume without loss of generality that $b(t,\cdot)$ is increasing for $t \in [0,T]$.
	
	We first prove the statement involving the lower-semicontinuous envelope of $\phi^+_{t,s}$, and in view of Proposition \ref{P:maxmin}\eqref{maxmin:env}, we need only prove the opposite inequality $(\phi^+_{t,s})_* \le \phi^-_{t,s}$.  For $t \in [0,T]$, define
	\[
		b^\eps(t,\cdot) = b(t,\cdot) * \rho^\eps \quad \text{and} \quad b_\eps(t,\cdot) = b(t,\cdot) * \rho_\eps,
	\]
	where the one-sided mollifiers $\rho^\eps$ and $\rho_\eps$ are defined in \eqref{onesidedmollifier}, and let $\phi^{+,\eps}$ and $\phi^{-,\eps}$ be the corresponding flows. Arguing exactly as in the proof of Proposition \ref{P:maxmin}, appealing to Lemma \ref{L:approxid}, as $\eps \to 0$, $\phi^{+,\eps}_{\cdot,s} (x)\searrow \phi^+_{\cdot,s}(x)$ and $\phi^{-,\eps}_{\cdot,s}(x) \nearrow \phi^-_{\cdot,s}(x)$ uniformly on $[s,T]$, for any $s \in [0,T]$ and $x \in \RR^d$.

	Fix $0 < \eps < \delta$ and define $X^{\delta,\eps}(t) = \phi^{+,\eps}_{t,s}(x - 2\delta \mbf 1) + 2\delta \mbf 1$. Then, in view of Lemma \ref{L:approxid}\eqref{moll:shift}, $X^{\delta,\eps}$ satisfies
	\[
		\dot X^{\delta,\eps}(t) = b^\eps(t, X^{\delta,\eps}(t) - 2\delta \mbf 1) 
		\le b_\eps(t, X^{\delta,\eps}(t) )
		\quad \text{in } [s,T] \quad \text{and} \quad X^{\delta,\eps}(s) = x.
	\]
	We may thus appeal to Lemma \ref{L:comparison}, and find that, for all $t \in [s,T]$,
	\[
		\phi^{+,\eps}_{t,s}(x - 2\delta \mbf 1) \le \phi^{-,\eps}_{t,s}(x) - 2 \delta \mbf 1.
	\]
	Sending first $\eps \to 0$ and then $\delta \to 0$ gives $(\phi^+_{t,s})_* \le \phi^-_{t,s}$, as desired.
	
	We now prove the final statement. For $0 < \eps < \delta$, we have the inequalities, for $t \in [s,T]$ and $x \in \RR^d$,
	\[
		b^*(t, \phi^+_{t,s}(x)) \le b^*(t, \phi^{+,\eps}_{t,s}(x)) \le b^\eps(t, \phi^{+,\eps}_{t,s}(x)) \le b^\delta(t, \phi^{+,\eps}_{t,s}(x)).
	\]
	Since $b^\delta$ is Lipschitz continuous in the space variable, 
	\[
		\lim_{\eps \to 0} b^\delta(t, \phi^{+,\eps}_{t,s}(x)) = b^\delta(t, \phi^+_{t,s}(x)).  
	\]
	Therefore,
	\[
		b^*(t, \phi^{+}_{t,s}(x))
		\le \liminf_{\eps \to 0} b^\eps(t, \phi^{+,\eps}_{t,s}(x))
		\le \limsup_{\eps \to 0} b^\eps(t, \phi^{+,\eps}_{t,s}(x))
		\le b^\delta(t, \phi^+_{t,s}(x)),
	\]
	and so, sending $\delta \to 0$, it follows that
	\[
		\lim_{\eps \to 0} b^\eps(t, \phi^{+,\eps}_{t,s}(x)) = b^*(t, \phi^+_{t,s}(x)).
	\]
	For arbitrary $\eps > 0$, $0 \le s \le t \le T$, and $x \in \RR^d$,
	\[
		\phi^{\eps,+}_{t,s}(x) = x + \int_s^t b^\eps(r, \phi^{\eps,+}_{r,s}(x))dr.
	\]
	Sending $\eps \to 0$ and appealing to dominated convergence gives
	\[
		\phi^+_{t,s}(x) = x + \int_s^t b^*(r, \phi^+_{r,s}(x))dr,
	\]
	and a similar argument gives
	\[
		\phi^-_{t,s}(x) = x + \int_s^t b_*(r, \phi^-_{r,s}(x))dr.
	\]

\end{proof}

Recalling that increasing functions are continuous almost-everywhere, it now follows from Proposition \ref{P:shiftover} that $\phi^+_{t,s}$ and $\phi^-_{t,s}$ are equal almost everywhere, given $0 \le s \le t \le T$. We thus finally arrive at the almost-everywhere unique solvability of the ODE \eqref{ODE}, and the identification of the unique regular Lagrangian flow.

\begin{theorem}\label{T:regLag}
	For every $s \in [0,T]$ and for almost every $x \in \RR^d$, there exists a unique absolutely continuous solution $[s,T] \ni t \mapsto \phi_{t,s}(x)$ of the differential inclusion \eqref{ODE:inc}. Moreover, $\phi_{t,s}$ satisfies the regular Lagrangian property: the map $C_c(\RR^d) \ni f \mapsto f \circ \phi_{t,s}$ extends continuously to $f \in L^1_\loc(\RR^d)$, and there exists $C > 0$ depending only on $T$ such that, for all $R > 0$, $f \in L^1_\loc$, $t \in [s,T]$,
	\begin{equation}\label{regLag}
		\norm{f \circ \phi_{t,s}}_{L^1(B_R)}  \le e^{d (\omega_1(t) - \omega_1(s))} \norm{f}_{L^1(B_{R + C(1+R)(\omega_0(t) - \omega_0(s))})}.
	\end{equation}
	In particular,
	\[
		\left\{ [s,T] \times \RR^d \ni (t,x) \mapsto b(t, \phi_{t,s}(x)) \in L^1_\loc(\RR^d, L^1([s,T]) ) \right\},
	\]
	and for a.e. $x \in \RR^d$, $[s,T] \ni t \mapsto \phi_{s,t}(x)$ is the unique absolutely continuous solution of the integral equation
	\[
		\phi_{t,s}(x) = x + \int_s^t b(r, \phi_{r,s}(x))dr.
	\]
	Finally, if $0 \le r \le s \le t \le T$, then the composition $\phi_{s,t} \circ \phi_{r,s}$ is well-defined a.e. and is equal to $\phi_{r,t}$.
\end{theorem}

\begin{remark}\label{R:regLag}
	Taking $f = \ind_A$ in the estimate above, for some $A \subset \RR^d$ of finite measure, we find the regular Lagrange property
	\[
		\left| \phi^{-1}_{t,s}(A) \right| \le e^{d(\omega_1(t) - \omega_1(s))} |A|.
	\]
\end{remark}

\begin{proof}[Proof of Theorem \ref{T:regLag}]
	Let $s \in [0,T]$ be fixed and, for $N \in \NN$, define $t^N_n = s + \frac{n(T-s)}{N}$, $n = 0, 1, 2, \ldots, N$. Then, by Lemma \ref{L:ABVcts} and Proposition \ref{P:shiftover}, there exist sets $B^N_n \subset \RR^d$ of full measure such that $\phi^+_{t^N_n,s} = \phi^-_{t^N_n,s}$ on $B^N_n$. Define now the full measure set
	\[
		B = \bigcap_{N \in \NN} \bigcap_{n = 0}^N B^N_n.
	\]
	Let $(t,x) \in [s,T] \times B$, and, for any $N \in \NN$, let $n(t) = 0,1,2,\ldots, N$ be such that $|t^N_n - t|$ is minimized. Then, in view of the uniform continuity in time of $\phi^+$ and $\phi^-$, for some $\omega_N > 0$ with $\lim_{N \to \oo} \omega_N = 0$,
	\[
		|\phi^+_{t,s}(x) - \phi^-_{t,s}(x)| \le |\phi^+_{t^N_{n(t)},s}(x) - \phi^+_{t,s}(x)| + |\phi^-_{t^N_{n(t)},s}(x) - \phi^-_{t,s}(x)| \le \omega_N,
	\]
	and, since $N$ was arbitrary, we find that $\phi^+_{t,s}(x) = \phi^-_{t,s}(x)$. It follows that, up to a full measure set, $\phi^+_{\cdot,s}$ and $\phi^-_{\cdot,s}$ may be identified as absolutely continuous functions on $[s,T]$.
	
	For $f \in L^1(\RR^d) \cap C_c(\RR^d)$, let $\eps > 0$ and let $b^\eps$ and $\phi^{+,\eps}$ be as in the proof of Proposition \ref{P:maxmin}. Then
	\[
		\del_t \det(D_x \phi^{+,\eps}_{t,s}(x)) = \div b^\eps(t, \phi^{+,\eps}_{t,s}(x)) \det(D_x \phi^{+,\eps}_{t,s}(x)) \ge -dC_1 (t)\det(D_x \phi^{+,\eps}_{t,s}(x)) ,
	\]
	from which it follows that $\det(D_x \phi^{+,\eps}_{t,s}(x))  \ge \exp\left( - d\int_s^t C_1(r)dr \right)$. The change of variables formula then gives
	\[
		\int_{\RR^d} f(\phi^{+,\eps}_{t,s}(x))dx \le \exp\left(  d\int_s^t C_1(r)dr \right) \int_{\RR^d} f(x)dx.
	\]
	The set $(\phi^{+,\eps}_{t,s}(x))_{x \in \supp f, \, \eps > 0}$ is uniformly bounded in view of Lemma \ref{ODE:apriori}, and so the bounded convergence theorem implies, upon taking $\eps \to 0$, that
	\[
		\int_{\RR^d} f(\phi_{t,s}(x))dx \le \exp\left(  d\int_s^t C_1(r)dr \right)  \int_{\RR^d} f(x)dx.
	\]
	This implies that the map $f \mapsto f \circ \phi_{t,s}$ extends continuously to $L^1(\RR^d)$. The local statement follows from the finite speed of propagation implied by the a priori estimates in Lemma \ref{ODE:apriori}.
	
	It now follows easily that $b(\cdot, \phi_{\cdot,s})$ belongs to $L^1_\loc([0,T] \times \RR^d)$, and the uniqueness statement for absolutely continuous solutions of the integral equation follows from Proposition \ref{P:shiftover} and the fact that $\phi = \phi^+ = \phi^-$ a.e.
	
	Finally, the composition $\phi_{t,s} \circ \phi_{r,s}$ is justified because $\phi_{t,s} \in L^1_\loc$, and its equality to $\phi_{r,t}$ a.e. is a consequence of the a.e. uniqueness of the ODE and the flow properties in Proposition \ref{P:maxmin}\eqref{maxmin:flow}.
\end{proof}

We now demonstrate that any regularizations of $b$ lead to the flow $\phi$ in the limit, not only the one-sided regularizations used above.

\begin{theorem}\label{T:ODE:beps}
	Assume $(b^\eps)_{\eps > 0}$ is a family satisfying \eqref{A:b} uniformly in $\eps$, and, for every $\eps$, 
	\[
		\int_0^T \norm{\nabla b^\eps(t,\cdot)}_\oo dt < \oo.
	\]
	If $\phi^\eps$ is the unique flow corresponding to $b^\eps$, then, for all $s \in [0,T]$ and $p \in [1,\oo)$, as $\eps \to 0$, $\phi^\eps \to \phi$ in $C([s,T], L^p_\loc(\RR^d))$ and in $L^p_\loc(\RR^d, W^{1,1}([s,T]))$.
\end{theorem}

\begin{proof}
	The regular Lagrange property \eqref{regLag} in Theorem \ref{T:regLag} implies that, for all $p \in [1,\oo)$ and $R > 0$,
	\[
		\sup_{\eps > 0} \sup_{t \in [s,T]}\int_{B_R} |\phi^\eps_{t,s}(x)|^p dx < \oo.
	\]
	Moreover, for all $t \in [s,T]$, $\phi^\eps_{t,s}$ is increasing, and thus $BV$. This along with the a priori estimates in Lemma \ref{ODE:apriori} imply that $(\phi^\eps_{\cdot,s})_{\eps > 0}$ is precompact in $C([s,T], L^p_\loc(\RR^d))$, and therefore, for some subsequence $\eps_n \xrightarrow{n \to \oo} 0$, $\phi^{\eps_n}_{\cdot,s}$ converges as $n \to \oo$ in $C([s,T], L^p_\loc(\RR^d))$ to some $\psi_{\cdot,s}$, which then also satisfies the regular Lagrangian property \eqref{regLag}.
	
	We now use \eqref{regLag} to deduce that, for any $R > 0$, there exists $R' > 0$ independent of $\eps$ such that, for a.e. $r \in [s,T]$ and for any $\delta > \eps$,
	\begin{align*}
		&\norm{b^\eps(r, \phi^\eps_{r,s}) - b(r, \psi_{r,s}) }_{L^p(B_R)}\\
		&\le \norm{b^\eps(r,\phi^\eps_{r,s}) - b^\delta(r, \phi^\eps_{r,s})}_{L^p(B_R)}
		+ \norm{b^\delta(r, \phi^\eps_{r,s}) - b^\delta(r, \psi_{r,s} }_{L^p(B_R)}
		+ \norm{b^\delta(r, \psi_{r,s}) - b(r, \psi_{r,s})}_{L^p(B_R)}\\
		&\le e^{d(\omega_1(t) - \omega_1(s))} \left( \norm{b^\eps(r,\cdot) - b^\delta(r,\cdot)}_{L^p(B_{R'})}
		+ \norm{b^\delta(r, \cdot) - b(r,\cdot)}_{L^p(B_{R'})} \right) +  \norm{b^\delta(r, \phi^\eps_{r,s}) - b^\delta(r, \psi_{r,s} }_{L^p(B_R)}.
	\end{align*}
	Taking first $\eps_n \to 0$ (using the Lipschitz continuity of $b^\delta$) and then $\delta \to 0$, we find that the right-hand side above converges to zero. By Lemma \ref{ODE:apriori}, we may use the dominated convergence theorem to deduce that
	\[
		\lim_{n \to \oo \to 0} \int_s^T \norm{b^{\eps_n}(r, \phi^{\eps_n}_{r,s}) - b(r, \psi_{r,s}) }_{L^p(B_R)}dr = 0.
	\]
	Sending $n \to \oo$ in the integral equation
	\[
		\phi^{\eps_n}_{t,s}(x) = x + \int_s^t b^{\eps_n}(r, \phi^{\eps_n}_{r,s}(x))dr
	\]
	thus gives, in the distributional sense,
	\begin{equation}\label{psi?}
		\psi_{t,s}(x) = x + \int_s^t b(r,\psi_{r,s}(x))dr.
	\end{equation}
	We also have, by Minkowski's inequality,
	\[
		\norm{\del_t \phi^{\eps_n}_{s,\cdot} - \del_t \psi_{s,\cdot} }_{L^p(B_R, L^1([s,T])}
		\le \int_s^T \norm{b^{\eps_n}(r, \phi^{\eps_n}_{r,s}) - b(r, \psi_{r,s}) }_{L^p(B_R)}dr \xrightarrow{n \to \oo} 0,
	\]
	and thus \eqref{psi?} is satisfied in the integral sense for a.e. $x \in \RR^d$. By Theorem \ref{T:regLag}, $\psi = \phi$, and therefore the convergence statements hold for the full family $\eps \to 0$.
\end{proof}

Another means of regularization is through the addition of stochastic noise. In particular, if $s \in [0,T]$, $x \in \RR^d$, $\eps > 0$, and $W: \Omega \times [0,T] \to \RR^d$ is a Wiener process defined over a given probability space $(\Omega, \mcl F ,\mbb P)$, then there exists a unique strong solution of the SDE
\begin{equation}\label{SDE:eps}
	d_t \phi^\eps_{t,s}(x) = b(t, \phi^\eps_{t,s}(x))dt + \sqrt{2\eps} dW_t.
\end{equation}
This is true even if $b$ is merely locally bounded; see \cite{D07, V80}.

The following is then proved exactly as for Theorem \ref{T:ODE:beps}.

\begin{theorem}\label{T:ODE:epsW}
	Assume $b$ satisfies \eqref{A:b} and let $\phi^\eps$ be the solution of \eqref{SDE:eps}. Then, for all $s \in [0,T]$, with probability one, as $\eps \to 0$, $\phi^\eps_{s,\cdot}$ converges in $C([s,T], L^p_\loc(\RR^d))$ and $L^p_\loc(\RR^d, C([s,T]))$ to $\phi$.
\end{theorem}

\begin{proof}
	The transformation $\tilde \phi_{t,s}(x) = \phi_{t,s}(x) - \sqrt{2\eps}(W_t - W_s)$ leads to the equation
	\[
		\del_t \tilde \phi^\eps_{t,s}(x) = b(t, \tilde \phi^\eps_{t,s}(x) + \sqrt{2\eps}(W_t - W_s)).
	\]
	The exact same arguments as those above show that $\tilde \phi^\eps_{t,s}$ is increasing for all $0 \le s \le t \le T$, satisfies the regular Lagrange property \eqref{regLag}, and, for any fixed Brownian path $W$, the a priori estimates of Lemma \ref{ODE:apriori} may be applied. Arguing as in the proof of Theorem \ref{T:ODE:beps}, we may then extract a subsequence $\eps_n \xrightarrow{n \to \oo} 0$ such that $\tilde \phi^{\eps_n}$, and therefore $\phi^{\eps_n}$, converges in $C([s,T], L^p_\loc(\RR^d))$ and $L^p_\loc(\RR^d, C([s,T])$. The same arguments as in Theorem \ref{T:ODE:beps} may then be used to conclude that any such limit is the unique regular Lagrangian flow $\phi$ from Theorem \ref{T:regLag}.
\end{proof}

If $b$ is smooth, then $\del_s \phi_{t,s}(x) = - b(s, \phi_{t,s}(x))$. It is then straightforward to show that all the same theory above can be developed for the terminal value problem for the flow corresponding to $-b$.

\begin{theorem}\label{T:-b}
	For every $t \in [0,T]$ and $p \in [1,\oo)$, $\phi_{t,\cdot} \in C([0,t], L^p_\loc(\RR^d)) \cap L^p_\loc(\RR^d, W^{1,1}([0,t]))$, and, for almost every $x \in \RR^d$, $[0,t] \ni s \mapsto \phi_{t,s}(x)$ is the unique absolutely continuous solution of 
	\[
		\phi_{t,s}(x) = x + \int_s^t b(r, \phi_{t,r}(x))dr, \quad s \in [0,t].
	\]	
\end{theorem}

\subsection{Some remarks on the backward flow} \label{ss:reverse?}
We discuss next the question of solvability of the backward flow: for fixed $s \in [0,T]$, this is the terminal value problem
\begin{equation}\label{ODE:backward}
	\del_t \phi_{t,s}(x) = b(t, \phi_{t,s}(x)) \quad t \in [0,s], \quad \phi_{s,s}(x) = x.
\end{equation}
Formally, the lower bound on $\div b$ suggests that the backward flow should concentrate positive measure sets to null sets. As an example when $d = 1$, if $b(t,x) = \sgn x$, then the backward flow is given by
\[
	\phi_{t,s}(x) = x - (\sgn x) (s-t) \quad \text{for } x \in \RR, \quad t < s,
\]
and so all trajectories eventually concentrate at $x = 0$. This situation can be generalized to multiple dimensions when $b$ satisfies the half-Lipschitz property
\begin{equation}\label{monotone}
	(b(t,x) - b(t,y)) \cdot (x-y) \ge -C_1(t) |x-y|^2, \quad C_1 \in L^1_+([0,T]),
\end{equation}
The condition \eqref{monotone} then implies the existence of a unique, Lipschitz, concentrating solution of the backward flow \eqref{ODE:backward}. This condition was used by Filippov \cite{Fil60} to build unique solutions of differential inclusions; see also \cite{Conway_67, Pe_Po_1d_99, Pe_Po_01} and the recent work of the authors \cite{LSother}.

The situation is very different when $b$ satisfies \eqref{A:b}. The differential inclusion corresponding to \eqref{ODE:backward} then takes the form
\begin{equation}\label{ODE:backwardinc}
	-\del_t \phi_{t,s}(x) \in \left[ -b^*(t, \cdot), - b_*(t, \cdot) \right]\left( \phi_{t,s}(x) \right).
\end{equation}
 As it turns out, there is no unique flow, and, in some cases, no Lipschitz flow.

\begin{example}\label{ex:line}
For $x = (x_1,x_2) \in \RR^2$, consider the vector field
\[
	b(x) = \sgn x_1 (1,1). 
\]
Since $b$ is independent of $t$, we formally write the flow $\phi_{t,s}$ as $\phi_{t-s}$. 

Observe that, for all $x_2 \in \RR$, $b_*(0,x_2) = -(1,1)$ and $b^*(0,x_2) = (1,1)$. It then follows that the direction in which $\phi_{-\tau}(0,x_2)$ moves for $\tau > 0$ lies in the box $[-1,1]^2$, and so any flow that solves \eqref{ODE:backwardinc} must take the form, for some $c \in [-1,1]$ and all $\tau \ge 0$,
\[
	\phi^{(c)}_{-\tau}(x) =
	\begin{dcases}
		x - (\sgn x_1)(1,1)\tau, & |x_1| \ge \tau\\
		(0,x_2 + c(\tau - |x_1|), & |x_1| < \tau.
	\end{dcases}
\]
There is therefore no unique flow in the strip $\{|x_1| < \tau\}$, even under the restriction that the flow be Lipschitz.

\end{example}

Recall that \eqref{ODE:backwardinc} is a weaker notion of solution than the classical Filippov formulation (see Remark \ref{R:F?}). In Example \ref{ex:line} above, this means $b(0,x_2)$ is taken to be the smallest convex set containing all limit points of $b(y)$ as $y \to (0,x_2)$, that is, the line segment connecting $(-1,-1)$ and $(1,1)$. The unique such solution is then the flow $\phi^{(0)}$. However, we can construct an example where even the Filippov flow is not unique. 

\begin{example}\label{ex:diagonal}
Taking again $d = 2$, we set
\[
	b(x_1,x_2) =
	\begin{dcases}
		(1,1), & x_1 > 0 \text{ and } x_2 > 0, \\
		(1,0), & x_1 < 0 \text{ and } x_2 > 0, \\
		(0,1), & x_1 > 0 \text{ and } x_2 < 0, \\
		(-1,-1), & x_1 < 0 \text{ and } x_2 < 0.
	\end{dcases}
\]
There are then three possible solution flows, even with the general Filippov formulation, when the starting point $(x_1,x_2)$ satisfies $x_1 = x_2$ and $\tau > |x_1| = |x_1|$. Indeed, the nonuniqueness occurs once trajectories reach the origin. Then \eqref{ODE:backwardinc} admits the three solutions
\[
	\phi^{(1)}_{-\tau}(0) = -\tau(1,0), \quad \phi^{(2)}_{-\tau}(0) = -\tau(0,1), \quad \text{and} \quad \phi^{(3)}_{-\tau}(0) = 0 \quad \text{for } \tau > 0.
\]
In fact, $\phi^{(k)}_{-\tau}$ is not continuous for any $\tau > 0$ and $k = 1,2,3$. This can be seen by considering the flows starting from $(x_1,x_2) < (0,0)$ with $x_1 \ne x_2$: for $\tau > 0$,
\[
	\phi^{(1)}_{-\tau}(x) = \phi^{(2)}_{-\tau}(x) = \phi^{(3)}_{-\tau}(x)
	=
	\begin{dcases}
		(x_1 + \tau, x_2 + \tau) & \text{if } 0 < \tau < -\max(x_1, x_2) \\
		(x_1 - 2x_2 - \tau, 0 ) & \text{if } x_1 < x_2 < 0 \text{ and } \tau > |x_2|, \\
		(0,x_2 - 2x_1 - \tau) & \text{if } x_2 < x_1 < 0 \text{ and } \tau > |x_1|.
	\end{dcases}
\]
In particular, given $\tau > 0$ and $\eps > 0$, for $k = 1,2,3$,
\[
	\phi^{(k)}_{-\tau}(-\tau/2 + \eps, -\tau/2) - \phi^{(k)}_{-\tau}(-\tau/2, -\tau/2+\eps)
	= \left(\frac{\tau}{2} + 2\eps \right) (1,-1).
\]
\end{example}

In \cite{PR_97}, Poupaud and Rascle explore the connection between the uniqueness (for \emph{every} $x \in \RR^d$) of Filippov solutions of \eqref{ODE:backward} and a stable notion of measure-valued solutions to the continuity equation\footnote{Observe, from the minus sign in front of the velocity field, that this is not the same initial value problem as \eqref{CE}.}
\begin{equation}\label{CE:reverse?}
	\del_t f - \div(b(t,x)f) = 0 \quad \text{in }(0,T) \times \RR^d \quad \text{and} \quad f(0,\cdot) = f_0.
\end{equation}
As Example \ref{ex:diagonal} shows, we cannot take this approach in analyzing \eqref{CE:reverse?}, because there are three distinct measure-valued solutions when $f_0 = \delta_0$.

\subsection{Stochastic differential equations}

For $b$ satisfying \eqref{A:b}, we now consider the stochastic flow for the SDE
\begin{equation}\label{SDE}
	d_t \Phi_{t,s}(x) = b(t, \Phi_{t,s}(x))dt + \sigma(t,\Phi_{t,s}(x)) dW_t, \quad t \in [s,T], \quad \Phi_{s,s}(x) = x.
\end{equation}
Here, for some $m \in \NN$, $W: [0,T] \times \Omega \to \RR^m$ is an $m$-dimensional Wiener process on a given probability space $(\Omega, \mcl F, \mbb P)$, and, for $(t,x) \in [0,T] \times \RR^d$, $\sigma$ is a $d \times m$-dimensional matrix denoted coordinate-wise by $\sigma_{ij}$, $i = 1,2,\ldots, d$, $j = 1,2,\ldots, m$.

Because $\sigma$ may be degenerate, the addition of noise does not necessarily imply the existence and uniqueness of a strong solution. In reproducing the theory for the ODE \eqref{ODE} , we are therefore led to the condition
\begin{equation}\label{A:sigma}
	\left\{
	\begin{split}
		&\sigma \in L^2([0,T], C^{0,1}(\RR^d; \RR^{d \times m})) \text{ and, for a.e. } t\in [0,T] \text{ and all } i = 1,2,\ldots, d, \; j = 1,2,\ldots, m,\\
		&\sigma_{ij}(t,\cdot) \text{ is independent of }x_k\text{ for } k \ne i.
	\end{split}
	\right.
\end{equation}
We assume in addition to \eqref{A:b} a bounded oscillation condition for $b$:
\begin{equation}\label{A:bosc}
	\left\{
	\begin{split}
	&\text{for some } C_2 \in L^1_+([0,T]), \text{ and for a.e. } (t,x,y) \in [0,T] \times \RR^d \times \RR^d,\\
	&|b(t,x) - b(t,y)| \le C_2(t)(|x-y| + 1).
	\end{split}
	\right.
\end{equation}
While not strictly necessary, this helps simplify some of the arguments by allowing certain regularizations of $b$ to be globally Lipschitz.

The reformulation of \eqref{SDE} as a differential inclusion then takes the form
\begin{equation}\label{SDE:inc}
	d_t \Phi_{t,s}(x) = \alpha_t dt + \sigma(t, \Phi_{t,s}(x))dW_t, \quad \alpha_t \in [b_*(t,\cdot), b^*(t,\cdot)](\Phi_{t,s}(x)), \quad \Phi_{s,s} = x.
\end{equation}

\begin{theorem}\label{T:SDE}
	Assume $b$ satisfies \eqref{A:b} and \eqref{A:bosc}, and $\sigma$ satisfies \eqref{A:sigma}. Then the following hold with probability one, for all $s \in [0,T]$.
	\begin{enumerate}[(a)]
	\item\label{SDE:maxmin} There exist solutions of $\Phi^+_{\cdot,s}$ and $\Phi^-_{\cdot,s}$ of \eqref{SDE:inc} such that, for all $t \in [s,T]$, $\Phi^+_{t,s}$ is upper-semicontinuous and increasing and $\Phi^-_{t,s}$ is lower-semicontinuous and increasing.
	\item\label{SDE:ordering} Any other solution $\Phi$ of \eqref{SDE:inc} satisfies $\Phi^-_{\cdot,s} \le \Phi_{\cdot,s} \le \Phi^+_{\cdot,s}$.
	\item\label{SDE:flow} If $0 \le r \le s \le t \le T$, then $\Phi^+_{s,t} \circ \Phi^+_{r,s} = \Phi^+_{r,t}$ and $\Phi^-_{s,t} \circ \Phi^-_{r,s} = \Phi^-_{r,t}$.
	\item\label{SDE:unique} For all $t \in [s,T]$, $(\Phi^+_{t,s})_* = \Phi^-_{t,s}$.
	\item\label{SDE:solution} For a.e. $x \in \RR^d$, \eqref{SDE} admits a unique integral solution, which is equal a.e. to $\Phi^+_{\cdot,s}$ and $\Phi^-_{\cdot,s}$.
	\item\label{SDE:regLag} There exists $\omega: [0,\oo) \to [0,\oo)$ such that $\lim_{r\to 0^+} \omega(r) = 0$ and, for all $f \in L^1(\RR^d)$ and $0 \le s \le t \le T$,
	\begin{equation}\label{regLag:SDE}
		\EE[ \norm{f \circ \Phi_{t,s}}_{L^1} ] \le e^{\omega(t-s)}\norm{f}_{L^1}.
	\end{equation}
	\end{enumerate}
\end{theorem}

We first establish an analogue of Lemma \ref{L:comparison}.

\begin{lemma}\label{L:SDEcomparison}
	Assume $B: [0,T] \times \RR^d \to \RR^d$ satisfies
	\begin{equation}\label{B:globalLip}
		\left\{
		\begin{split}
		&\left\{ t \mapsto \norm{\nabla B(t,\cdot)}_{\oo} \right\} \in L^1([0,T]) \text{ and}\\
		&\del_{x_i} B^j \ge 0 \quad \text{for all } i \ne j,
		\end{split}
		\right.
	\end{equation}
	and $\sigma$ satisfies \eqref{A:sigma}. Let $X,Y,\alpha,\beta: [0,T] \to \RR^d$ be adapted with respect to the Wiener process $W$ such that, with probability one, $X$ and $Y$ are continuous, $\alpha,\beta \in L^1([0,T])$,
	\[
		\alpha_t \le B(t,X_t) \quad \text{and} \quad \beta_t \ge B(t,Y_t) \quad \text{a.e. in } [0,T],
	\]
	and
	\[
		dX_t = \alpha_t dt + \sigma(t,X_t)dW_t, \quad dY_t = \beta_t dt + \sigma(t,X_t)dW_t \quad \text{in } [0,T], \quad X_0 \le Y_0.
	\]
	Then $X_t \le Y_t$ for all $t \in [0,T]$.
\end{lemma}

\begin{proof}
	For $t \in [0,T]$, define $\gamma_t = \alpha_t - \beta_t$, which satisfies $\gamma \in L^1$ and 
	\[
		\gamma^i_t \le B^i(t, X_t) - B^i(t,Y_t) = \sum_{j =1}^d \mu^{ij}_t(X^j_t - Y^j_t)
	\]
	for a.e. $t \in [0,T]$, where
	\[
		\mu^{ij}_(t) = \int_0^1 \del_j B^i(t, \theta X_t + (1-\theta)Y_t)d\theta.
	\]
	Note that $\mu^{ij} \in L^1$ for all $i,j \in \{1,2,\ldots, d\}$, $\mu^{ij} \ge 0$ for $i \ne j$, and $\mu^{ii}_t \le C(t)$ for a.e. $t \in [0,T]$ and some $C \in L^1_+$. Define also, for $i = 1,2,\ldots, d$ and $k = 1,2,\ldots, m$,
	\[
		\nu^{ik}_t = \int_0^1 \del_x \sigma_{ik}(t, \theta X^i_t + (1-\theta) Y^i_t)d\theta,
	\]
	which satisfies $\nu^{ik} \in L^2$ (recall that $\sigma_{ik}$ depends only on $(t,x_i)$). Therefore, if, for $i =1,2,\ldots, d$, we set $\Delta^i = X^i - Y^i$, we find that, for $i = 1,2,\ldots, d$,
	\[
		d\Delta^i_t = \gamma^i_t dt + \sum_{k=1}^m \nu^{ik}_t\Delta^i_t dW^k_t, \quad 
		\gamma^i_t \le \sum_{j=1}^d \mu^{ij}_t \Delta^j_t, \quad
		\Delta^i_0 \le 0.
	\]
	Now, for $i = 1,2,\ldots, d$, define
	\[
		Z^i_t = \exp\left( \int_0^t \left( \frac{1}{2} \sum_{k=1}^m (\nu^{ik}_s)^2 - C(s) \right)ds - \sum_{k=1}^m \int_0^t \nu^{ik}_s dW_s \right),
	\]
	which satisfies the SDE
	\[
		dZ^i_t = \left( \sum_{k=1}^m (\nu^{ik}_t)^2 - C(t) \right)Z^i_tdt - \sum_{k=1}^m \nu^{ik}_tZ^i_t dW^k_t, \quad Z^i_0 = 1.
	\]
	Then It\^o's formula yields
	\[
		d(Z^i_t \Delta^i_t) \le \sum_{j \ne i} \mu^{ij}_t Z^i_t \Delta^j_t dt.
	\]
	Define, for some $\delta > 0$ and absolutely continuous and increasing $\omega: [0,T] \to \RR_+$ with $\omega(0) = 1$,
	\[
		\tau := \inf \left\{ t \in [0,T] : Z^i_t \Delta^i_t - \delta e^{\omega(t)} > 0 \text{ for some } i = 1,2,\ldots, d \right\}.
	\]
	We have $\tau > 0$ because $Z^i_0 \Delta^i_0 \le 0$. Assume for the sake of contradiction that $\tau < \oo$. If $i = 1,2,\ldots, m$ is such that $Z^i_\tau \Delta^i_\tau = \delta e^{\omega(\tau)}$, then
	\[
		\delta e^{\omega(\tau)} = Z^i_\tau \Delta^i_\tau
		\le \sum_{j \ne i} \int_0^\tau \mu^{ij}_t Z^i_t \Delta^j_t dt
		\le \delta \sum_{j \ne i} \int_0^\tau \mu^{ij}_t \frac{Z^i_t}{Z^j_t} e^{\omega(t)}dt.
	\]
	We have
	\[
		\frac{Z^i_t}{Z^j_t} = \prod_{k=1}^m \exp \left(\frac{1}{2} \int_0^t \left[ (\nu^{ik}_s)^2 - (\nu^{jk}_s)^2 \right] ds + \int_0^t (\nu^{jk}_s - \nu^{ik}_s)dW^k_s \right),
	\]
	and so we obtain the contradiction $\delta < \frac{\delta}{2}$ for any $\delta > 0$ if $\omega$ is to chosen to satisfy
	\[
		\omega(0) = 1, \quad \omega'(t) = 2\max_{i=1,2,\ldots, d} \sum_{j\ne i} \mu^{ij}_t \frac{Z^i_t}{Z^j_t} .
	\]
	This implies $\Delta^i_t \le \delta e^{\omega(t)}(Z^i_t)^{-1}$ for all $i = 1,2,\ldots, m$, $t \in [0,T]$, and $\delta > 0$, and we conclude upon sending $\delta \to 0$.
\end{proof}

\begin{proof}[Proof of Theorem \ref{T:SDE}]
	Let $b^\eps$ and $b_\eps$ be the $\sup$- and $\inf$-convolutions of $b$ as in the proof of Proposition \ref{P:maxmin}, and define the stochastic flows $\Phi^{+,\eps}$ and $\Phi^{-,\eps}$ by
	\[
		d_t\Phi^{+,\eps}_{t,s}(x) = b^\eps(t, \Phi^{+,\eps}_{t,s}(x))dt + \sigma(t, \Phi^{+,\eps}_{t,s}(x))dW_t \quad \text{in } [s,T], \quad \Phi^{+,\eps}_{s,s} = \Id
	\]
	and
	\[
		d_t\Phi^{-,\eps}_{t,s}(x) = b_\eps(t, \Phi^{-,\eps}_{t,s}(x))dt + \sigma(t, \Phi^{-,\eps}_{t,s}(x))dW_t
		\quad \text{in } [s,T], \quad \Phi^{-,\eps}_{s,s} = \Id.
	\]
	Parts \eqref{SDE:maxmin}, \eqref{SDE:ordering}, and \eqref{SDE:flow} are then established exactly as in the proof of Proposition \ref{P:maxmin}, making use of the comparison result of Lemma \ref{L:SDEcomparison} above.
	
	Part \eqref{SDE:unique}, and then, as a consequence, \eqref{SDE:solution}, is proved similarly as in the proof of Proposition \ref{P:shiftover} and Theorem \ref{T:regLag}, using now instead the one-sided mollifiers $\rho^\eps$ and $\rho_\eps$ to regularize $b$. Note that, in view of the assumption \eqref{A:bosc}, for some $C \in L^1_+([0,T])$,
	\[
		|\nabla (b(t,\cdot) * \rho^\eps)(x)| \le  \int |b(t, x) - b(t,y)| |\nabla \rho^\eps(x-y)|dy \le \frac{C(t)}{\eps},
	\]
	and similarly for $b * \rho_\eps$. This allows for the use of the comparison Lemma \ref{L:SDEcomparison} in the argument, which requires global Lipschitz regularity in view of the lack of finite speed of propagation.
	
	Finally, \eqref{SDE:regLag} follows exactly as in the proof of Theorem \ref{T:regLag}.
\end{proof}

\section{Linear transport equations with increasing drift}\label{sec:TE}

For velocity fields $b$ satisfying \eqref{A:b}, we discuss the terminal value problem for the nonconservative equation
\begin{equation}\label{TE}
	-\del_t u - b(t,x) \cdot \nabla u = 0 \quad \text{in }(0,T) \times \RR^d, \quad u(T,\cdot) = u_T.
\end{equation}
as well the initial value problem for the conservative equation
\begin{equation}\label{CE}
	\del_t f + \div (b(t,x) f) = 0 \quad \text{in } (0,T) \times \RR^d, \quad f(0,\cdot) = f_0.
\end{equation}

The lower bound on $\div b$ implied by \eqref{A:b}, which gave rise to the regular Lagrangian property for the unique flow $\phi$ in the previous section, here allows for a theory of weak solutions of \eqref{TE} and \eqref{CE} in $L^p$-spaces, due to the expansive property of the flow.

\subsection{The nonconservative equation}\label{subsec:TE}

It is important to note that solutions $u$ of \eqref{TE} taking values in $L^p_\loc$ cannot be understood in the sense of distributions. This is because, under the assumption \eqref{A:b}, $\div b$ is a measure that need not necessarily be absolutely continuous with respect to Lebesgue measure. Instead, we identify unique solution operators for \eqref{TE} that are continuous on Lebesgue spaces and stable under regularizations. This is done through the relationship with the flow in the previous section, and also by characterizing the solutions using ``one-sided'' regularizations for increasing or decreasing solutions.

\subsubsection{Representation formula}

When $b$ and $u_T$ are smooth, the unique solution of the terminal value problem for \eqref{TE} with terminal value $\oline u$ is
\begin{equation}\label{TE:semigroup}
	u(s,x) = S(s,t)\oline{u}(x) := \oline{u}( \phi_{t,s}(x)), \quad s \in [0,t],
\end{equation}
where $\phi$ is the flow corresponding to the ODE \eqref{ODE}. In view of Theorem \ref{T:regLag}, this formula extends to $\oline{u} \in L^p_\loc$, $1 \le p \le \oo$ if the assumption on $b$ is relaxed to \eqref{A:b}. We thus identify a family of solution operators for \eqref{TE} that are continuous on $L^p$ and evolve continuously in time.

\begin{theorem}\label{T:TEform}
	Assume that $b$ satisfies \eqref{A:b}, and define the solution operators in \eqref{TE:semigroup} using the flow constructed in Section \ref{sec:ODE}. Then the following hold:
	
	\begin{enumerate}[(a)]
	\item\label{formula:regLag} For all $0 \le s \le t \le T$, $S(t,s)$ is continuously on $L^p_\loc(\RR^d)$ for $1 \le p \le \oo$, and there exists $C > 0$ depending only on $T$ such that, for any $R > 0$,
	\[
		\norm{S(s,t)\oline{u}}_{L^p(B_R)} \le e^{d(\omega_1(t) - \omega_1(s))} \norm{\oline{u}}_{L^p(B_{R + C(1 +R)(\omega_0(t) - \omega_0(s)) } }.
	\]
	\item\label{formula:semigroup} For all $0 \le r \le s \le t \le T$, we have $S(t,s) \circ S(s,r) = S(t,r)$. 
	\item\label{formula:cts} If $0 < t \le T$, then $S(s,t) \oline{u} \in C([0,t], L^p_\loc(\RR^d))$ if $\oline{u} \in L^p_\loc$ for $p < \oo$, and $S(s,t) \oline{u} \in C([0,t], L^\oo_{\loc,\ws}(\RR^d))$ if $\oline{u} \in L^\oo_\loc$.
	\item\label{formula:conv} If $(b^\eps)_{\eps > 0}$ is any family of smooth approximations of $b$ satisfying \eqref{A:b} uniformly in $\eps$, $(u_T^\eps)_{\eps > 0}$ are smooth functions such that $u_T^\eps \xrightarrow{\eps \to 0} u_T$ in $L^p_\loc$ for some $p < \oo$, and $u^\eps$ is the corresponding solution of \eqref{TE}, then, as $\eps \to 0$, $u^\eps$ converges to $S(t,T)u_T$ strongly in $C([0,T], L^p_\loc(\RR^d))$. The same statement is true if $u^\eps$ solves
	\begin{equation}\label{TE:visc}
		-\del_t u^\eps - b(t,x) \cdot \nabla u^\eps = \eps \Delta u^\eps \quad \text{in } (0,T) \times \RR^d, \quad u^\eps(T,\cdot) = u^\eps_T.
	\end{equation}
	\end{enumerate}
\end{theorem}

\begin{proof}
	Properties \eqref{formula:regLag} and \eqref{formula:semigroup} follow immediately from Theorem \ref{T:regLag}, while properties \eqref{formula:cts} and \eqref{formula:conv} follow from Theorem \ref{T:ODE:beps} and the dominated convergence theorem (and see also Theorem \ref{T:-b}). For the statement involving the viscous equation \eqref{TE:visc}, we remark that, in that case, $u^\eps$ is given by $u^\eps(t,x) = u^\eps_T(\phi^\eps_{T,t}(x))$, where $\phi^\eps$ is the stochastic flow corresponding to \eqref{SDE:eps}. The proof is then finished because of Theorem \ref{T:ODE:epsW}.
\end{proof}

\begin{remark}\label{R:strongcty}
	The uniqueness of the semigroup is a consequence of the uniqueness of the flow established in the previous section. Note, however, that, solely under the assumption $\div b \ge -C_1$ for some $C_1 \in L^1_+([0,T])$, any weakly-limiting family of solution operators must lead to solutions in $C([0,T],L^p_\loc(\RR^d))$, in the strong topology.
	
	More precisely, for $\eps > 0$, we have the easy a priori bound
	\[
		\norm{ S^\eps(s,t)\oline{u}}_{L^p} \le \exp \left( \int_s^t C_1(r)dr\right) \norm{\oline{u}}_{L^p}, \quad \oline{u} \in L^p.
	\]
	It follows from a diagonal argument that there exists $\eps_n \to 0$ and a family of continuous linear operators on $L^p$ such that $S(r,s)S(s,t) = S(r,t)$ for $r \le s \le t$, and, for all $t \in [0,T]$,
	\[
		S^{\eps_n}(\cdot,t) \oline{u} \rightharpoonup S(\cdot,t) \oline{u} \quad \text{weakly in } L^\oo([0,T], L^p(\RR^d)) \text{ for all } \oline{u} \in L^p.
	\]
	In particular, $u := S(\cdot,t) \oline{u} \in C([0,T], L^p_\w(\RR^d))$, and so, for any $s \in [0,t]$,
	\[
		\liminf_{h \to 0^+} \norm{u(s+h,\cdot)}_{L^p} \ge \norm{u(s,\cdot)}.
	\]
	On the other hand,
	\[
		\norm{u(s+h)}_{L^p} \le \exp \left( \int_s^{s+h} C_1(r)dr \right) \norm{u(s)}_{L^p},
	\]
	so that
	\[
		\limsup_{h \to 0^+} \norm{u(s+h,\cdot)}_{L^p} \le \norm{u(s,\cdot)}.
	\]
	We find then that $\norm{u(s+h, \cdot)}_{L^p} \to \norm{u(s,\cdot)}_{L^p}$, which, coupled with the weak convergence, means that $u(s+h, \cdot) \to u(s,\cdot)$ strongly in $L^p$ if $p > 1$ (and so for all $p$ locally). A similar argument holds with $h$ replaced by $-h$.
\end{remark}

The following renormalization property for the solution operator $S(s,t)$, $0 \le s \le t \le T$, is then immediate from the formula.

\begin{proposition}
	If $\beta: \RR \to \RR$ is smooth and $|\beta(s)| \le C(1 + |s|^r)$ for some $r \ge 1$, then, for all $\oline{u} \in L^r_\loc$, $\beta( S(s,t) \oline{u}) = S(s,t) (\beta \circ \oline{u})$.
\end{proposition}

\subsubsection{Characterizing increasing/decreasing solutions}
The solution of the transport equation in the previous subsection was characterized as the unique limit under arbitrary regularizations, as well as through the formula involving the regular Lagrangian flow. The remainder of the section is dedicated to understanding further ways to characterize the solution, and in particular on the level of the equation \eqref{TE} itself. This will become useful in the study of nonlinear equations in the final section.

We first observe that, if $u_T$ is increasing/decreasing with respect to the partial order \eqref{order}, then so is $u(t,\cdot)$ for all $t \in [0,T]$. While this is immediately clear from the formula $u(t,\cdot) = u_T \circ \phi_{T,t}$ and the fact that $\phi_{T,t}$ is increasing, it can also be seen directly from the equation. Indeed, if $u$ is a smooth solution of \eqref{TE} and $v_i = \del_{x_i} u$, $i = 1,2,\ldots, d$, then
\[
	-\del_t v_i - b \cdot \nabla v_i - (\del_{x_i} b^i ) v_i  =  \sum_{j \ne i}( \del_{x_i} b^j) v_j.
\]
Therefore, if $v_i \ge 0$ (or $v_i \le 0)$ when $t = T$ for all $i = 1,2,\ldots, d$, then the same is true for $t < T$ by the maximum principle Lemma \ref{L:particles}. The result for general $b$ satisfying \eqref{A:b} follows from approximating $b$ and using the limiting result in Theorem \ref{T:TEform}. By linearity, we have thus established the following:

\begin{proposition}\label{P:ABVprop}
	For all $0 \le s \le t \le T$, $S(s,t): ABV \to ABV$.
\end{proposition}

In particular, since $ABV$ is densely contained in $L^p_\loc(\RR^d)$ and $S(s,t)$ is continuous on $L^p_\loc$, Proposition \ref{P:ABVprop} implies that belonging to $ABV$ is a suitable criterion for the propagation of compactness. Notice for instance that this provides another proof of the fact that the convergence statements in Theorem \ref{T:TEform} are with respect to \emph{strong} convergence in $C([0,T], L^p_\loc)$ for $p < \oo$.

We now demonstrate how the propagation of the increasing or decreasing property leads to a method for characterizing solutions of \eqref{TE}, independently of the solution formula. The idea is to regularize $u$ in a one-sided manner, as in subsection \ref{ss:regs}.

Formally, if $u$ solves \eqref{TE} and, for $(t,x) \in [0,T] \times \RR^d$ (recall that $\omega_1$ is as in \eqref{acomega}),
\[
	\tilde u(t,x) = u\left( t, e^{\omega_1(T) - \omega_1(t)}x \right) \quad \text{and} \quad
	\tilde b(t,x) = e^{-(\omega_1(T) - \omega_1(t))} b\left( t, e^{\omega_1(T) - \omega_1(t)} x \right) + C_1(t)x,
\]
then $\tilde u$ solves \eqref{TE} with $b$ replaced by $\tilde b$, and $\tilde b$ satisfies \eqref{A:b} with $C_1 = 0$ and a possibly different $C_0$. We may therefore assume here, without loss of generality, that $b(t,\cdot)$ is increasing for a.e. $t \in [0,T]$.

Recall the definition of the one-sided mollifiers $\rho_\eps$ and $\rho^\eps$ defined in \eqref{onesidedmollifier}.

\begin{definition}\label{D:TEsubsuper}	
	A function $u: [0,T] \times \RR^d \to \RR$ is called a super (sub)solution of \eqref{TE} if, for all $t \in [0,T]$, $u(t,\cdot)$ is decreasing, and, if $u_\eps = u * \rho_\eps$ and $u^\eps = u * \rho^\eps$, then
	\begin{equation}\label{TE:ineq}
		-\del_t u_\eps - b(t,x) \cdot \nabla u_\eps \ge 0 \quad \left( \text{resp.} \quad -\del_t u^\eps - b(t,x) \cdot \nabla u^\eps \le 0  \right)\quad \text{for a.e. } (t,x) \in [0,T] \times \RR^d.
	\end{equation}
	A solution is both a sub- and supersolution.
\end{definition}

\begin{remark}
	A well-posed notion of sub- and supersolutions can be defined where $u$ is approximated using, instead of the one-sided mollifiers, the method of $\inf$- and $\sup$-convolution:
	\[
		u_\eps(t,x) = \inf_{y} \left\{ u(x-y) + \frac{|y|}{\eps} \right\} \quad \text{and} \quad \sup_{y} \left\{ u(x-y) - \frac{|y|}{\eps} \right\}.
	\]
\end{remark}

\begin{theorem}\label{T:TE:comparison}
	Assume $b$ satisfies \eqref{A:b} with $C_1 = 0$. Then, for all $R > 0$, there exists a modulus of continuity $\omega_R: [0,\oo) \to [0,\oo)$ such that, if $u,v:[0,T] \times \RR^d \to \RR$ are respectively a sub- and supersolution of \eqref{TE} in the sense of Definition \ref{D:TEsubsuper}, then, for all $t \in [0,T]$ and $p \in [1,\oo)$,
	\begin{equation}\label{TE:comparison}
		\int_{B_R} \left( u(t,x) - v(t,x)\right)_+^p dx \le \int_{B_{R + \omega_R(T-t)}} \left( u(T,x) - v(T,x) \right)_+^pdx.
	\end{equation}
	In particular, for any decreasing $u_T: \RR^d \to \RR^d$, there exists a unique solution $u$ of \eqref{TE} in the sense of Definition \ref{D:TEsubsuper}, which is given by $u(t,\cdot) = S(t,T) u_T$, and which is continuous a.e. in $[0,T] \times \RR^d$.
\end{theorem}

\begin{proof}
	Set $u^\eps = u * \rho^\eps$ and $v_\eps = v * \rho_\eps$; then, combining the inequalities for $u^\eps$ and $v_\eps$ given by Definition \ref{D:TEsubsuper}, we obtain
	\[
		- \del_t (u^\eps - v_\eps) - b(t,x) \cdot \nabla (u^\eps - v_\eps) \le 0.
	\]
	Let $\beta: \RR \to \RR$ be smooth and increasing. Multiplying the above inequality by the positive term $\beta'( u^\eps - v_\eps)$ yields
	\[
		-\del_t \beta(u^\eps - v_\eps) - b(t,x) \cdot \nabla \beta(u^\eps - v_\eps) \le 0.
	\]
	We may then take $\beta(r) = r_+^p$ (arguing with an extra layer of regularizations if $p = 1$) and find that, in the sense of distributions,
	\[
		-\del_t (u^\eps - v_\eps)_+^p - b(t,x) \cdot \nabla (u^\eps - v_\eps)_+^p \le 0.
	\]
	Let $\hat \psi: [0,\oo) \to [0,\oo)$ be smooth and decreasing, such that $\hat \psi = 1$ on $[0,1]$ and $\hat \psi = 0$ on $[2,\oo)$, and set $\psi(t,x) := \hat \psi(|x|/R(t))$ for some $R: [0,T] \to [0,\oo)$ to be determined. Using $\psi$ as a test function, we discover
	\begin{align*}
		- \del_t \int_{\RR^d} (u^\eps(t,x) - v_\eps(t,x))_+^p \psi(t,x)dx
		&\le \left\langle -\del_t \psi(t,\cdot) - \div (b(t,\cdot) \psi(t,\cdot)) , (u^\eps(t,\cdot) - v_\eps(t,\cdot))_+^p  \right\rangle\\
		&\le \int_{\RR^d} \left( u^\eps(t,x) - v_\eps(t,x) \right)_+^p \left( -\del_t \psi(t,x) - b(t,x) \cdot \nabla \psi(t,x) \right) dx,	
	\end{align*}
	where in the last line we used the fact that $\div b \ge 0$ and $\psi \ge 0$. Using the fact that $\hat \psi' \le 0$, with $\hat \psi' \ne 0$ only if $r \in [1,2]$, we find that
	\[
		- \del_t \psi - b \cdot \nabla \psi
		= - \hat \psi'(x/R(t)) \left( \frac{\dot R(t)}{R(t)^2}- \frac{b(t,x) \cdot x}{R(t)|x|^2} \right)
		\ge - R(t)^{-1}  \hat \psi'(x/R(t)) \left( \dot R(t) - 2C_0(t)( 1 + 2 R(t) ) \right).
	\]
	For $t_0 \in [0,T]$, this is made nonnegative on $[t_0,T]$ by choosing, for any fixed $R > 0$,
	\[
		R(t) = R e^{4[\omega_0(t) - \omega_0(t_0) ] } + \frac{1}{2} \left( e^{4[\omega_0(t) - \omega_0(t_0) ] } - 1 \right),
	\]
	and so
	\[
		\int_{\RR^d} ( u^\eps(t_0,x) - v_\eps(t_0,x) )_+^p \hat \psi\left( \frac{|x|}{R(t_0)} \right) dx
		\le \int_{\RR^d} (u^\eps(T,x) - v_\eps(T,x) )_+^p \hat \psi \left( \frac{|x|}{R(T)} \right)dx.
	\]
	We may then choose functions $\hat \psi$ that approximate $\ind_{[0,1]}$ from above, and then, by the monotone convergence theorem,
	\[
		\int_{B_R} (u^\eps(t_0, x) - v_\eps(t_0,x))_+^p dx
		\le \int_{B_{R(T)}} (u^\eps(T,x) - v_\eps(T,x))_+^p dx.
	\]
	The proof of \eqref{TE:comparison} is finished upon sending $\eps \to 0$ and setting
	\[
		\omega_R(r) := \sup_{0 \le s \le t \le s + r \le T} \left\{ \left( R + \frac{1}{2} \right) \left( e^{4[\omega_0(t) - \omega_0(s)] } - 1 \right) \right\}.
	\]
	
	The comparison inequality \eqref{TE:comparison} implies uniqueness for a solution with terminal condition $u_T$, and so it remains to show that $S(t,T) u_T$ is a solution in the sense of Definition \ref{D:TEsubsuper}.
	
	Let $(b^\delta)_{\delta > 0}$ be a family of smooth approximations of $b$ satisfying \eqref{A:b} with $C_1 = 0$ and uniform $C_0$, and let $u^\delta_T = u_T* \gamma_\delta$ for a family of standard mollifiers $(\gamma_\delta)_{\delta > 0}$. Note then that, for $\delta > 0$, $u^\delta_T$ is decreasing, and, as $\delta \to 0$, $u^\delta_T \to u_T$ in $L^p_\loc$ for all $p < \oo$. Let $u^\delta$ be the corresponding solution of the terminal value problem \eqref{TE}. It follows that $u^\delta(t,\cdot)$ is decreasing for all $t \in [0,T]$.
	
	For any $\delta > 0$ and $\eps > 0$,
	\[
		-\del_t (u^\delta * \rho^\eps) - b^\delta(t,x) \cdot \nabla (u^\delta * \rho^\eps)
		= \int_{\RR^d} \left( b^\delta(t,y) - b^\delta(t,x) \right) \cdot \nabla u^\delta(t,y) \rho^\eps(x-y)dy
		\le 0,
	\]
	where the last inequality follows from the fact that $\rho^\eps(x-y) \ne 0$ only if $x \le y$, $b^\delta(t,\cdot)$ is increasing, and $\nabla u^\delta \le 0$. Sending $\delta \to 0$, it follows from Theorem \ref{T:TEform} that, in the sense of distributions, if $u = S(\cdot,T)u_T$ and $u^\eps = u * \rho^\eps$,
	\[
		-\del_t u^\eps - b \cdot \nabla u^\eps \le 0.
	\]
	It follows that $S(\cdot,T) u_T$ is a subsolution. A similar argument shows that it is a supersolution, and therefore the unique solution in the sense of Definition \ref{D:TEsubsuper}.
	
	Now, for some $M > 0$, define $\psi(t) = \exp(M \omega_1(t))$ and
	\[
		\tilde u^\delta(t,x) := u^\delta(t,x + \psi(t)\mbf 1).
	\]
	For any fixed $R > 0$, there exists $M > 0$ such that, in view of the linear growth of $b^\delta$ given by \eqref{A:b} uniformly in $\delta$, and the fact that $u^\delta$ is decreasing in the spatial variable, for any $x \in [-R,R]^d$,
	\[
		\frac{\del}{\del t} \tilde u^\delta(t,x) = (-\nabla u^\delta) \cdot \left( b^\delta(t,x + \psi(t) \mbf 1) - \psi'(t) \mbf 1 \right)
		\le C_1(t)(-\nabla u^\delta) \cdot \left( (1 + |x| + \psi(t)) - M \psi(t) \right) \le 0.
	\]
	Since $M$ is independent of $\delta$, we may send $\delta \to 0$ and conclude that $\tilde u(t,x) := u(t,x + \psi(t))$ is increasing on $[0,T] \times [-R,R]^d$, and therefore continuous a.e. in that set by Lemma \ref{L:ABVcts}. The transformation leading from $u$ to $\tilde u$ preserves null sets, and we conclude that $u$ is continuous almost everywhere.
\end{proof}

\begin{remark}
	The idea behind Definition \ref{D:TEsubsuper} is to establish a sign for the commutator between convolution and differentiation along irregular vector fields, as compared to the work of DiPerna and the first author \cite{DL89}, where the commutator is shown to be small for Sobolev vector fields. We must take convolution kernels with a specific one-sided structure in order to analyze the commutators; a less crude example of this idea is seen in the work of Ambrosio \cite{A04} for general $BV$ velocity fields.
		
	A different notion of sub and supersolutions, which also selects $S(\cdot,T) u_T$ as the unique solution of \eqref{TE}, can be obtained by instead regularizing $b$. Recall that we have assumed without loss of generality that $b$ is increasing. Then a decreasing function $u$ can be said to be a sub (super)solution if, for all $\eps > 0$, in the sense of distributions,
	\[
		-\del_t u - b_\eps \cdot \nabla u \le 0 \quad \left( \text{resp. } -\del_t u - b^\eps \cdot \nabla u \ge 0 \right),
	\]
	where $b_\eps \le b \le b^\eps$ are one-sided regularizations of $b$, for example the one-sided mollifiers or the $\inf$- and $\sup$-convolutions. The notion of sub and supersolution in Definition \ref{D:TEsubsuper} turns out to be more amenable to the study of the nonlinear systems in Section \ref{sec:NTE}.
\end{remark}

\begin{remark}
	Throughout this section, we have studied the setting where $u_T$, and therefore $u(t,\cdot)$ for $t < T$, are decreasing. The same analysis can be achieved for increasing solutions, in which case the inequalities in  \eqref{TE:ineq} are reversed. Note, in particular, that the solution flows $\phi_{t,s}(x)$ are vector-valued solutions in this sense; i.e.
	\[
		- \del_s \phi_{t,s}(x) - b(s,x) \cdot \nabla \phi_{t,s}(x) = 0, \quad \phi_{t,t}(x) = x.
	\] 
	We now observe that the full family of solution operators $S(s,t) : L^p_\loc \to L^p_\loc$, $0 \le s \le t \le T$, can be constructed independently of the ODE flows $\phi_{s,t}$, which can then be recovered with the theory of renormalization. Indeed, Definition \ref{D:TEsubsuper}, and its counterpart for increasing solutions, can be used to define $S(s,t) \bar u$ for any $\bar u \in ABV$. The density of $ABV$ in $L^p_\loc(\RR^d)$, and the $L^p$-continuity of $S(s,t)$, then allow the solution operators to be continuously extended to $L^p_\loc$. It is, however, not clear whether solutions of \eqref{TE} can be characterized for arbitrary $u_T \in L^p_\loc$, other than by the formula \eqref{TE:semigroup} or as limits of solutions to regularized equations.
\end{remark}

\subsubsection{Lower order terms}

We briefly explain how to extend the above results to equations with additional lower order terms, as in
\begin{equation}\label{eq:lowerorder}
	-\del_t u - b(t,x) \cdot \nabla u - c(t,x)u - d(t,x) = 0 \quad \text{in } [0,T] \times \RR^d , \quad u(T,\cdot) = u_T,
\end{equation}
for functions 
\begin{equation}\label{A:cd}
	c \in L^1([0,T], L^\oo_\loc) \quad \text{and} \quad d \in L^1([0,T], L^q_\loc) \quad \text{for some } q \in [1,\oo].
\end{equation}

\begin{theorem}\label{T:TE:lowerorder}
	Assume $p \in [1,\oo) \cap [1,q]$ and $u_T \in L^p_\loc$. Then there exists a unique function $u \in C([0,T], L^p_\loc)$ with the following properties:
	\begin{enumerate}[(a)]
	\item\label{lowerorder:Lp} There exists $C_2 \in L^1_+([0,T])$, and, for $R > 0$, a modulus $\omega_R: [0,T] \to [0,T]$, depending only on the assumptions in \eqref{A:b} and \eqref{A:cd} such that, for $R > 0$,
	\[
		\norm{u(t,\cdot)}_{L^p(B_R)} \le \exp\left( \int_t^T C_2(s)ds \right) \norm{u_T}_{L^p_{B_{R + \omega_R(T-t)}}} + \int_t^T C_2(s)ds.
	\]
	\item\label{lowerorder:smooth} Let $(b^\eps)_{\eps > 0}$, $(c^\eps)_{\eps > 0}$, and $(d^\eps)_{\eps > 0}$ be families satisfying \eqref{A:b} and \eqref{A:cd} uniformly in $\eps$, such that, as $\eps \to 0$, $(b^\eps,c^\eps, d^\eps) \to (b,c,d)$ a.e. Let $(u^\eps_T)_{\eps > 0}$ be a family of smooth functions approximating $u_T$ in $L^p_\loc$, and let $u^\eps$ be the corresponding solution of \eqref{eq:lowerorder}. Then, as $\eps \to 0$, $u^\eps$ converges strongly to $u$ in $C([0,T], L^p_\loc)$. The same statement is true if $u^\eps$ solves
	\begin{equation}\label{eq:visc:lowerorder}
		-\del_t u^\eps - b(t,x) \cdot \nabla u^\eps - c(t,x)u^\eps - d(t,x) = \eps \Delta u^\eps \quad \text{in } [0,T] \times \RR^d, \quad u^\eps(T,\cdot) = u^\eps_T.
	\end{equation}
	\item\label{lowerorder:formula} For $(t,x) \in [0,T] \times \RR^d$, $u$ has the formula
	\begin{equation}\label{formula:lowerorder}
		u(t,x) = u_T(\phi_{T,t}(x)) \exp\left( \int_t^T c(s, \phi_{s,t}(x))ds \right) + \int_t^T d(s, \phi_{s,t}(x))\exp\left( \int_t^s c(r, \phi_{r,t}(x))dr \right)ds.
	\end{equation}
	\end{enumerate}
	Analogous statements hold when $p = q = \oo$, in which case $u \in C([0,T], L^\oo_{\loc,\ws})$ and the convergence in part \eqref{lowerorder:smooth} is weak-$\star$ in $L^\oo$.
\end{theorem}

\begin{proof}
	For $b^\eps$, $c^\eps$, $d^\eps$, $u^\eps_T$, and $u^\eps$ as in the statement of part \eqref{lowerorder:smooth}, if $(\phi^\eps_{t,s}(x))_{s,t \in [0,T]}$ denotes the corresponding smooth flow, we have the formula
	\[
		u^\eps(t,x) = u^\eps_T(\phi^\eps_{T,t}(x)) \exp\left( \int_t^T c^\eps(s, \phi^\eps_{s,t}(x))ds \right) + \int_t^T d^\eps(s, \phi^\eps_{s,t}(x))\exp\left( \int_t^s c^\eps(r, \phi^\eps_{r,t}(x))dr \right)ds.
	\]
	The convergence statements, and thus the formula in part \eqref{lowerorder:formula}, are then proved just as in Theorem \ref{T:ODE:beps}. In particular, arguing just as in that proof, we have	
	\[
		\lim_{\eps \to 0}\left( \norm{u_T^\eps(\cdot,\phi^\eps_{\cdot,t}) - u_T(\cdot, \phi_{\cdot,t}) }_{L^1([t,T], L^p_\loc)}  + \norm{d^\eps(\cdot,\phi^\eps_{\cdot,t}) - d(\cdot, \phi_{\cdot,t}) }_{L^1([t,T], L^p_\loc)}  \right) = 0
	\]
	and, for all $r < \oo$,
	\begin{align*}
		\lim_{\eps \to 0}  \norm{c^\eps(\cdot,\phi^\eps_{\cdot,t}) - c(\cdot, \phi_{\cdot,t}) }_{L^1([t,T], L^r_\loc)} =0.
	\end{align*}
	In addition, $c^\eps(t, \phi^\eps_{t,s})$ is uniformly bounded in $L^1([0,T], L^\oo_\loc)$, in view of Lemma \ref{ODE:apriori}, and so we conclude parts \eqref{lowerorder:smooth} and \eqref{lowerorder:formula} by the dominated convergence theorem. The $L^p$-estimates in part \eqref{lowerorder:Lp} are proved just as before, either from the regularized equation itself or from \eqref{formula:lowerorder}, using the lower bound on the divergence of $b$.
\end{proof}

Sub and supersolutions can be characterized when $u_T$ is increasing/decreasing under the additional assumption that
\begin{equation}\label{A:cdextra}
	c = c(t) \in L^1_+([0,T]) \quad \text{and} \quad d(t,\cdot) \text{ is decreasing for all } t \in [0,T].
\end{equation}

The propagation result Proposition \ref{P:ABVprop} is then easily generalized (the proof is almost identical and so we omit it):

\begin{proposition}\label{P:ABVprop:lowerorder}
	Assume \eqref{A:b}, \eqref{A:cd}, and \eqref{A:cdextra}. If $u_T \in L^\oo_\loc$ is increasing (decreasing) and $u$ is the solution of \eqref{eq:lowerorder} specified by Theorem \ref{T:TE:lowerorder}, then, for all $t < T$, $u(t,\cdot)$ is increasing (decreasing).
\end{proposition}

Definition \ref{D:TEsubsuper} is then generalized as follows:

\begin{definition}\label{D:TEsubsuper:lowerorder}
	A function $u: [0,T] \times \RR^d \to \RR$ is called a super (sub)solution of \eqref{eq:lowerorder} if, for all $t \in [0,T]$, $u(t,\cdot)$ is decreasing, and, if $u_\eps = u * \rho_\eps$ and $u^\eps = u * \rho^\eps$, then
	\begin{equation}\label{TE:ineq}
		\begin{split}
		&-\del_t u_\eps - b(t,x) \cdot \nabla u_\eps - c(t) u_\eps - d(t,x) \ge 0 \\
		& \left( \text{resp.} \quad -\del_t u^\eps - b(t,x) \cdot \nabla u^\eps - c(t) u^\eps - d(t,x) \le 0  \right)\quad \text{for a.e. } (t,x) \in [0,T] \times \RR^d.
		\end{split}
	\end{equation}
	A solution is both a sub- and supersolution.
\end{definition}

Finally, the following is proved almost identically to Theorem \ref{T:TE:comparison}.

\begin{theorem}\label{T:TE:comparison:lowerorder}
Assume $b$ satisfies \eqref{A:b} with $C_1 = 0$, and $c$ and $d$ satisfy \eqref{A:cd} and \eqref{A:cdextra}. Then, for all $R > 0$, there exist moduli of continuity $\omega,\omega_R: [0,\oo) \to [0,\oo)$ such that, if $u,v:[0,T] \times \RR^d \to \RR$ are respectively a sub- and supersolution of \eqref{eq:lowerorder} in the sense of Definition \ref{D:TEsubsuper:lowerorder}, then, for all $t \in [0,T]$ and $p \in [1,\oo)$,
	\begin{equation}\label{TE:comparison:lowerorder}
		\int_{B_R} \left( u(t,x) - v(t,x)\right)_+^p dx \le e^{\omega(T-t)} \int_{B_{R + \omega_R(T-t)}} \left( u(T,x) - v(T,x) \right)_+^pdx.
	\end{equation}
	In particular, for any decreasing $u_T \in L^p_\loc(\RR^d)$, there exists a unique solution $u$ of \eqref{eq:lowerorder} in the sense of Definition \ref{D:TEsubsuper:lowerorder}, which is given by \eqref{formula:lowerorder}, and which is continuous a.e. in $[0,T] \times \RR^d$.
\end{theorem}

\subsection{The conservative equation}

In contrast to the theory for the nonconservative equation, solutions to \eqref{CE} belonging to Lebesgue spaces can be made sense of in the sense of distributions. However, under the general assumption \eqref{A:b}, these are not in general unique, as the simple example $b(x) = \sgn x$ on $\RR$ shows. Drawing once more an analogy with the setting studied in \cite{BJ,BJM,LSother} of half-Lipschitz velocity fields, the ``good'' (stable) solution of \eqref{CE} is identified using a particular solution formula, and, in \cite{LSother}, this is shown to coincide with the pushforward by the regular Lagrangian flow of the initial density. As discussed in subsection \ref{ss:reverse?} above, the former strategy is unavailable; however, in view of the theory of the nonconservative equation and the forward regular Lagrangian flow that was built in the previous section, we may define solutions by duality.

\subsubsection{Duality solutions}
For $0 \le s \le t \le T$< denote by $S^*(t,s)$ the adjoint of the solution operator $S(s,t)$.

\begin{theorem}\label{T:CE}
	Let $1 \le p \le \oo$ and $f_0 \in L^p_\loc$, and define $f(t,x) = S^*(t,0) f_0$. Then $f \in C([0,T], L^p_{\loc}(\RR^d))$ if $p < \oo$ or $f \in C([0,T], L^\oo_{\loc,\ws})$ if $p = \oo$, and $f$ is a distributional solution of \eqref{CE} and 
	\begin{equation}\label{pushforward}
		S^*(t,0)f_0 = (\phi_{t,0})_\sharp f_0,
	\end{equation}
	where $\phi$ is the flow obtained in Section \ref{sec:ODE}. If $(b^\eps)_{\eps > 0}$ is a family of smooth approximations of $b$ satisfying \eqref{A:b} uniformly in $\eps > 0$, $(f_0^\eps)_{\eps > 0}$ is a family of approximations of $f_0$ in $L^p_{\loc}$ for $1 \le p < \oo$, and $f^\eps$ is the corresponding solution of \eqref{CE}, then, as $\eps \to 0$, $f^\eps \to f$ weakly in $C([0,T], L^p_{\loc}(\RR^d))$. The same is true if $f^\eps$ is taken to be the unique smooth solution of
	\[
		\del_t f^\eps + \div (b f^\eps) = \eps \Delta f^\eps \quad \text{in } [0,T] \times \RR^d, \quad f^\eps(0,\cdot) = f^\eps_0,
	\]
	and analogous convergence statements hold in the weak-$\star$ sense if $p = \oo$.
\end{theorem}

\begin{proof}
	The identity \eqref{pushforward} follows immediately from \eqref{TE:semigroup}; observe that it is well-defined for $f_0 \in L^p_\loc$ in view of the regular Lagrange property \eqref{regLag}.
	
	We now prove the convergence statements, and it suffices to prove the results for $p < \oo$, since $L^\oo_\loc \subset L^p_\loc$ for any $p < \oo$. Let $(b^\eps)_{\eps > 0}$, $(f^\eps_0)_{\eps > 0}$, $(f^\eps)_{\eps > 0}$ be as in the statement. In view of the lower bound on the divergence, it is straightforward to prove a priori $L^p$ bounds. Namely, for all $R > 0$, there exists a modulus of continuity $\omega_R: [0,\oo) \to [0,\oo)$ depending only on $T$ and $C_0$ from \eqref{A:b}, such that, for all $\eps > 0$ and $t \in [0,T]$,
	\begin{equation}\label{CE:apriori}
		\int_{B_R} |f^\eps(t,x)|^p dx \le e^{d(p-1)\omega_1(t)} \int_{B_{R + \omega_R(t)}} |f^\eps_0(x)|^p dx.
	\end{equation}
	It follows that there exists a subsequence $\eps_n \xrightarrow{n \to \oo} 0$ and $f \in L^\oo([0,T], L^p_\loc(\RR^d))$ such that $f^{\eps_n} \to f$ weakly in $L^\oo([0,T], L^p_\loc(\RR^d))$. Sending $n \to \oo$ for $\eps = \eps_n$ in the equation
	\[
		\del_t f^{\eps_n} + \div (b^{\eps_n} f^{\eps_n}) = 0,
	\]
	using the fact that $b^\eps$ converges strongly to $b$ in $L^p_\loc$ for any $p < \oo$, we find that $f$ is a distributional solution of \eqref{CE}, and thus, moreover, $f \in C([0,T], L^p_{\loc,\w}(\RR^d))$. The weak-$\star$ convergence statements when $p = \oo$ are proved analogously.
	
	Fix $\oline u \in C^1_c(\RR^d)$, and $t_0 \in [0,T]$, and let $u^\eps$ denote the solution of \eqref{TE} with terminal value $u^\eps(t_0,\cdot) = \oline{u}$; in view of the results of the previous subsection, $u^\eps$ has compact support in $[0,T] \times \RR^d$ uniformly in $\eps > 0$.
	
	Exploiting the duality of the equations and integrating by parts gives
	\[
		\int_{\RR^d} \oline{u}(x)f^\eps(t_0, x)dx = \int_{\RR^d} u^\eps(0,x) f_0^\eps(x) dx.
	\]
	Sending $n \to \oo$ for $\eps = \eps_n$ gives the identity
	\[
		\int_{\RR^d} \oline{u}(x) f(t_0,x)dx = \int_{\RR^d} u(0,x) f_0(x)dx.
	\]
	It follows that $f(t_0,\cdot) = S(t_0,0)f_0$, and we therefore have the full convergence for all $\eps \to 0$. An identical argument can be used to prove the convergence statement for the vanishing viscosity limit.
	
	Finally, the the fact that $f$ is continuous from $[0,T]$ into $L^p_\loc(\RR^d)$ with the strong topology is a consequence of the continuity of the upper bound in the $L^p$-estimate \eqref{CE:apriori} (see Remark \ref{R:strongcty}).
\end{proof}

\begin{remark}\label{R:Lpstrong}
	It is clear from the proof above that, when $1 < p < \oo$, the initial functions $f^\eps_0$ need only converge weakly in $L^p_\loc$ to $f_0$ as $\eps \to 0$.
\end{remark}

\begin{remark}\label{R:strong?}
	It is not clear whether the convergence of regularizations $f^\eps$ to the duality solution $f$ of \eqref{CE} can be upgraded to strong convergence, except when $d = 1$, in which case \eqref{A:b} coincides with a half-Lipschitz condition on $b$ (see \cite{LSother}).
\end{remark}

\subsubsection{Nonnegative solutions and renormalization}

It is by now standard in the theory of the continuity equation \eqref{CE} that there is uniqueness of nonnegative solutions when the measure $f_0$ is concentrated on sets where the ODE \eqref{ODE} has a unique solution; see \cite{A04, ACetraro}. In view of Theorem \ref{T:regLag}, we then have the following:

\begin{theorem}\label{T:CEunique}
	If $1 \le p \le \oo$, $f_0 \in L^p_\loc(\RR^d)$, and $f \ge 0$, then $f := S^*(\cdot,0) f_0$ is the unique nonnegative distributional solution of \eqref{CE}.
\end{theorem}

We present here an alternative proof using the characterization of $f$ as the duality solution, in order to emphasize again that the theory in this section can be developed independently of the analysis of the ODE \eqref{ODE} in the previous section.

\begin{proof}[Proof of Theorem \ref{T:CEunique}]
	By Theorem \ref{T:CE}, $f  = S^*(\cdot,0)f_0$ is a distributional solution, and its nonnegativity can be seen through an approximation argument, since weak convergence preserves nonnegativity.
	
	Assume now that $f \in C([0,T], L^p_\loc(\RR^d))$ is a nonnegative distributional solution of \eqref{CE}. Fix $t_0 \in [0,T]$ and a decreasing function $\uline{u}: \RR^d \to \RR$, and let $u = S(\cdot,t_0)\oline{u}$. Then $u(t,\cdot)$ is decreasing for $t \in [0,t_0]$, and is a solution of \eqref{TE} in the sense of Definition \ref{D:TEsubsuper}. In particular, $u_\eps = u * \rho_\eps$ and $u^\eps = u * \rho^\eps$ satisfy respectively
	\[
		-\del_t u_\eps - b \cdot \nabla u_\eps \ge 0 \quad \text{and} \quad -\del_t u^\eps - b \cdot \nabla u^\eps \le 0 \quad \text{in } [0,t_0] \times \RR^d.
	\]
	Let $\psi \in C^1_c(\RR^d)$ with $\psi \ge 0$. Using $u_\eps(t,x) \psi(x)$ as a test function for $f$, we find that
	\begin{align*}
		\int_{\RR^d} &f(t_0,x) (\oline{u} * \rho_\eps)(x)\psi(x)dx - \int_{\RR^d} f_0(x) u_\eps(0,x)\psi(x)dx\\
		&= \int_0^{t_0} \int_{\RR^d} f(s,x) \left[ \del_s (u_\eps(s,x) \psi(x)) + b(s,x) \cdot \nabla (u_\eps(s,x) \psi(x) ) \right]dxds\\
		&\le \int_0^{t_0} \int_{\RR^d} f(s,x) u_\eps(s,x) b(s,x) \cdot \nabla \psi(x)dxds,
	\end{align*}
	where we have used the nonnegativity of $\psi$ and $f$. Similarly, 
	\[
		\int_{\RR^d} f(t_0,x) (\oline{u} * \rho^\eps)(x)\psi(x) dx - \int_{\RR^d} f_0(x) u^\eps(0,x)\psi(x)dx
		\ge \int_0^{t_0} \int_{\RR^d} f(s,x) u^\eps(s,x) b(s,x) \cdot \nabla \psi(x)dxds.
	\]
	Sending $\eps \to 0$, we conclude that
	\[
		\int_{\RR^d} f(t_0,x) \oline{u}(x) \psi(x) dx - \int_{\RR^d} f_0(x) S(0,t_0)\oline{u}(x)\psi(x) dx
		= \int_0^{t_0} \int_{\RR^d} f(s,x) u(s,x) b(s,x) \cdot \nabla \psi(x)dx.
	\]
	By linearity, the same is true for all increasing $\oline{u}$ as well, and therefore, by a density argument, for all $\oline{u} \in L^{p'}_\loc(\RR^d)$. In particular, we take $\oline{u} \in C_c(\RR^d)$, in which case $S(\cdot,t_0)\oline{u}$ is supported in $[0,T] \times B_R$ for some $R > 0$, by the finite speed of propagation property. To conclude, we may then take $\psi \in C^1_c(\RR^d)$ such that $\psi \equiv 1$ in $B_R$, and therefore also $\nabla \psi = 0$ in $B_R$.
\end{proof}

\begin{corollary}\label{C:criterion}
	Assume that $f$ and $|f|$ are both distributional solutions of \eqref{CE}. Then $f$ is the unique duality solution of \eqref{CE}; that is, $f(t,\cdot) = S(t,0) f(0,\cdot)$.
\end{corollary}

Corollary \ref{C:criterion} gives a sufficient criterion for a distributional solution to be the correct duality solution. However, we do not know whether this condition is necessary. In other words, it is an open question whether $|S^*(t,s)\oline{f}| = S^*(t,s) |\oline f|$ for any $\oline f \in L^p_\loc$. This renormalization property is equivalent to a kind of injectivity for the forward flow $\phi_{t,s}$, which we describe with the next result.

\begin{proposition}\label{P:renorm?}
	Let $0 \le s \le t \le T$, $1 \le p \le \oo$, and $\oline{f} \in L^p_\loc(\RR^d)$. Then the following statements are equivalent:
	\begin{enumerate}[(a)]
	\item\label{renorm} $|S^*(t,s)\oline{f}| = S^*(t,s) |\oline f|$.
	\item\label{inject} If $A_+ := \left\{ \oline{f} > 0 \right\}$ and $A_- = \left\{ \oline{f} < 0 \right\}$, then
	\[
		\int_{\RR^d} \phi_{t,s}^\sharp \ind_{A_+}(x) \phi_{t,s}^\sharp \ind_{A_-}(x)dx = 0.
	\]
	\end{enumerate}
\end{proposition}

\begin{proof}
	Let $\rho \in L^p_\loc(\RR^d)$ with $\rho \ge 0$, and set $A := \left\{ x : \rho(x) > 0 \right\}$. We first claim that
	\begin{equation}\label{supports}
		\left| \left\{ x : S(t,s)\rho(x) > 0 \right\} \triangle \left\{ x : S(t,s) \ind_A(x) > 0 \right\} \right| = 0.
	\end{equation}
	Let $B \subset \RR^d$ be measurable. Then by Theorem \ref{T:CE},
	\[
		\int_B S^*(t,s) \rho(x) dx = \int \rho(x) \ind\{ \phi_{t,s}(x) \in B \}dx
		\quad \text{and} \quad
		\int_B S^*(t,s) \ind_A(x)dx = |  \left\{x \in A : \phi_{t,s}(x) \in B \right\} |.
	\]
	It follows that $(S^*(t,x) \rho)\ind_B = 0$ a.e. if and only if $(S^*(t,x) \ind_A) \ind_B = 0$ a.e., whence \eqref{supports}.

	It now follows that \eqref{inject} is equivalent to
	\begin{equation}
		\left| \left \{ S^*(t,s) \oline f_+ \ne 0 \right\} \cap \left\{ S^*(t,x) \oline f_- \ne 0 \right\} \right| = 0.
	\end{equation}
	By linearity, 
	\[
		S^*(t,s)\oline f = S^*(t,s)\oline f_+ - S^*(t,x) \oline f_-,
	\]
	and so this is equivalent to $\left(S^*(t,x)\oline f \right)_\pm = S^*(t,s) \oline f_{\pm}$, and thus \eqref{renorm}.
\end{proof}

Either of the two renormalization properties would follow from the strong convergence of regularizations in Theorem \ref{T:CE}. However, we do not know at this time whether the strong convergence actually holds; see Remark \ref{R:strong?}.

The characterization in part \eqref{inject} of Proposition \ref{P:renorm?} is a reformulation of the renormalization property in terms of the injectivity of the flow. For instance, we have the following.

\begin{lemma}\label{L:inject?}
	Assume that $b$ satisfies \eqref{A:b}, let $0 \le s \le t \le T$ and $\oline{f} \in L^p_\loc$, and suppose that
	\[
		\left| \left\{ x \in \RR^d : \phi_{t,s}^{-1}(\{x\}) \cap \{ \oline f > 0 \} \ne \emptyset \text{ and } \phi_{t,s}^{-1}(\{x\}) \cap \{ \oline f < 0 \} \ne \emptyset \right\} \right| = 0.
	\]
	Then the renormalization property in Proposition \ref{P:renorm?} is satisfied.
\end{lemma}

\begin{proof}
	It suffices to establish part \eqref{inject} in Proposition \ref{P:renorm?}. Let $A_\pm = \{ \pm \oline f > 0\}$. Then, for any measurable $B \subset \RR^d$,
	\[
		\int_B \phi_{t,s}^\sharp \ind_{A_\pm}(x)dx = |A_\pm \cap \phi^{-1}_{t,s}(B) |.
	\]
	Let $B$ be any finite-measure set contained in 
	\[
		B_+ := \left\{ y \in \RR^d: \phi^{-1}_{t,s}(\{y\}) \subset A_+^c \right\}.
	\]
	Then
	\[
		\int_B \phi_{t,s}^{\sharp} \ind_{A_+}dx = |A_+ \cap \phi^{-1}_{t,s}(B) | = 0.
	\]
	It follows that $\phi_{t,s}^\sharp \ind_{A_+} = 0$ a.e. in $B_+$. Similarly, $\phi_{t,s}^\sharp \ind_{A_-} = 0$ a.e. in 
	\[
		B_- := \left\{ y \in \RR^d: \phi^{-1}_{t,s}(\{y\}) \subset A_-^c \right\}.
	\]
	We conclude that $\phi_{t,s}^\sharp \ind_{A_+}(y)$ and $\phi_{t,s}^\sharp \ind_{A_-}(y)$ are both positive only if $\phi_{t,s}^{-1}(\{y\})$ intersects both $A_+$ and $A_-$. In view of the assumption of the lemma, the set of such $y$ has measure $0$. This establishes property \eqref{inject} of Proposition \ref{P:renorm?}.
\end{proof}

\begin{remark}
	Observe that, when $d = 1$, $\phi_{t,s}$ is actually injective for any $0 \le s \le t \le T$. Notice that, for the drifts introduced in Examples \ref{ex:line} or \ref{ex:diagonal}, the corresponding flow has the property that $\phi_{t,s}^{-1}(\{y\})$ is at most a singleton for any $0 \le s \le t \le T$ and a.e. $y \in \RR^d$.
\end{remark}

In view of the regular Lagrangian property, a kind of injectivity of the flow can be seen for particular ordered sets. Suppose that $x,y \in \RR^d$, $x \le y$, and $\phi_{t,s}(x) = \phi_{t,s}(y)$. Then it cannot be true that $x_i < y_i$ for all $i = 1,2,\ldots, d$. If this were the case, then $\phi_{t,s}$ would be constant on the cube $[x,y]$, which violates the regular Lagrange property. The following then follows from Lemma \ref{L:inject?}.

\begin{proposition}\label{P:orderedf}
	Assume that $\oline f \in L^p_\loc(\RR^d)$ and, for a.e. $x,y \in \RR^d$ such that $\oline{f}(x) > 0$ and $\oline{f}(y) < 0$, either $x < y$ or $y < x$. Then renormalization is satisfied for $S^*(t,s) \oline{f}$ for all $ 0 \le s \le t \le T$.
\end{proposition}

The condition on $\oline{f}$ in Proposition \ref{P:orderedf} is satisfied if there exist cubes $(Q_n)_{n \in \ZZ}$ such that $Q_n < Q_{n+1}$\footnote{In other words, for all $x \in Q_n$ and $y \in Q_{n+1}$, we have $x < y$.}, $\{ \oline{f} > 0\} \subset \bigcup_{n \in \ZZ} Q_{2n}$, and $\{\oline{f} < 0\} \subset \bigcup_{n \in \AA} Q_{2n+1}$.

\subsection{Some remarks on ``time-reversed'' equations}\label{ss:timereversed?}

As discussed in subsection \ref{ss:reverse?}, for velocity fields $b$ satisfying \eqref{A:b}, there is not a satisfactory notion of reverse flow for the ODE \eqref{ODE}. Nevertheless, we can indirectly make sense of the backward Jacobian $\det(D_x \phi_{0,t}(x))$, which, formally, should be the solution of \eqref{CE} with $f_0 = 1$, that is,
\begin{equation}\label{J}
	J(t,\cdot) := S^*(t,0)1.
\end{equation}

\begin{proposition}\label{P:J}
	Let $J$ be defined by \eqref{J}.
	\begin{enumerate}[(a)]
	\item\label{Jreg} For all $p \in [1,\oo)$, $J \in L^\oo \cap C([0,T], L^p_\loc(\RR^d))$.
	\item\label{Jconv} If $(b^\eps)_{\eps > 0}$ are smooth, satisfy \eqref{A:b} uniformly in $\eps$, and converge a.e. to $b$ as $\eps \to 0$, and if $\phi^\eps_{0,t}$ is the solution of
	\[
		\del_t \phi^\eps_{0,t}(x) = - b^\eps(t, \phi^\eps_{0,t}(x)) \quad \text{in } [0,T], \quad \phi^\eps_{0,0}(x) = x,
	\]
	then, as $\eps \to 0$, $\det(D_x \phi^\eps_{0,\cdot})$ converges weakly in $C([0,T], L^p_\loc(\RR^d))$ to $J$.
	\item\label{Jcontrol} For all $f_0 \in L^\oo$ and $(t,x) \in [0,T] \times \RR^d$,
	\[
		\left| S^*(t,0)f_0(x) \right| \le \norm{f}_\oo J(t,x).
	\]
	\end{enumerate}
\end{proposition}

\begin{proof}
	Items \eqref{Jreg} and \eqref{Jconv} follow immediately from Theorem \ref{CE}. To prove \eqref{Jcontrol}, let $b^\eps$ and $\phi^\eps_{0,t}$ be as in the statement, let $f^\eps$ be the solution of \eqref{CE} with drift $b^\eps$ and initial condition $f^\eps_0 = f * \rho_\eps$, where $\rho_\eps$ is a standard mollifier. Then, for $(t,x) \in [0,T] \times \RR^d$,
	\[
		|f^\eps(t,x)| = |f^\eps_0(\phi^\eps_{0,t}(x))| \det(D_x \phi^\eps_{0,t}(x)).
	\]
	The statement follows upon sending $\eps \to 0$ and appealing to the weak convergence of $f^\eps$ and $\det(D_x \phi^\eps_{0,t})$.
\end{proof}

Continuing the formal discussion from above, note that, if $f_0$ and $b$ are smooth, then $v(t,x) = f_0(\phi_{0,t}(x))$ solves the initial value problem for the transport equation
\begin{equation}\label{TE:wrong}
	\del_t v + b(t,x) \cdot \nabla v = 0 \quad \text{in } [0,T] \times \RR^d, \quad v(0,\cdot) = f_0.
\end{equation}
The time direction for \eqref{TE:wrong} is forward, in contrast to \eqref{TE}, where it is backward. We therefore cannot appeal to the theory for that equation. Nevertheless, if $f_0 \in L^\oo$ and $b$ satisfies \eqref{A:b}, then a candidate for the solution of \eqref{TE:wrong} is
\begin{equation}\label{TE:wrong:solution?}
	v(t,x) := \frac{J(t,x)}{S^*(t,0)f_0(x)}, \quad (t,x) \in [0,T] \times \RR^d,
\end{equation}
which, by Proposition \ref{P:J}\eqref{Jcontrol}, is a well-defined bounded function. Note, however, that studying the stability properties of the formula \eqref{TE:wrong:solution?} is complicated by the fact that $J$ and $S^*(t,0)f_0$ are stable only under weak convergence in $C([0,T], L^p_\loc)$.

To complement \eqref{TE:wrong}, we also consider the terminal value problem for the continuity equation:
\begin{equation}\label{CE:wrong}
	\del_t g + \div(b(t,x) g) = 0 \quad \text{in }(0,T) \times \RR^d \quad \text{and} \quad g(T,\cdot) = g_T. 
\end{equation}
The formula in this case should be
\begin{equation}\label{CE:wrong:formula?}
	g(t,x) = g_T(\phi_{T,t}(x)) \det(D_x \phi_{T,t}(x)).
\end{equation}
In fact, both terms in the product have meaning: $u(t,x) := g_T(\phi_{T,t}(x))$ is the solution of \eqref{TE} with terminal value $g_T$, and $\det(D_x \phi_{T,t}(x))$ is well-defined almost-everywhere by Lemma \ref{L:ABVcts} and the fact that $\phi_{T,t}$ is increasing. Furthermore, by regularizing $b$ and taking weak distributional limits, it turns out that $\det(D_x \phi_{T,t}(x))$ is a measure bounded from below. However, $u$ is not continuous in general, and so it is not possible to make sense of the product in \eqref{CE:wrong:formula?}. This is exactly what leads to multiple measure-valued solutions of the equation in general; see the discussion of Example \ref{ex:diagonal} above.

\subsection{Second-order equations}

We finish this section by briefly demonstrating how the first-order results can be extended to the second-order equations
\begin{equation}\label{TE:2}
	-\del_t u - b(t,x) \cdot \nabla u - \tr[ a(t,x) \nabla^2 u] = 0\quad \text{in } (0,T) \times \RR^d \quad \text{and} \quad u(T,\cdot) = u_T
\end{equation}
and
\begin{equation}\label{CE:2}
	\del_t f + \div (b(t,x) f) - \nabla^2 \cdot (a(t,x) f) = 0 \quad \text{in }(0,T) \times \RR^d \quad \text{and} \quad f(0,\cdot) = f_0,
\end{equation}
where $b$ satisfies \eqref{A:b} and \eqref{A:bosc}, and $a: [0,T] \times \RR^d \to \mbb S^d$ is given by $a(t,x) = \frac{1}{2} \sigma(t,x) \sigma(t,x)^\tau$ and $\sigma$ is a matrix-valued function satisfying \eqref{A:sigma}. Observe that this means that
\begin{equation}\label{aconseq}
	a_{ij}(t,x) \quad \text{depends only on }(t,x_i,x_j) \in [0,T] \times \RR^2 \quad \text{for all }i,j = 1,2,\ldots, d.
\end{equation}

\begin{theorem}\label{T:TE:2}
	For all $u_T \in L^p(\RR^d)$, $1 \le p < \oo$, there exists a unique $u \in C([0,T], L^p)$ with the following properties:
	\begin{enumerate}[(a)]
	\item\label{2TE:Lp} There exists a modulus $\omega: [0,T] \to \RR_+$ such that
	\[
		\max_{t \in [0,T]} \norm{u(t,\cdot)}_{L^p} \le e^{\omega(T-t)} \norm{u_T}_{L^p}.
	\]
	\item\label{2TE:reg} If $(b^\eps)_{\eps > 0}$ is any family of smooth functions satisfying \eqref{A:b}, \eqref{A:bosc} uniformly in $\eps$, converging a.e. to $b$ as $\eps \to 0$, and $(u^\eps_T)_{\eps > 0}$ is a family of smooth functions converging in $L^p$ to $u_T$, then, as $\eps \to 0$, the unique solution $u^\eps$ of \eqref{TE:2} converges in $C([0,T] , L^p)$ to $u$. The same is true for vanishing viscosity limits.
	\item\label{2TE:form} For any $(t,x) \in [0,T] \times \RR^d)$, 
	\begin{equation}\label{TE:2:rep}
		u(t,x) = \EE[ \Phi_{T,t}(x)],
	\end{equation}
	where $\Phi$ denotes the stochastic flow from Theorem \ref{T:SDE}.
	\end{enumerate}
\end{theorem}

\begin{proof}
	The argument follows almost exactly as in the first order case (Theorem \ref{T:TEform}), using the stability and uniqueness results in Theorem \ref{T:SDE} for the SDE \eqref{SDE} (recall that we are assuming the bounded-oscillation condition \eqref{A:bosc} in addition to \eqref{A:b}). Upon regularizing the velocity field $b$, the formal a priori $L^p$ estimate in part \eqref{2TE:Lp} can be made rigorous, which, in particular, gives boundedness of the solution operator on $L^p$ for all $p \in [1,\oo)$, uniformly in $\eps> 0$, so that the initial datum $u_0$ can always be assumed to belong to $C_c(\RR^d)$ without loss of generality. The existence and uniqueness of the strong limit and its identification with the formula in part \eqref{2TE:form} are then a consequence of Theorem \ref{T:SDE}.
\end{proof}

\begin{remark}
	If $u_T$ is increasing (decreasing), then the same is true for $u(t,\cdot)$ for all $t \in [0,T]$, which, again, can be checked with the representation formula \eqref{TE:2:rep}, or by the differentiating the equation and using \eqref{aconseq}. For now, we do not discuss the question of characterizing solutions in a PDE sense.
\end{remark}

\begin{remark}
	Similar results can be obtained for the equation with lower order terms
	\begin{equation}\label{TE:allterms}
		-\del_t u - \tr[a(t,x) \nabla^2 u] - b(t,x) \cdot \nabla u - c(t,x)u - d(t,x) = 0 \quad \text{in }(0,T) \times \RR^d, \quad u(T,\cdot) = u_T.
	\end{equation}
\end{remark}

\begin{theorem}\label{T:CE:2}
	For every $f_0 \in L^p$, $1 \le p < \oo$, there exists a distributional solution of \eqref{CE:2} with the following properties:
	\begin{enumerate}[(a)]
	\item\label{2CE:stab} $f$ is obtained uniquely from weak limits in $C([0,T], L^p)$ upon replacing $b$ with a regularization, or from vanishing viscosity limits, and the resulting solution operator is bounded on $L^p$, with bound depending only on the assumptions for $b$, for all $p \in [1,\oo)$.
	\item\label{2CE:duality} If $p \in [1,\oo)$, $t \in [0,T]$ and $\oline{u} \in L^{p'}$, and if $u$ is the solution identified in Theorem \ref{T:TE:2} of \eqref{TE:2} in $[0,t] \times \RR^d$ with terminal value $u(t,\cdot) = \oline{u}$, then
	\[
		\int f(t,x)\oline{u}(x)dx = \int f_0(x) u(0,x)dx.
	\]
	\item\label{2CE:form} For all $t \in [0,T]$,
	\begin{equation}\label{CE:2:form}
		f(t,\cdot) = \EE (\Phi_{t,0})_\# f_0,
	\end{equation}
	where $\Phi_{t,0}$ is the stochastic flow from Theorem \ref{T:SDE}. Thus, if $f_0$ is a probability density and $X_0$ is a random variable independent of the Wiener process with density function $f_0$, it follows that $f(t,\cdot)$ is the probability density function of $\Phi_{t,0}(X_0)$ (which is absolutely continuous in view of the regular Lagrange property).
	\item\label{2CE:unique} If $f_0 \ge 0$ and $g \ge 0$ is a nonnegative distributional solution of \eqref{CE:2}, then $f = g$.
	\end{enumerate}
\end{theorem}

\begin{proof}
	The proofs of parts \eqref{2CE:stab}-\eqref{2CE:form} proceed similarly to the proof of Theorem \ref{T:CE}, by first regularizing $b$, proving uniform $L^p$-estimates, and passing to the limit, exploiting the uniqueness results for \eqref{SDE}. The uniqueness of nonnegative solutions in part \eqref{2CE:unique} now follows from the Ambrosio-Figalli-Trevisan superposition principle; see for instance \cite{Fig, Trevisan_16}.
\end{proof}

\section{Nonlinear transport systems}\label{sec:NTE}

We now turn to the study of the nonlinear transport systems discussed in the introduction, that is
\begin{equation}\label{eq:nonlinear}
	\del_t u + f(t,x,u) \cdot \nabla _x u + g(t,x,u) = 0 \quad \text{in } (0,T) \times \RR^d, \quad u(T,\cdot) = u_T,
\end{equation}
where, for some integer $m \ge 1$, $u: [0,T] \times \RR^d \to \RR^m$, $f:[0,T] \times \RR^d \times \RR^m \to \RR^d$, and $g: [0,T] \times \RR^d \times \RR^m \to \RR^m$. We also consider the associated forward-backward system of characteristics posed for fixed $(t,x) \in [0,T] \times \RR^d$ with $s \in [t,T]$ by
\begin{equation}\label{chars}
	\begin{dcases}
		-\del_s U_{s,t}(x) = g(s,X_{s,t}(x),U_{s,t}(x)) & U_{T,t}(x) = u_T( X_{T,t}(x)), \\
		\del_s X_{s,t}(x) = f(s,X_{s,t}(x), U_{s,t}(x)) & X_{t,t}(x) = x.
	\end{dcases}
\end{equation}

\subsection{Weak solutions}
We will introduce assumptions on $f$, $g$, and $u_T$, that, freezing $u$, make the equation \eqref{eq:nonlinear} exactly of the form of those nonconservative linear equations studied in the previous section. This then leads to a natural notion of weak solution via a fixed point operator.

Assume 
\begin{equation}\label{A:FG1}
	\left\{
	\begin{split}
	&(f,g) :[0,T] \times \RR^d \times \RR^m \to \RR^d \times \RR^m, \quad f(t,x,\cdot), g(t,x,\cdot) \text{ are continuous for }(t,x) \in [0,T] \times \RR^d,\\
	&\int_0^T \sup_{(x,u) \in \RR^d \times \RR^m} \left( \frac{ |f(t,x,u)|}{1+|x|} + \frac{ |g(t,x,u)|}{1 + |u|} \right)dt < \oo,\\
	&\text{and, for some } C_0 \in L^1_+([0,T]), \text{ a.e. } t \in [0,T], \text{ and all } i,j \in \{1,2,\ldots ,d\} \text{ and } k,\ell \in \{1,2,\ldots, m \},\\
	&\del_{x_i} f^j(t,\cdot,\cdot) \ge - C_0(t) \delta_{ij}, \quad \del_{u_k} g^\ell(t,\cdot,\cdot) \ge -C_0(t) \delta_{k\ell},\\
	&\del_{x_i} g^\ell \le 0, \quad \text{and} \quad \del_{u_k} f^j \le 0.
	\end{split}
	\right.
\end{equation}

Under these assumptions, any solution operator for \eqref{eq:nonlinear} should preserve the decreasing property of solutions, which we show formally assuming the data are smooth.

\begin{proposition}\label{P:NL:dec}
	Assume $f$, $g$, and $u_T$ are smooth with uniformly bounded derivatives, $f$ and $g$ satisfy \eqref{A:FG1}, and $u_T:\RR^d \to \RR^m$ is decreasing with respect to \eqref{order}. If $u$ is a smooth solution of the terminal value problem \eqref{eq:nonlinear}, then, for all $t \in [0,T]$, $u(t,\cdot)$ is decreasing.
\end{proposition}

\begin{proof}
	Let $i \in \{1,2,\ldots,m\}$ and $k \in \{1,2,\ldots, d\}$ be fixed, and define $v_{ik} := \del_k u^i$. Taking the derivative in $x_k$ of the $i$-component of \eqref{eq:nonlinear} gives the system
	\begin{align*}
	&\del_t v_{ik} + f(t,x,u)\cdot \nabla v_{ik} + \left( \del_{x_k} f^k(t,x,u) + \del_{u_i} g^i(t,x,u) \right) v_{ik} + \del_{x_k} g^i(t,x,u)\\
	&= - \sum_{j \ne k}\del_{x_k} f^j(t,x,u) v_{ij} - \sum_{\ell \ne i} \del_{u_\ell} g^i(t,x,u) v_{\ell k} - \sum_{j=1}^d \sum_{\ell=1}^m \del_{u_\ell} f^j(t,x,u)v_{\ell k} v_{i j}.
	\end{align*}
	In view of \eqref{A:FG1}, the system satisfied by $\left\{ v_{ik} : i = 1,2,\ldots, m, \; k = 1,2,\ldots, d \right\}$, after reversing time, is of the form in \eqref{E:Vsystem}, and so the result follows from Lemma \ref{L:particles}.
\end{proof}

We now make the connection with the linear transport equation theory of the previous sections. In particular, note that if $u(t,\cdot)$ is decreasing for all $t \in [0,T]$, then $b(t,x) := f(t, x, u(t,x))$ satisfies the assumptions of \eqref{A:b}, while $c(t) = C_0(t)$ and $d(t,x) = g(t,x,u(t,x)) + C_0(t) u(t,x)$ satisfy \eqref{A:cdextra}. 

\begin{definition}\label{D:NL:transport}
	Assume $f$ and $g$ satisfy \eqref{A:FG1} and $u_T$ is decreasing. Then a locally bounded function $u: [0,T] \times \RR^d \to \RR^m$ that is decreasing in the $\RR^d$-variable and satisfies
	\begin{equation}\label{solutionbound}
		\left\{ (t,x) \mapsto \frac{u(t,x)}{1 + |x|}\right\} \in L^\oo
	\end{equation}
	 is called a solution of \eqref{eq:nonlinear} if $u$ is a solution of the linear equation \eqref{eq:lowerorder} with $b(t,x) = f(t,x,u(t,x))$, $c(t) = C_0(t)$, and $d(t,x) = g(t,x,u(t,x)) + C_0(t) u(t,x)$.
\end{definition}

Solving \eqref{eq:nonlinear} in the sense of Definition \ref{D:NL:transport} thus amounts to solving a fixed point problem. We similarly give a weak sense to the system \eqref{chars} by using the properties of the flow from Section \ref{sec:ODE}.

\begin{proposition}\label{P:chars}
	Assume that $u_T \in L^\oo_\loc$ and $u$ is a solution of \eqref{eq:nonlinear} in the sense of Definition \ref{D:NL:transport}. Let $\phi$ be the forward regular Lagrangian flow as in Section \ref{sec:ODE} corresponding to
	\[
		\del_t \phi_{t,s}(x) = f(t, \phi_{t,s}(x), u(t, \phi_{t,s}(x))) \quad t \in [s,T], \quad \phi_{s,s}(x) = x,
	\]
	and define, for $0 \le s \le t \le T$ and $x \in \RR^d$,
	\[
		X_{t,s}(x) = \phi_{t,s}(x) \quad \text{and} \quad U_{t,s}(x) := u(t, \phi_{t,s}(x)). 
	\]
	Then, for all $1 \le p < \oo$, $X_{\cdot,s}, U_{\cdot,s} \in C([s,T], L^p_\loc) \cap L^p_\loc(\RR^d, W^{1,1}([0,T]))$, and, for a.e. $x \in \RR^d$, \eqref{chars} is satisfied in the integral sense.
\end{proposition}

\begin{proof}
	The regularity properties of $X$ are seen from Theorem \ref{T:regLag}, as well as the fact that, for a.e. $x \in \RR^d$, 
	\[
		X_{t,s}(x) = x + \int_s^t f(r, X_{r,s}(x), u(r, X_{r,s}(x))dr = x + \int_s^t f(r, X_{r,s}(x), U_{r,s}(x))dr.
	\]
	Theorems \ref{T:regLag} and \ref{T:TE:lowerorder} together imply that $U_{\cdot,s} \in C([s,T], L^p_\loc)$. Also, in view of Theorem \ref{T:TE:lowerorder}\eqref{lowerorder:formula}, $u$ satisfies, for any $0 \le s \le t \le T$,
	\[
		U_{s,s}(x) = u(s,x) = u(t, X_{t,s}(x))+ \int_s^t g(r, X_{r,s}(x), u(r, \phi_{r,s}(x)) ) dr
		= U_{t,s}(x) + \int_s^t g(r, X_{r,s}(x), U_{r,s}(x))dr.
	\]
	It follows that, in the distributional sense,
	\[
		\del_t U_{t,s}(x) = g(t, X_{t,s}(x), U_{t,x}(x)).
	\]
	Arguing as in the proof of Theorem \ref{T:regLag}, the right-hand side belongs to $L^p_\loc(\RR^d, L^1([0,T]))$, and we conclude as in that theorem.
\end{proof}

\subsection{Minimal and maximal solutions}

In this section, we show that the assumptions \eqref{A:FG1} give an increasing structure to the equation \eqref{eq:nonlinear}, which allows for the identification of a unique minimal and maximal solution.

\begin{theorem}\label{T:NL:maxmin}
	Assume $f$ and $g$ satisfy \eqref{A:FG1}, $u_T$ is decreasing, and $u_t (1 + |\cdot|)^{-1} \in L^\oo$. Then there exist two decreasing solutions $u^+$, $u^-$ of \eqref{eq:nonlinear} in the sense of Definition \ref{D:NL:transport} with the following properties:
	\begin{enumerate}[(a)]
	\item For all $t \in [0,T]$, $u^+(t,\cdot)$ is upper semicontinuous, $u^-(t,\cdot)$ is lower semicontinuous, and both $u^+$ and $u^-$ are continuous a.e. in $[0,T] \times \RR^d$.
	\item If $u$ is any other solution, then $u^- \le u \le u^+$.
	\end{enumerate}
\end{theorem}

Throughout this subsection, we may assume, without loss of generality, that the term $C_0 \in L^1_+$ in \eqref{A:FG1} is $0$. Indeed, the function
\[
	\tilde u(t,x) = \exp\left( \int_t^T C_0(s)ds\right) u\left( t, \exp\left( \int_t^TC_0(s)ds \right)x \right),
\]
is decreasing in $x$, satisfies \eqref{solutionbound}, and, formally, solves the equation \eqref{eq:nonlinear} with $f$ and $g$ replaced by
\[
	\tilde f(t,x,u) = \exp\left( -\int_t^T C_0(s)ds\right) f\left( t, \exp\left( \int_t^T C_0(s)ds\right) x, \exp\left( -\int_t^T C_0(s)ds\right) u \right) + C_0(t) x
\]
and
\[
	\tilde g(t,x,u) = \exp\left( \int_t^T C_0(s)ds\right) g \left( t, \exp\left( \int_t^T C_0(s)ds\right) x, \exp\left( -\int_t^T C_0(s)ds\right) u \right) + C_0(t) u.
\]

\begin{proof}[Proof of Theorem \ref{T:NL:maxmin}]
	As described above, we assume without loss of generality that $C_0 \equiv 0$. For some $\oline C > 0$ to be determined, set	
	\begin{equation}\label{lattice}
	\begin{split}
	\mcl L := \Big\{ u: [0,T] \times \RR^d \to \RR^m : \; &u(T,\cdot) = u_T, \; u(t,\cdot) \text{ is decreasing for all }t \in [0,T], \text{ and}\\
	&|u(t,x)| \le \oline{C}(1 + |x|) \text{ for all } (t,x) \in [0,T] \times \RR^d \Big\}.
	\end{split}
	\end{equation}
	We define a map $\mcl S$ on $\mcl L$ as follows: for $u \in \mcl L$, let $v := \mcl S(u)$ be the solution, as in Theorem \ref{T:TE:lowerorder}, of the linear transport equation 
	\[
		\del_t v + f(t,x, u(t,x)) \cdot \nabla v + g(t,x,u(t,x)) = 0 \quad \text{in }(0,T) \times \RR^d, \quad v(T,\cdot) = u_T.
	\]
	Then, in view of the solution formula \eqref{formula:lowerorder} and the bounds \eqref{A:FG1} on $f$ and $g$, there exists a sufficient large $\oline{C}$, depending on $u_T$, such that $\mcl S$ maps $\mcl L$ into $\mcl L$.
	
	We now note that $\mcl L$ forms a \emph{complete lattice} under the partial order
	\begin{equation}\label{fnorder}
		u \le \tilde u \quad \Leftrightarrow \quad u^i(t,x) \le \tilde u^i(t,x) \quad \text{for all } (t,x) \in [0,T] \times \RR^d, \, i = 1,2,\ldots, m;
	\end{equation}
	that is, every subset of $\mcl L$ has a greatest lower bound and least upper bound, which is a consequence of the uniformly bounded linear growth of solutions in $\mcl L$.
	
	Suppose now that $u, \tilde u \in \mcl L$ satisfy $u \le \tilde u$ under the order \eqref{fnorder}, and set
	\[
		b(t,x) := f(t,x, u(t,x)) \quad \text{and} \quad d(t,x) := g(t,x,u(t,x)).
	\]
	Then \eqref{A:FG1} with $C_0 \equiv 0$ implies that $b(t,x) \ge f(t,x,\tilde u(t,x))$ and $d(t,x) \le g(t,x,\tilde u(t,x))$, and so, in particular, $v := \mcl S(u)$ and $\tilde v := \mcl S(\tilde u)$ are respectively a sub and supersolution of the linear equation \eqref{eq:lowerorder} with $c(t) \equiv 0$. It follows from Theorem \ref{T:TE:comparison:lowerorder} that $v \le \tilde v$, and therefore $\mcl S$ is increasing on the complete lattice $\mcl L$ with respect to the partial order \eqref{fnorder}. The existence of a unique maximal and minimal solution are now a consequence of the Tarski lattice-theoretical fixed point theorem \cite{Tarski}.
	
	The continuity a.e. in $[0,T] \times \RR^d$ of $u^+$ and $u^-$ now follows from Theorem \ref{T:TE:comparison:lowerorder}. Observe now that any version of the maximal solution $u^+$ is a solution in the sense of Definition \ref{D:NL:transport}. Because $u^+(t,\cdot)$ is decreasing, it is continuous a.e., and its maximal version is upper semicontinuous, which is therefore also the unique maximal solution in the \emph{everywhere}-pointwise sense. A similar argument shows that the minimal everywhere-pointwise solution $u^-$ is lower-semicontinuous in the spatial variable, and we conclude.
\end{proof}

The fixed point theorem of \cite{Tarski} used above further characterizes $u^+$ and $u^-$ as 
\[
	u^+ = \sup\left\{ u \in \mcl L: \mcl S(u) \ge u \right\} 
	\quad \text{and} 
	\quad u^- = \inf \left\{ u \in \mcl L: \mcl S(u) \le u \right\},
\]
where the $\sup$ and $\inf$ are understood with respect to the partial order \eqref{fnorder}. We alternatively characterize the maximal and minimal solutions in terms of appropriately defined sub and supersolutions of the equation \eqref{eq:nonlinear}.

Recall that $\rho^\eps$ and $\rho_\eps$ are the one-sided mollifying functions from subsection \ref{ss:regs}. We then define a notion of sub and supersolution.

\begin{definition}\label{D:NL:subsuper}
	Assume $f$ and $g$ satisfy \eqref{A:FG1} with $C_0 \equiv 0$ and $u_T$ is decreasing. A function $u:[0,T] \times \RR^d \to \RR^m$ that is decreasing in $\RR^d$ is called a subsolution (resp. supersolution) of \eqref{eq:nonlinear} if $u(T,\cdot) \le u_T$ (resp. $u(T,\cdot) \ge u_T$) and $u^\eps := u * \rho^\eps$ (resp. $u_\eps := u * \rho_\eps$ satisfies
	\begin{align*}
		&-\del_t u^\eps - f(t,x, u(t,x)) \cdot \nabla u^\eps - g(t,x,u(t,x)) \le 0\\
		& \left( \text{resp. } -\del_t u_\eps - f(t,x,u(t,x)) \cdot \nabla u_\eps - g(t,x,u(t,x)) \ge 0 \right)
		\quad \text{a.e. in } [0,T] \times \RR^d.
	\end{align*}
\end{definition}

In other words, sub and supersolutions in the sense of \eqref{D:NL:subsuper} are sub and supersolutions of the corresponding linear equations identified in Definition \ref{D:NL:transport}. It follows that a function which is both a sub and supersolution is in fact a solution.

\begin{lemma}\label{L:maxmin}
	Under the conditions of Definition \ref{D:NL:subsuper}, assume that $u$ and $v$ are two subsolutions (resp. supersolutions) of \eqref{eq:nonlinear}. Then $u \vee v$ (resp. $u \wedge v$) is also a subsolution (resp. supersolution).
\end{lemma}

\begin{proof}
	We prove only the subsolution statement, as the other one is analogously proved. It is clear that $\max\{ u(T,\cdot) , v(T,\cdot) \} \le u_T$ and $u \vee v$ is decreasing in $\RR^d$.
	
	Set $u^\eps = u * \rho^\eps$ and $v^\eps = v * \rho^\eps$. Let $\Psi: \RR^m \times \RR^m \to \RR^m$ be smooth, increasing, and satisfy $\Psi(u,v) \ge u \vee v$ for all $u,v \in \RR^m$. Note that, using \eqref{A:FG1} with $C_0 \equiv 0$,
	\begin{align*}
		&f(t,x,u(t,x)) \wedge f(t,x,v(t,x)) \ge f(t,x, u(t,x) \vee v(t,x)) \quad \text{and} \\
		&g(t,x,u(t,x)) \vee g(t,x, v(t,x)) \le g(t,x,u(t,x) \vee v(t,x)),
	\end{align*}
	and so a simple application of the chain rule gives
	\[
		\del_t \Psi(u^\eps, v^\eps) + f(t,x, u \vee v) \cdot \nabla \Psi(u^\eps, v^\eps) + g(t,x,u \vee v) \ge 0.
	\]
	In particular, for any smooth positive test function $\phi \in C_c^\oo((0,T) \times \RR^d)$,
	\[
		\int_0^T \int_{\RR^d} \left( \Psi(u^\eps,v^\eps)(t,x) \del_t \phi(t,x) - g(t,x,u \vee v)\phi(t,x) \right) dxdt+ \langle \nabla (f(\cdot,\cdot, u \vee v)\phi) , \Psi(u^\eps, v^\eps) \rangle \le 0,
	\]
	where the last term is understand as the pairing between locally finite measures and continuous functions, because $f(t,x,(u \vee v)(t,x))$ is $BV$ in $x$. We may then approximate $\max(\cdot,\cdot)$ with such functions $\Psi$ and determine that, in the distributional sense,
	\[
		\del_t (u^\eps \vee v^\eps) + f(t,x, u \vee v) \cdot \nabla (u^\eps \vee v^\eps) + g(t,x,u \vee v) \ge 0.
	\]
	For $\delta > 0$, we convolve both sides of the above inequality with the one-sided mollifier $\rho^\delta$ and obtain, in view of \eqref{A:FG1},
	\[
		\del_t (u^\eps \vee v^\eps) * \rho^\delta + f(t,x,u \vee v) \cdot \nabla \left[ (u^\eps \vee v^\eps) * \rho^\delta \right] + g(t,x,u \vee v) \ge 0.
	\]
	For fixed $\delta$, as $\eps \to 0$, $(u^\eps \vee v^\eps) * \rho^\delta \to (u \vee v) * \rho^\delta$ locally uniformly, and we may therefore send $\eps \to 0$ in the above inequality, again using $f \in BV$, to obtain the desired subsolution inequality for $(u \vee v) * \rho^\delta$.
\end{proof}

\begin{lemma}\label{L:ctsae}
	Assume $f$ and $g$ satisfy \eqref{A:FG1} with $C_0 \equiv 0$, $u_T$ is decreasing, and $u$ is a subsolution (supersolution) of \eqref{eq:nonlinear} in the sense of Definition \ref{D:NL:subsuper}. Then there exists a subsolution (supersolution) $\tilde u$ such that $u \le \tilde u$ ($u \ge \tilde u$) and $\tilde u$ is continuous a.e. in $[0,T] \times \RR^d$.
\end{lemma}

\begin{proof}
	For such a subsolution $u$, let $\tilde u$ be the solution of the linear transport equation
	\[
		\del_t \tilde u + f(t,x,u(t,x)) \cdot \nabla \tilde u + g(t,x,u(t,x)) = 0 \quad \text{in }(0,T) \times \RR^d, \quad \tilde u(T,\cdot) = u_T.
	\]
	Then, by Theorem \ref{T:TE:comparison:lowerorder}, $u \le \tilde u$, and $\tilde u$ is continuous a.e. in $[0,T] \times \RR^d$. By \eqref{A:FG1},
	\[
		f(t,x,u(t,x)) \ge f(t,x,\tilde u(t,x)) \quad \text{and} \quad g(t,x,u(t,x)) \le g(t,x,\tilde u(t,x)),
	\]
	and it follows that $\tilde u$ is a subsolution of \eqref{eq:nonlinear}. The argument for supersolutions is identical.
\end{proof}

\begin{proposition}\label{P:Perron}
	Let $f$ and $g$ satisfy \eqref{A:FG1} with $C_0 \equiv 0$, and assume $u_T$ has linear growth and is decreasing. Let $u^+$ and $u^-$ be the maximal and minimal solution from Theorem \ref{T:NL:maxmin}. Then, in the sense of Definition \ref{D:NL:subsuper}, $u^+$ is the pointwise maximum of all subsolutions, and $u^-$ is the pointwise minimum of all supersolutions.
\end{proposition}

\begin{remark}
	In view of Lemma \ref{L:ctsae}, the maximum/minimum in Proposition \ref{P:Perron} may be restricted to sub/supersolutions that are continuous a.e. in $[0,T] \times \RR^d$.
\end{remark}

\begin{proof}[Proof of Proposition \ref{P:Perron}]
	We prove only the statement for $u^+$, since the proof is analogous for $u^-$. Let
	\[
		\tilde u^+(t,x) := \sup\left\{ u(t,x) : u \text{ is a subsolution} \right\}.
	\]
	Because $u^+$ is itself a subsolution, we clearly have $u^+ \le \tilde u^+$. Suppose now that $u$ is a subsolution of \eqref{eq:nonlinear}, and let $v$ be the solution of 
	\[
		\del_t v + f(t,x,u) \cdot \nabla v + g(t,x,u) = 0 \quad \text{in }(0,T) \times \RR^d \quad \text{and} \quad v(T,\cdot) = u_T.
	\]
	It then holds that $u$ and $v$ are respectively a  subsolution and solution of the linear equation \eqref{eq:lowerorder} with $b(t,x) = f(t,x,u(t,x))$, $c(t,x) = 0$, and $d(t,x,) = g(t,x,u(t,x))$. Then Theorem \ref{T:TE:comparison:lowerorder} implies $u \le v$. Note also that $v$ belongs to the lattice $\mcl L$ from the proof of Theorem \ref{T:NL:maxmin}, because $v(T,\cdot) = u_T$. If $\mcl S$ is the fixed-point map from the proof of that theorem, we then set $w = \mcl S(v)$, which, by a similar argument, satisfies $w \in \mcl L$ and $S(v) = w \ge v$. By the characterization of $u^+$ by the Tarski fixed point theorem, we must have $u \le u^+$, and therefore $\tilde u^+ \le u^+$.
\end{proof}

\subsection{Continuous solutions}

We now investigate when the maximal and minimal solution $u^+$ and $u^-$ identified in the previous subsection coincide. In this subsection, we prove a comparison principle for \emph{continuous} sub and supersolutions, so that, in particular, $u^+ = u^-$ if both are continuous. As we will see by example in the forthcoming subsections, this can fail in general if the assumption of continuity is dropped.

We will first introduce a different but equivalent notion of solution for continuous solutions, using the theory of viscosity solutions. Let us assume throughout this subsection, in addition to \eqref{A:FG1}, that 
\begin{equation}\label{A:FGcts}
	f \text{ and } g \text{ are uniformly Lipschitz continuous and bounded, and } C_0 \ge 0 \text{ is constant.}
\end{equation}

\begin{definition}\label{D:nonlinear:viscosity}
	A function $u: [0,T] \times \RR^d \to  \RR^m$ is a viscosity subsolution (supersolution) of \eqref{eq:nonlinear} if $u(t,\cdot)$ is decreasing for all $t \in [0,T]$, and, whenever $\phi \in C^1([0,T] \times \RR^d, \RR)$, $i = 1,2,\ldots, m$, and $u^{i,*}(t,x) - \phi(t,x)$ (resp. $u^i_*(t,x) - \phi(t,x)$) attains a local maximum (resp. minimum) at $(t_0,x_0)$, 
	\[
		-\phi_t(t_0,x_0) - f(t_0,x_0, u^*(t_0,x_0)) \cdot \nabla \phi(t_0,x_0) - g(t_0,x_0,u^*(t_0,x_0)) \le 0
	\]
	(resp.
	\[
		-\phi_t(t_0,x_0) - f(t_0,x_0, u_*(t_0,x_0)) \cdot \nabla \phi(t_0,x_0) - g(t_0,x_0,u_*(t_0,x_0)) \ge 0).
	\]
	A viscosity solution is both a sub and supersolution.
\end{definition}

\begin{lemma}\label{L:equivdef}
	Assume \eqref{A:FG1} and \eqref{A:FGcts}. If $u$ is a sub and supersolutions in the sense of Definition \ref{D:NL:transport} and is almost-everywhere continuous in $[0,T] \times \RR^d$, then $u$ is a viscosity sub (super) solution in the sense of Definition \ref{D:nonlinear:viscosity}.
\end{lemma}

\begin{remark}
	In view of Lemma \ref{L:ctsae}, no generality is lost in considering sub or supersolutions that are continuous a.e. in $[0,T] \times \RR^d$.
\end{remark}

\begin{proof}[Proof of Lemma \ref{L:equivdef}]
	We prove only the subsolution property. Note that, because $u$ is decreasing, we have $u(t,\cdot)^* = u(t,\cdot)$ a.e. Then $u^\eps = u * \rho^\eps$ is a (classical) subsolution of
	\[
		-\del_t u^\eps - f(t,x, u^*(t,x)) \cdot \nabla u^\eps - g(t,x, u^*(t,x)) \le 0,
	\]
	and, as $\eps \to 0$, $u^\eps \searrow u^*$.
	
	Assume that $u^*(t,x) - \phi(t,x)$ attains a maximum at $(t_0,x_0)$, which we may assume to be strict by adding a quadratic penalization to $\phi$. For $\eps > 0$, let $(t_\eps,x_\eps)$ denote the maximum point of $u^\eps(t,x) - \phi(t,x)$ on $[0,T] \times \oline{B_1(x_0)}$, and let $(s,y)$ be a limit point of $(t_\eps,x_\eps)_{\eps > 0}$ along some subsequence $(\eps_n)_{n \in \NN}$. By Lemma \ref{L:approxid}\eqref{moll:shift}, 
	\[
		u^\eps(t_\eps,x_\eps) \le u(t_\eps, x_\eps - 2\eps \mbf 1) + \eps \mbf 1,
	\]
	and so
	\[
		\limsup_{n \to \oo} u^{\eps_n}(t_{\eps_n}, x_{\eps_n}) \le u^*(s,y).
	\]
	Let now $(s_n, y_n)_{n \in \in\NN}\subset [0,T] \times \oline{B_1(x_0)}$ be such that 
	\[
		(s_n,y_n) \xrightarrow{n \to \oo} (t_0,x_0) \quad \text{and} \quad u^{\eps_n}(s_n,y_n) \xrightarrow{n \to \oo} u^*(t_0,x_0).
	\]
	Taking $n \to \oo$ in the inequality $u^{\eps_n}(s_n,y_n) - \phi(s_n,y_n) \le u^{\eps_n}(t_{\eps_n}, x_{\eps_n}) - \phi(t_{\eps_n}, x_{\eps_n})$ yields $u^*(t_0,y_0) - \phi(t_0,y_0) \le u^*(s,y) - \phi(s,y)$. This implies that $(s,y) = (t_0,x_0)$, the limit holds along the full family $\eps \to 0$, and $\lim_{\eps \to 0} u^\eps(t_\eps,x_\eps) = u^*(t_0,x_0)$.
	
	By standard maximum principle considerations, and the fact that $-f$ and $g$ are increasing in $u$ and $\nabla u^\eps \le 0$,
	\[
		-\del_t u^\eps(t_\eps,x_0) - f(t_\eps,x_\eps, u^\eps(t_\eps,x_\eps)) \cdot \nabla u^\eps(t_\eps,x_\eps) - g(t_\eps,x_\eps, u^\eps(t_\eps,x_\eps)) \le 0.
	\]
	Sending $\eps \to 0$ gives the desired subsolution inequality.
\end{proof}

\begin{theorem}\label{T:NL:cts}
	Assume that $f$ and $g$ satisfy \eqref{A:FG1} and \eqref{A:FGcts}, and let $u$ and $v$ be respectively a bounded sub and supersolution of \eqref{eq:nonlinear} in the sense of Definition \ref{D:nonlinear:viscosity} such that either $u$ is continuous and $v$ is lower-semicontinuous, or $u$ is upper-semicontinuous and $v$ is continuous. If $u(T,\cdot) \le v(T,\cdot)$, then, for all $t \in [0,T]$, $u(t,\cdot) \le v(t,\cdot)$.
\end{theorem}

\begin{proof}
	The proofs of both statements are almost identical, so we prove only the statement when $u$ is continuous and $v$ is lower-semicontinuous.

	We first prove the result under the additional assumption that, for some $\delta > 0$, $u$ is a sub-solution of
	\begin{equation}\label{E:deltaperturb}
		-\del_t u - \left[ f(t,x,u) + \delta \mbf 1\right]\cdot \nabla u - g(t,x,u) \le - \delta \mbf 1
	\end{equation}
	and $u(T,\cdot) - v(T,\cdot) \le -\delta \mbf 1$. Standard arguments from the theory of viscosity solutions then imply that $w(t,x,y) := u(t,x) - v(t,y)$ is a subsolution of
	\begin{equation}\label{doubled}
		-\del_t w - \left[ f(t,x,u) + \delta \mbf 1 \right] \cdot \nabla_x w -  f(t,x,v) \cdot \nabla_y w - g(t,x,u) + g(t,x,v) \le - \delta \mbf 1.
	\end{equation}

	Let $M > 0$ be such that $\nor{u}{\oo} \le M$ and $\nor{v}{\oo} \le M$, fix $\lambda \ge 1$ and $\beta > 0$, set
	\[
		[0,T] \ni t \mapsto \phi(t) := \max_{x,y \in \RR^d} \max_{i=1,2,\ldots,m} \left( u^{i}(t,x) - v^i(t,y) - \frac{\lambda}{2} |x-y|^2 - \frac{\beta}{2}(|x|^2 + |y|^2) \right),
	\]
	and, for some constant $C > 0$ to be chosen, define
	\[
		\hat t := \sup\left\{ t \in (-\oo,T] : e^{-C\lambda^{1/2} (T-t)} \phi(t) > 0 \right\}.
	\]
	As $\lambda \to \oo$ and $\beta \to 0$,
	\[
		\phi(T) = \max_{x,y,i} \left( u^{i}(T,x) - v^i(T,y) - \frac{\lambda}{2} |x-y|^2 - \frac{\beta}{2} (|x|^2 + |y|^2) \right) \le -\delta + o(1),
	\]
	and so, if $\lambda$ is sufficiently large and $\beta$ is sufficiently small, depending on $\delta$, then $\phi(T) < 0$, and therefore $\hat t  < T$.
	
	We next claim that $\hat t  < 0$ for all sufficiently large $\lambda$ and small $\beta$, which implies $\phi(t) \le 0$ for all $t \in [0,T]$, and, hence, gives the result. If this were not true and $0 \le \hat t < T$, then choose $i, \hat x,\hat y$ such that
	\[
		0 = \phi(\hat t) = u^{i}(\hat t,\hat x) - v^i(\hat t,\hat y) - \frac{\lambda}{2} |\hat x - \hat y|^2 - \frac{\beta}{2} (|\hat x|^2 + |\hat y|^2).
	\]
	We then have
	\begin{equation}\label{jiordering}
		u^{j}(\hat t,\hat x) - v^j(\hat t,\hat y) \le u^{i}(\hat t,\hat x) - v^i(\hat t,\hat y) \quad \text{for all } j =1,2,\ldots, m,
	\end{equation}
	and, from the fact that $u$ and $v$ are decreasing,
	\begin{equation}\label{testfnmon}
		\lambda(\hat x_j - \hat y_j)  + \beta \hat x_j \le 0 \quad \text{and} \quad \lambda(\hat x_j - \hat y_j) - \beta \hat y_j \le 0 \quad \text{for all } j = 1,2,\ldots, d.
	\end{equation}
	Moreover, in view of the boundedness of $u$ and $v$, arguing as in for instance \cite[Lemma 3.1]{CIL},
	\begin{equation}\label{Lipapprox}
		\lambda|\hat x - \hat y| \le 2^{3/2} M^{1/2} \lambda^{1/2},
	\end{equation}
	\begin{equation}\label{Lipapprox2}
		\max(|\hat x|, |\hat y|) \le 2^{3/2} M^{1/2} \beta^{-1/2},
	\end{equation}
	and, for some $\delta_\lambda$ and $\eps_\beta$ satisfying $\lim_{\lambda \to \oo} \delta_\lambda = \lim_{\beta \to 0} \eps_\beta = 0$,
	\begin{equation}\label{vanishingpenal}
		\lambda|\hat x - \hat y|^2 \le \delta_\lambda \quad \text{and} \quad \beta(|\hat x|^2 + |\hat y|^2 \le \eps_\beta.
	\end{equation}
	On $[\hat t,T] \times \RR^d \times \RR^d$, the function
	\[
		u^{i}(t,x) - v^i(t,y) - \frac{\lambda}{2} |x-y|^2  - \frac{\beta}{2} (|x|^2 + |y|^2) - e^{-C \lambda^{1/2}(t - \hat t)} \phi(\hat t)
	\]
	achieves a maximum at $(\hat t, \hat x,\hat y)$, and so, applying Definition \ref{D:nonlinear:viscosity} to the doubled equation \eqref{doubled},
	\begin{equation}\label{applydef}
	\begin{split}
		-\delta \ge C \lambda^{1/2} \phi(\hat t)
		&- \sum_{j = 1}^d (f^j(\hat t, \hat x, u(\hat t, \hat x)) + \delta)\left( \lambda (\hat x_j - \hat y_j) + \beta \hat x_j\right) - g^i(\hat t, \hat x, u(\hat t, \hat x)) \\
		&+ \sum_{j=1}^d f^j(\hat t, \hat y, v(\hat t, \hat y))\left( \lambda (\hat x_j - \hat y_j) - \beta \hat y_j \right)  + g^i(\hat t, \hat y, v(\hat t, \hat y)).
	\end{split}
	\end{equation}	
	Because of \eqref{A:FG1}, \eqref{A:FGcts}, \eqref{jiordering}, and \eqref{testfnmon}, we have
	\begin{align*}
		&\sum_{j=1}^d \left[ f^j(\hat t, \hat x,u(\hat t,\hat x)) + \delta \right] \left(\lambda(\hat x_j - \hat y_j) + \beta \hat x_j \right)\\
		&\le \sum_{j=1}^d \left[ f^j(\hat t, \hat x,v(\hat t,\hat y) + (u^{i}(\hat t,\hat x) - v^i(\hat t,\hat y))\mbf 1 ) - \delta \right]\left( \lambda( \hat x_j - \hat y_j) + \beta \hat x_j \right)\\
		&\le -\frac{1}{d^{1/2}} \delta \lambda|\hat x - \hat y| +  \sum_{j=1}^d f^j(\hat t, \hat y,v(\hat t, \hat y)) \left( \lambda (\hat x_j - \hat y_j) - \beta \hat y_j \right) + \nor{\nabla_u f}{\oo} \lambda|\hat x - \hat y|(u^{i}(\hat t,\hat x) - v^i(\hat t, \hat y)) \\
		&+ \nor{\nabla_x f}{\oo}\lambda|\hat x - \hat y|^2 + \left( \frac{\delta}{d^{1/2}} + \nor{f}{\oo} \right) \beta |\hat x| + \nor{f}{\oo} \beta |\hat y|.
	\end{align*}
	Similarly, using the fact that $g^i$ is nondecreasing in the $u_j$-variable for all $j \ne i$, we have
	\[
		g^i(\hat t, \hat x, u(\hat t, \hat x))
		\le g(\hat t, \hat y, v(\hat t, \hat y)) + \norm{\nabla_u g}_\oo ( u^i(\hat t, \hat x) - v^i(\hat t, \hat y)) 
		+ \norm{\nabla_x g}_\oo |\hat x - \hat y|.
	\]
	Therefore, \eqref{applydef} becomes, using \eqref{Lipapprox}, \eqref{Lipapprox2}, and \eqref{vanishingpenal},
	\begin{equation}\label{contradiction}
	\begin{split}
		-\delta &\ge \left( C\lambda^{1/2} - \nor{\nabla_u f}{\oo} \lambda|\hat x - \hat y| - \nor{\nabla_u g}{\oo} \right)(u^i(\hat t,\hat x) - v^i(\hat t,\hat y)) \\
		& + \lambda|\hat x - \hat y| \left( \frac{1}{d^{1/2}} \delta  - \frac{C \delta_\lambda^{1/2}}{2} \right)\\
		& - \nor{\nabla_x f}{\oo} \delta_\lambda - 2^{3/2} M^{1/2} \nor{\nabla_x g}{\oo} \lambda^{-1/2}\\
		&  - 2^{3/2} M^{1/2} \left( \frac{\delta}{d^{1/2}} + \nor{f}{\oo} \right)\beta^{1/2} - 2^{3/2} M^{1/2}\nor{f}{\oo} \beta^{1/2} - 2\frac{C\lambda^{1/2}}{2} \eps_\beta.
	\end{split}
	\end{equation}
	Once again using \eqref{Lipapprox},
	\[
		C\lambda^{1/2} - \nor{\nabla_u f}{\oo} \lambda|\hat x - \hat y| - \nor{\nabla_u g}{\oo}
		\ge C \lambda^{1/2} - 2^{3/2} \nor{\nabla_u f}{\oo} \lambda^{1/2}M^{1/2} - \nor{\nabla_u g}{\oo} \ge 0,
	\]
	provided $C$ is large enough, depending only $f$, $g$ and $M$. Therefore, since $u^i(\hat t, \hat x) \ge v^i(\hat t, \hat y)$, the first two lines of \eqref{contradiction} are nonnegative if $\lambda$ is sufficiently large, depending on $\delta$. Taking first $\lambda$ sufficiently large and then $\beta$ sufficiently small, we conclude that the right-hand side of \eqref{contradiction} is strictly larger than $-\delta$, which is the desired contradiction.
	
	We now prove the general statement. For $\delta > 0$ and $(t,x) \in [0,T] \times \RR^d$, set
	\[
		\tilde u(t,x) = u(t,x + \delta \psi_1(t)\mbf 1) - \delta \psi_2(t)\mbf 1,
	\]
	where $\psi_1$ and $\psi_2$ are two nonnegative scalar functions satisfying $\psi_1(T) = 0$ and $\psi_2(T) = 1$. Then
	\[
		\tilde u(T,\cdot) = u(T,\cdot) - \delta\mbf 1,
	\]
	and so, using the the fact that $u$ is nonincreasing, we formally compute\footnote{The computations can be made rigorous using test functions and Definition \ref{D:nonlinear:viscosity}. Note that the argument for $u$ above is always $(t, x + \delta \psi_1(t) \mbf 1)$.}
	\begin{equation}\label{ineqtocheck}
	\begin{split}
		-\del_t \tilde u &- \left[ f(t,x,\tilde u) - \delta \right]\nabla \tilde u - g(t,x, \tilde u)\\
		&= \del_t u - \delta \nabla_x u \cdot \mbf 1 \dot \psi_1 + \delta \dot \psi_2 \mbf 1
		- \left[ f(t, x + \delta \psi_1\mbf 1, u - \delta \psi_2\mbf 1) + \delta \right] \nabla u - g(x + \delta \psi_1\mbf 1, u - \delta \psi_2\mbf 1) + \delta \mbf 1\\
		&\le \delta (-\mbf 1 \cdot \nabla u) \left( \dot \psi_1 + \norm{\nabla_u f}_\oo \psi_2 + 1\right)
		+ \delta\left( \dot \psi_2 + \norm{\nabla_u g}_\oo \psi_2  + 1 \right) \mbf 1.
	\end{split}
	\end{equation}
	We then choose $\psi_1$ and $\psi_2$ so as to satisfy
	\[
		\dot \psi_1 = - \norm{\nabla_u f}_\oo\psi_2 - 1 \quad \text{and} \quad
		\dot \psi_2 = - \norm{\nabla _u g}_\oo\psi_2  - 1,
	\]
	that is,
	\begin{align*}
		\psi_1(t) &= \frac{ \norm{\nabla_u f}_\oo ( \norm{\nabla_u g}_\oo + 1) }{\norm{\nabla_u g}_\oo^2} ( e^{\norm{\nabla_u g}_\oo(T-t)} - 1) + \left(1 - \frac{\norm{\nabla_u f}_\oo}{\norm{ \nabla_u g}_\oo} \right) (T-t)
		\quad \text{and} \\
		\psi_2(t) &= \left( 1 + \frac{1}{ \norm{\nabla_u g}_\oo} \right)e^{\norm{\nabla_u g}_\oo(T-t)} - \frac{1}{\norm{\nabla_u g}_\oo}.
	\end{align*}
	Then \eqref{ineqtocheck} becomes exactly \eqref{E:deltaperturb}, and so, for $(t,x) \in [0,T] \times \RR^d$,
	\[
		u(t,x + \delta\psi_1(t)) - \delta \psi_2(t) \le v(t,x).
	\]
	We conclude upon sending $\delta \to 0$ and appealing to the continuity of $u$.
	\end{proof}
	
Let us note the following corollary of Theorem \ref{T:NL:maxmin}, Lemma \ref{L:equivdef}, and Theorem \ref{T:NL:cts}.

\begin{corollary}
	Under the same conditions on $f$ and $g$ as in Theorem \ref{T:NL:cts}, let $u_T$ be bounded and decreasing and let $u^+$ and $u^-$ be the maximal and minimal solutions identified in Theorem \ref{T:NL:maxmin}. Then
	\[
		u^+(t,x) = \inf\left\{ u(t,x) : u \text{ is a continuous viscosity supersolution with } u(T,\cdot) \ge u_T \right\}
	\]
	and
	\[
		u^-(t,x) = \sup \left\{ u(t,x) : u \text{ is a continuous viscosity subsolution with } u(T,\cdot) \le u_T \right\}.
	\]
\end{corollary}

	The comparison principle with continuous sub and supersolutions also implies a conditional uniqueness and stability statement.
		
	\begin{theorem}
		Assume $f$ and $g$ satisfy \eqref{A:FG1} and \eqref{A:FGcts} and $u_T$ is bounded, continuous, and decreasing. If there exists a bounded and continuous viscosity solution $u$ of \eqref{eq:nonlinear}, then it is the unique viscosity solution. Moreover, if $u^\eps$ is the unique classical solution of 
		\begin{equation}\label{eq:nonlinear:visc}
			\del_t u^\eps + f(t,x,u^\eps) \cdot \nabla u^\eps + g(t,x,u^\eps) + \eps \Delta u^\eps = 0 \quad \text{in } (0,T) \times \RR^d, \quad u^\eps(T,\cdot) = u_T,
		\end{equation}
		then, as $\eps \to 0$, $u^\eps$ converges locally uniformly to $u$.
	\end{theorem}
	
	\begin{proof}
		The uniqueness is a consequence of the comparison principle in Theorem \ref{T:NL:cts}, since $u$ is both a sub and supersolution. The local uniform convergence as $\eps \to 0$ of $u^\eps$ to the unique continuous viscosity solution $u$ is then proved with standard stability arguments from the theory of viscosity solutions, arguing with half-relaxed limits and applying the comparison principle, Theorem \ref{T:NL:cts}.
	\end{proof}

\subsection{A one-dimensional example}

For the rest of the paper, we assume $d = m= 1$ and consider the example of \eqref{eq:nonlinear} with $f(t,x,u) = -u$ and $g(t,x,u) = 0$:
\begin{equation}\label{eq:burgers}
	-\del_t u + u \del_x u = 0 \quad \text{in } (0,T) \times \RR, \quad u(T,\cdot) = u_T.
\end{equation}
Observe that \eqref{eq:burgers} can be formally written in conservative form, where it becomes the Burgers equation with flux $\frac{u^2}{2}$. When $d > 1$, it is not necessarily the case that \eqref{eq:nonlinear} can be written as a conservation law, and the product of $f(t,x,u)$ with $\nabla u$ cannot be understood by integrating by parts. 

If the decreasing function $u_T$ is Lipschitz continuous, then the system of characteristics \eqref{chars} can be solved in some maximal time interval $[T-\tau, T]$ depending on the Lipschitz constant for $u_T$. This gives rise to a Lipschitz continuous solution of \eqref{eq:burgers}, which can easily be checked to be a viscosity solution in the sense of Definition \ref{D:nonlinear:viscosity} and is therefore unique. However, for $t < \tau$, the system of characteristics fails to be solvable on $[t,T]$, due to the formation of shocks. This is in contrast to the case where $u_T$ is Lipschitz and \textit{increasing}, in which case \eqref{eq:burgers} has a Lipschitz continuous solution on $(-\oo, T]$. We therefore see that the situation where a continuous solution exists is not typical on an arbitrary time horizon, even if the function $u_T$ is smooth.

We therefore study in this subsection a simple example of a decreasing and discontinuous terminal data, namely
\begin{equation}\label{Heaviside}
	u_T(x) = \ind\{x \le 0 \}.
\end{equation}
Viewed as a scalar conservation law with flux $\frac{u^2}{2}$, \eqref{eq:burgers} is a Riemann problem whose solvability is resolved with the theory of entropy solutions. Indeed, the unique entropy solution is given by
\begin{equation}\label{entropy}
	u(t,x) = \ind\left\{ x \le \frac{T-t}{2} \right\},
\end{equation}
describing a shock wave moving with constant speed $\frac{1}{2}$ (in reverse time), the constant $\frac{1}{2}$ being uniquely determined from the Rankine-Hugoniot condition.

\subsubsection{Nonuniqueness of discontinuous solutions}

When viewed as a nonlinear transport equation, there is a strong failure of uniqueness for this problem.

\begin{proposition}
	Let $c: [0,T] \to \RR$ satisfy $c(	T) = 0$ and $-c' \in [0,1]$ a.e. Then
	\begin{equation}\label{csolution}
		u_c(t,x) := \ind\left\{ x \le c(t) \right\}
	\end{equation}
	is a solution, in the sense of Definition \ref{D:NL:transport}, of \eqref{eq:burgers} with terminal data \eqref{Heaviside}.
\end{proposition}

\begin{proof}
	Define $u_c^\eps = u_c * \rho^\eps$. Then
	\[
		u_c^\eps(t,x) = \int_{-\oo}^{c(t)} \rho^\eps(x-y)dy,
	\]
	and so
	\[
		\del_t u_c^\eps(t,x) = \rho^\eps(x-c(t)) c'(t)
		\quad \text{and} \quad
		\del_x u_c^\eps(t,x) = - \rho^\eps(x - c(t)).
	\]
	Thus,
	\[
		-\del_t u^\eps_c(t,x) + u_c(t,x) \del_x u^\eps_c(t,x) = \rho^\eps(x - c(t))\left[ -c'(t) - \ind\{ x \le c(t) \} \right].
	\]
	Recall that $\supp \rho^\eps \in [-2\eps, 0]$, and so the above expression is nonzero only if $x \le c(t)$. In that case, because $-c' \le 1$,
	\[
		-c'(t) - \ind\{x \le c(t) \} =  -c'(t) - 1 \le 0,
	\]
	and we conclude that $u^\eps_c$ is subsolution. The proof of the supersolution property is similar, and uses the fact that $-c' \ge 0$.
\end{proof}

\begin{remark}
	Note that $u_c$ is continuous a.e. in $[0,T] \times \RR^d$, and, therefore, by Lemma \ref{L:equivdef}, is also a viscosity solution in the sense of Definition \ref{D:nonlinear:viscosity} 
\end{remark}

The system of characteristics \eqref{chars} becomes, for $(t,x) \in [0,T] \times \RR^d$,
\begin{equation}\label{burgerschars}
	\begin{dcases}
		- \del_s U_{s,t}(x) = 0, & U_{T,t}(x) = \ind\left\{ X_{T,t}(x) \le 0 \right\},\\
		\del_s X_{s,t}(x) = -U_{s,t}(x), & X_{t,t}(x) = x.
	\end{dcases}
\end{equation}
Recall that the system, and, in particular, the equation for $X$, is viewed as a differential inclusion, where, for all $s \in [t,T]$,
\[
	U_{s,t}(x) = U_{T,t}(x) =
	\begin{dcases}
		\{1\} & \text{if } X_{T,t}(x) < 0,\\
		\{0 \} & \text{if } X_{T,t}(x) > 0, \text{ and}\\
		[0,1] & \text{if } X_{T,t}(x) = 0.
	\end{dcases}
\]
By Proposition \ref{P:chars}, each solution $u_c$ corresponds to a solution $(X^c,U^c)$ of \eqref{burgerschars}, which we can compute explicitly: namely, for $x \in \RR$ and $0 \le s \le t \le T$,
\[
	U^c_{s,t}(x) = u_c(t,x) = \ind\left\{ x \le c(t) \right\}
\]
and
\begin{equation}\label{Xc}
	X^c_{s,t}(x) = 
	\begin{dcases}
		x - s + t, & x \le c(t), \\
		x - c(t) \frac{ s-t}{T-t}, & x = c(t), \\
		x, & x \ge c(t).
	\end{dcases}
\end{equation}
We note that, for any solution $(U,X)$, we must have $X_{s,t}(x) = x - s + t$ for $x < 0$ and $X_{s,t}(x) = x$ for $x \ge T-t$. However, for $x \in [0,T-t]$, there is ambiguity in the speed at which the $X_{s,t}(x)$ travels: it can move with speed $-1$, $0$, or anything in between, where in the latter case the characteristic is constrained to end at $X(T,t,x) = 0$. The precise value $x \in [0,1]$ for which $X_{T,t}(x) = 0$ thus encodes the choice of the shock-wave speed $c(t)$ in the definition of $u_c$. 

\subsubsection{Stochastic selection}

An important feature in the theory of entropy solutions of scalar conservation laws is the stability under regularizations, and in particular under vanishing viscosity limits. In the above context, the entropy solution \eqref{entropy} of \eqref{eq:burgers} arises as the strong limit in $C([0,T], L^1_\loc(\RR^d))$, as $\eps \to 0$, of the unique smooth solution $u^\eps$ of
\begin{equation}\label{eq:epsburgers}
	-\del_t u^\eps + u^\eps\del_x u^\eps = \eps \del_x^2 u^\eps, \quad u^\eps(T,\cdot) = u^\eps_T,
\end{equation}
where
\begin{equation}\label{uTregs}
	u^\eps_T: \RR^d \to \RR^d \text{ is smooth and } \lim_{\eps \to 0} u^\eps_T = \ind_{(-\oo,0)} \text{ in }  L^1_\loc.
\end{equation}

By contrast, we show here that any solution $u_c$ can arise as a limit from suitably regularized equations. 
\begin{theorem}\label{T:burgers:selection}
	Fix any $c \in W^{2,1}([0,T], \RR)$ satisfying $c(T) = 0$ and $-c' \in (0,1)$. Then there exists $\theta_c \in L^1([0,T])$ such that, if $u^\eps_T$ is as in \eqref{uTregs} and $u^\eps$ is the unique classical solution of 
	\begin{equation}\label{eq:NL:epsburgers}
		-\del_t u^\eps + u^\eps\del_x u^\eps = \eps  \left( \del_x^2 u^\eps + \theta_c(t) |\del_x u^\eps|^2 \right), \quad u^\eps(T,\cdot) = u^\eps_T,
	\end{equation}
	then, for all $1 \le p < \oo$, as $\eps \to 0$, $u^\eps \to u_c$ strongly in $C([0,T], L^p_\loc)$.
\end{theorem}

\begin{proof}
	For $\theta: [0,T] \to \RR$ to be determined, define, for $t \in [0,T]$,
	\[
		f(t,v) = 
		\begin{dcases}
			\frac{\log(\theta(t) v + 1)}{\theta(t)} & \text{if } \theta(t) \ne 0 \text{ and } \theta(t) v > -1,\\
			v & \text{if }\theta(t) = 0
		\end{dcases}
	\]
	and note that
	\[
		f(t,\cdot)^{-1}(u) = \frac{ e^{\theta(t) u} - 1}{\theta(t)}.
	\]
	For $\eps > 0$, let $v^\eps$ be the classical solution of
	\begin{equation}\label{eq:veps}
		-\del_t v^\eps + f(t,v^\eps)v^\eps_x = \frac{\del_t f(t, v^\eps)}{\del_v f(t, v^\eps)} + \eps \del_x^2 v^\eps \quad \text{in } [0,T] \times \RR \quad \text{and} \quad v^\eps(T,\cdot) = f(T,\cdot)^{-1}( u^\eps_T)).
	\end{equation}
	By the maximum principle, it is easily checked that the values of $v^\eps$ fall within the domain of $f$ and its derivatives, and the solution of \eqref{eq:NL:epsburgers} is given exactly by $u^\eps(t,x) = f(t, v^\eps(t,x))$.

	 Standard stability results \cite{Kr_70} yield that, as $\eps \to 0$, $v^\eps$ converges strongly in $C([0,T], L^p_\loc)$ for all $p \in [1,\oo)$ to the unique entropy solution $v$ of
	\begin{equation}\label{eq:v}
		-\del_t v + f(t,v)\del_x v = \frac{\del_t f(t,v)}{\del_v f(t,v)} \quad \text{in } [0,T] \times \RR \quad \text{and} \quad v(T,\cdot) = f^{-1}(T, \cdot)(1)\ind_{x<0} + f^{-1}(T,\cdot)(0) \ind_{x > 0}.
	\end{equation}
	We then set
	\[
		F(t,v) = \frac{(\theta(t)v + 1)\log(\theta(t) v + 1) - \theta(t) v}{\theta(t)^2},
	\]
	which satisfies $\del_v F(t,v) = f(t,v)$, so that the equation \eqref{eq:v} is equivalently written as
	\begin{equation}\label{eq:v}
		-\del_t v + \del_x F(t,v) = \frac{\del_t f(t, v)}{\del_v f(t, v)}  \quad \text{in } [0,T] \times \RR \quad \text{and} \quad v(T,\cdot) = f^{-1}(T, \cdot)(1)\ind_{x<0} + f^{-1}(T,\cdot)(0) \ind_{x > 0}.
	\end{equation}
	The unique entropy solution $v$ of the conservative equation \eqref{eq:v} is then given by
	\[
		v(t,x) = 
		\begin{dcases}
			v^-(t), & x < c(t), \\
			v^+(t), & x > c(t),
		\end{dcases}
	\]
	where, for $0 \le t \le T$, $v_{t}^+$ and $v_{t}^-$ are defined by
	\[
		-\del_t v^\pm_{t} = \frac{\del_t f(t, v_{t}^\pm)}{\del_v f(t, v_t^\pm)} \quad t \in [0,T], \quad v^-_T = f^{-1}(T,\cdot)(0), \quad v^+_T = f^{-1}(T, \cdot)(1),
	\]
	and $c(t)$ is determined by the Rankine-Hugoniot condition
	\[
		c(T) = 0 \quad \text{and} \quad -c'(t) = \frac{ F(t, v^-(t)) - F(t, v^+(t))}{v^+(t) - v^-(t)}.
	\]
	Observe that $t \mapsto f(t, v^\pm(t))$ is constant, and therefore 
	\[
		-c'(t) = \frac{ F(t, f(t,\cdot)^{-1}(1)) - F(t, f(t, \cdot)^{-1}(0)) }{f(t, \cdot)^{-1}(1) - f(t, \cdot^{-1}(0)}
		C(\theta(t)),
	\]
	where, for $\theta \in \RR$,
	\[ 
		C(\theta) = \frac{ \theta e^\theta - e^\theta - 1}{\theta(e^\theta - 1)}.
	\]
	We observe that $C' > 0$, $\lim_{\theta \to -\oo} C(\theta) = 0$, and $\lim_{\theta \to +\oo} C(\theta) = 1$. Thus, letting $c: [0,T] \to \RR$ with $-c' \in (0,1)$ be as in the statement of the theorem, we finally choose $\theta_c(t) := C^{-1}(-c'(t))$, so that $u^\eps$ converges strongly in $C([0,T], L^p_\loc)$, $p \in [1,\oo)$, to $f^{-1}(t,\cdot)(v(t,x)) = u_c(t,x)$.
\end{proof}

\begin{remark}
	The restrictions that $c'' \in L^1$ and $c'$ lie strictly within $(-1,0)$ are put in place to ensure that $\theta_c \in W^{1,1}$ and $- \oo < \theta_c < \oo$, so that the equation for $v^\eps$ is well-posed. Achieving speeds $c$ that are only Lipschitz, and where $c'$ is allowed to be either $0$ or $1$, is possible by letting $\theta_c = \theta_c^\eps$ depend suitably on $\eps$.
\end{remark}

The selection principle in Theorem \ref{T:burgers:selection} can be recast in terms of the nonunique generalized characteristics problem \eqref{burgerschars}. Namely, we fix a filtered probability space $(\Omega, \mbb F = (\mbb F_t)_{t \in [0,T]}, \mbb P)$ with a complete, right-continuous filtration $\mbb F$ and a Wiener process $W: [0,T] \times \Omega \to \RR$ progressively measurable with respect to $\mbb F$, and, for $\eps > 0$ and $(t,x) \in [0,T] \times \RR$, we consider the forward-backward SDE (FBSDE) on the interval $s \in [t,T]$:
\begin{equation}\label{FBSDE}
	\begin{dcases}
	- d_s U^\eps_{s,t}(x) = Z^\eps_{s,t}(x)dW_s - \frac{1}{2} \theta_c(s) Z^\eps_{s,t}(x)^2 ds, & U^\eps_{T,t}(x) = u^\eps_T(X^\eps_{T,t}(x)), \\
	d_s X^\eps_{s,t}(x) = - U^\eps_{s,t}(x)ds + \sqrt{2\eps} dW_s, & X^\eps_{t,t}(x) = x.
	\end{dcases}
\end{equation}

\begin{theorem}\label{T:chars:selection}
	Let $c$ and $\theta_c$ be as in Theorem \ref{T:burgers:selection}. For every $\eps > 0$ and $(t,x) \in [0,T] \times \RR$, there exists a unique strong solution $(X^\eps_{\cdot,t}(x), U^\eps_{\cdot,t}(x), Z^\eps_{\cdot,t}(x))$ to the FBSDE \eqref{FBSDE}. Moreover, with probability one, as $\eps \to 0$, $X^\eps_{\cdot,t} \to X^c_{\cdot,t}$ in $C([t,T], L^p_\loc(\RR))$ for all $p \in [1,\oo)$, where $X^c$ is as in \eqref{Xc}.
\end{theorem}

\begin{proof}
	The existence of a unique solution is a consequence of the nondegeneracy of the noise; see \cite{PT_99}. Moreover, if $u^\eps$ solves \eqref{eq:NL:epsburgers}, then
	\[
		U^\eps_{s,t}(x) = u^\eps(s, X^\eps_{s,t}(x)) \quad \text{and} \quad Z^\eps_{s,t}(x) = \sqrt{2\eps} (\del_x u^\eps)(s, X^\eps_{s,t}(x)).
	\]
	It follows that $X^\eps_{\cdot,t}$ is the solution of the SDE \eqref{SDE:eps} with $b(t,x) = -u^\eps(t,x)$, and so, by Theorems \ref{T:ODE:epsW} and \ref{T:burgers:selection}, as $\eps \to 0$ with probability one, $X^\eps_{\cdot,t}$ converges in $C([t,T], L^p_\loc(\RR))$ to the regular Lagrangian flow corresponding to the drift $-u_c$. This is exactly $X^c$, and we conclude.
\end{proof}

\bibliography{transportbib}{}
\bibliographystyle{acm}

\end{document}